\documentclass[12pt]{amsart}
\usepackage{enumerate,amssymb,version,aliascnt,geometry,stmaryrd,mathrsfs,chngcntr,hhline,bbold,mathtools,textgreek}
\usepackage[all]{xy}
\usepackage[mathscr]{euscript}
\usepackage[pdfusetitle]{hyperref}

\hypersetup{
pdfkeywords={involutions, fixed points, Chern numbers, algebraic cobordism, equivariant cohomology theories},
pdftitle={On the algebraic cobordism ring of involutions},
hidelinks,
unicode
}

\usepackage[T1]{fontenc}
\usepackage{textcomp}

\swapnumbers

\usepackage{etoolbox}
\makeatletter
\pretocmd{\section}{\addtocontents{toc}{\protect\addvspace{6\p@}\bfseries}}{}{}
\pretocmd{\subsection}{\addtocontents{toc}{\protect\normalfont}}{}{}
\makeatother

\DeclareMathAlphabet{\mathdutchcal}{U}{dutchcal}{m}{n}
\SetMathAlphabet{\mathdutchcal}{bold}{U}{dutchcal}{b}{n}
\DeclareMathAlphabet{\mathdutchbcal}{U}{dutchcal}{b}{n}

\DeclareMathOperator{\id}{id}
\DeclareMathOperator{\Spec}{Spec}

\DeclareMathOperator{\im}{im}
\DeclareMathOperator{\coker}{coker}
\DeclareMathOperator{\carac}{char}
\DeclareMathOperator{\CH}{CH}

\DeclareMathOperator{\Hh}{H}
\DeclareMathOperator{\Rep}{Rep}
\DeclareMathOperator{\eff}{eff}

\newcommand{\Kt}{\underline{\K}}
\newcommand{\K}{K}
\newcommand{\Speck}{k}
\newcommand{\length}{\ell}

\DeclareMathOperator{\Sm}{\mathbf{Sm}}
\DeclareMathOperator{\V}{\mathcal{V}}
\DeclareMathOperator{\fdim}{bdim}
\DeclareMathOperator{\fdeg}{bdeg}

\newcommand{\Oc}{\mathcal{O}}

\newcommand{\Pc}{\mathcal{P}}
\newcommand{\Lc}{\mathcal{L}}
\newcommand{\Lch}{\mathscr{V}}
\newcommand{\Rc}{\mathcal{A}}
\newcommand{\Fc}{\mathcal{F}}
\newcommand{\Bc}{\mathcal{B}}

\newcommand{\Mcc}{\mathcal{M}}
\newcommand{\Zz}{\mathbb{Z}}
\newcommand{\Pp}{\mathbb{P}}
\newcommand{\Laz}{\mathbb{L}}
\newcommand{\Nn}{\mathbb{N}}

\newcommand{\Gm}{\mathbb{G}_m}
\newcommand{\Fd}{\mathbb{F}_2}
\newcommand{\X}{\mathcal{X}}
\newcommand{\Tan}{T}

\newcommand{\bb}{b_1,\ldots}
\newcommand{\bba}{\ba_1,\ldots}
\newcommand{\ba}{a}
\newcommand{\bx}{X_1,\ldots}
\newcommand{\lc}{\llbracket}
\newcommand{\rc}{\rrbracket}
\newcommand{\mud}{{\mu_2}}
\newcommand{\mun}{\mu_n}
\newcommand{\indic}{\omega}

\newcommand{\J}{\Gamma}
\newcommand{\jj}{g}
\newcommand{\g}{\gamma}
\newcommand{\rr}{\rho}
\newcommand{\ninv}{\psi}

\newcommand{\fund}[1]{{#1}_f}
\newcommand{\Inv}{\mathscr{O}}
\newcommand{\inv}{\Inv(\mud)}
\newcommand{\invh}{\widetilde{\Inv}(\mud)}
\newcommand{\fgl}[1]{\langle {#1} \rangle}
\newcommand{\Vi}{I}

\newcounter{enumi_resume}

\newtheorem*{theorem*}{Theorem}
\newtheorem{theorem}{Theorem}[section]
\newtheorem*{proposition*}{Proposition}
\newtheorem{proposition}[theorem]{Proposition}

\newtheorem{lemma}[theorem]{Lemma}
\newtheorem{corollary}[theorem]{Corollary}

\theoremstyle{definition}
\newtheorem{remark}[theorem]{Remark}
\newtheorem{example}[theorem]{Example}
\newtheorem{definition}[theorem]{Definition}

\newtheoremstyle{par}%                % Name
  {}%                                     % Space above
  {}%                                     % Space below
  {}%                                     % Body font
  {}%                                     % Indent amount
  {}%                 	                  % Theorem head font
  {.}%                                    % Punctuation after theorem head
  { }%                                    % Space after theorem head, ' ', or \newline
  {}
\theoremstyle{par}
\newtheorem{para}[theorem]{}

\counterwithin{equation}{theorem}
\numberwithin{equation}{theorem}

\newcommand{\rref}[1]{(\ref{#1})}
\newcommand{\dref}[2]{(\ref{#1}.\ref{#2})}

\begin{document}
\begin{abstract}
We consider the cobordism ring of involutions of a field of characteristic not two, whose elements are formal differences of classes of smooth projective varieties equipped with an involution, and relations arise from equivariant $K$-theory characteristic numbers. We investigate in detail the structure of this ring. Concrete applications are provided concerning involutions of varieties, relating the geometry of the ambient variety to that of the fixed locus, in terms of Chern numbers. In particular, we prove an algebraic analog of Boardman's five halves theorem in topology, of which we provide several generalisations and variations.
\end{abstract}

\author{Olivier Haution}

\title{On the algebraic cobordism ring of involutions}

\email{olivier.haution@gmail.com}
\address{Mathematisches Institut, Ludwig-Maximilians-Universit\"at M\"unchen, Theresienstr.\ 39, D-80333 M\"unchen, Germany}
\thanks{This work was supported by the DFG grant HA 7702/5-1 and Heisenberg fellowship HA 7702/4-1.}
\date{\today}

\maketitle

\numberwithin{theorem}{section}
\setcounter{tocdepth}{2}
\tableofcontents

\section*{Introduction}

A smooth closed manifold is nonbounding (in the unoriented sense) precisely when at least one of its Stiefel-Whitney numbers (with values modulo two) does not vanish. Conner and Floyd \cite[(27.1)]{CF-book-1st} observed in 1964 that the fixed loci of involutions of nonbounding manifolds cannot have arbitrary low dimension (compared to the dimension of the ambient manifold), and wondered how to explicitly compute that lower bound. In 1966, Boardman did so in its remarkable ``five halves theorem'':
\begin{theorem}[{\cite{Boardman-BAMS}}]
Let $X$ be a smooth closed $n$-manifold equipped with a smooth involution. If $X$ is nonbounding, then at least one component of the fixed locus has dimension greater than or equal to $2n/5$.
\end{theorem}

Boardman's proof \cite{Boardman-revisited} does not provide a clear geometric reason for the existence of this lower bound, but instead relies on a fine understanding of the multiplicative structure of the unoriented cobordism ring of involutions. Roughly speaking, Boardman constructs an explicit infinite family of manifolds with involutions, which generates that ring as a polynomial algebra over $\Fd$, after a certain stabilisation procedure. This reduces the proof to the verification of the validity of the theorem for the elements of this family, which is immediate.\\

The starting point of the current paper consists in transcribing Conner--Floyd's and Boardman's ideas to the algebro-geometric world. An important difference is that we insist on working ``integrally'', while unoriented cobordism is intrinsically $2$-torsion: more precisely formal multiplication by $2$ is nontrivial (in a universal way) in this paper, while it is trivial in the unoriented setting. This yields new results in a range of fixed loci dimensions not captured by unoriented cobordism. In fact involutions on smooth projective varieties which are nonbounding (in the sense that they possess a nonzero Chern number) may very well have low-dimensional fixed loci, so the picture might seem at first quite different from that in algebraic topology. A moment's reflection however reveals that these differences should dissipate once we consider unitary (instead of unoriented) cobordism in topology. We are not aware of analogs of the results of the current paper in that topological setting, although we believe that a completely parallel theory could be developed.

Voevodsky first introduced algebraic cobordism in \cite{Voe-A1} using homotopical techniques. Later Levine--Morel provided a more geometric construction \cite{LM-Al-07}, which is essentially limited to base fields of characteristic zero. In this paper, we will use a third approach pioneered by Merkurjev \cite{Mer-Ori}, which is more elementary, and works in arbitrary characteristic. The basic idea is to \emph{define} the cobordism class of a smooth projective variety in the Lazard ring by its collection of Chern numbers, computed using Chow's theory of cycle classes modulo rational equivalence. 

By an involution, we will mean for short a smooth projective variety over a fixed base field of characteristic not two, equipped with an action of the algebraic group $\mud$ (which is canonically isomorphic to $\Zz/2$). It might initially seem natural to define the cobordism ring of involutions  using equivariant analogues of the Chern numbers, with values in the equivariant Chow ring of the point. This yields the wrong theory though, which for instance does not distinguish the different possible involutions of a given finite set. This problem arises essentially because the approximations of the classifying space of $\mud$ are not cellular varieties. A simple solution consists in using $K$-theory instead of Chow's theory. The equivariant cobordism ring thus defined does contain the Burnside ring of the group $\Zz/2$, and in fact coincides in characteristic zero with the ring obtained using Levine--Morel's algebraic cobordism theory instead of $K$-theory. These two facts provide a conceptual justification to our definition of the equivariant cobordism ring (while a more concrete justification is provided by the applications obtained in \S\ref{sect:applications}). These points are discussed in detail in \S\ref{section:cyclic}, where more generally cyclic group actions are considered (results of Bix and tom Dieck \cite{Bix-tomDieck} in topology suggest that this $K$-theoretic approach should fail for all noncyclic groups). 

All elements are ``geometric'' in nonequivariant cobordism, in the sense that the cobordism ring of the point is generated by classes of smooth projective varieties. This is not true anymore in the equivariant setting. In this paper we study the ``geometric'' subring $\inv$ inside the ``cohomological'' cobordism ring $\Hh_{\mud}(\Speck)$ obtained using Borel's construction. The elements of $\inv$ are the classes of virtual involutions, that is, formal differences of (smooth projective $k$-varieties equipped with) involutions.\\

The structure of $\inv$ is described in \S\ref{sect:structure} by means of a morphism $\nu \colon \inv \to \Mcc$ mapping the class of an involution to the cobordism class of the normal bundle to its fixed locus (here $\Mcc$ is a polynomial ring over the Lazard ring in variables indexed by natural numbers, and should be thought of as the cobordism ring of $BGL$). Of course $\nu$ vanishes on the classes of involutions without fixed points, but more interestingly such involutions generate the kernel of $\nu$. We also describe the image of $\nu$, providing conditions which permit to decide whether a given vector bundle is cobordant to the normal bundle to the fixed locus of some involution. These results are expressed in the fundamental exact sequence of \rref{th:es}, which can be compared to Conner--Floyd's sequence in topology \cite[(28.1)]{CF-book-1st}.

The definition of the topological analog of the morphism $\nu$ is quite straight-forward (using the equivariant collaring theorem \cite[(21.2)]{CF-book-1st}), and is basic in Conner--Floyd's theory. By contrast, the construction of the morphism $\nu$ proved to be a serious problem for us. This is probably inherent to our elementary approach to cobordism, but that particular problem seems unlikely to disappear if one uses Levine--Morel's theory instead, because there are more cobordism relations than just naive cobordisms as in topology (certain degenerations must be allowed \cite{LP-Alg-09}). Our solution consists in developing a substantial part of the theory before even constructing the morphism $\nu$. In particular, the multiplicative structures of $\inv$ and $\Mcc$ are used in an essential way, and so is the analog \rref{th:J} of the injectivity of Boardman's $J$-homomorphism, a result appearing only at the very end of Boardman's paper \cite[Corollary 17]{Boardman-revisited}.

As a first step, we introduce in \S\ref{section:gamma} the characteristic class $\g$ (of a vector bundle). Three equivalent constructions are provided: the first uses the formal group law, the second arises from a certain stabilisation procedure applied to the associated projective bundle, and the third involves $\Gm$-equivariant considerations. The interplay between the different natures of these approaches can be exploited fruitfully: for instance, the multiplicative property of $\g$ is clear from the first definition, but not at all from the other two. From the class $\g$ we derive the class $\jj$, the analog of Boardman $J$-homomorphism \cite{Boardman-BAMS}. Boardman's approach is closer to our second construction, which explains that the multiplicativity of Boardman $J$-homomorphism was a delicate point to establish (according to Boardman \cite{Boardman-BAMS}: ``There ought to be a direct geometric proof that $J'$ is a ring homomorphism'').

It is interesting to note that the exact same characteristic class $\g$ appears in the formulation of Quillen's formula, expressing the class of a projective bundle in the cohomology of its base, and thus bearing no apparent relation with involutions. This formula was first stated by Quillen for complex cobordism \cite[Theorem~1]{Quillen-FGL}, then by Panin--Smirnov \cite{Panin-Smirnov} for oriented cohomology theories, and Vishik gave a complete proof for algebraic cobordism in characteristic zero \cite[\S5.7]{Vi-Sym}. As a byproduct, we derive a new proof of that formula in \rref{prop:Quillen}, substantially simpler than Vishik's, and working for oriented cohomology theories in arbitrary characteristic.

The second step consists in expressing the cobordism class of an involution in terms of the normal bundle to its fixed locus. This is achieved by giving explicit formulas in \S\ref{sect:fixed}, which are proved using deformation to the normal bundle.\\

We obtain in \rref{th:integrality} a purely algebraic description of $\inv$ as an algebra over the Lazard ring. This object is thus in principle completely determined, which is extremely valuable from an abstract point of view: for instance we deduce that this algebra does not depend on the base field \rref{cor:indep_field}, a fact which was not at all apparent from the definition. However it appears to be quite difficult to exploit this description for concrete applications, and it seems desirable to obtain alternative descriptions. In fact, certain crucial aspects of the multiplicative structure of $\inv$ only become apparent by ignoring its structure of algebra over the Lazard ring (the analog of this observation was a key insight of Boardman).\\

In \S\ref{section:explicit}, we construct an explicit family of virtual involutions $\X_n$ indexed by nonzero positive integers $n$. This is achieved by endowing the varieties representing the usual polynomial generators of the cobordism ring with certain natural involutions, which are chosen so as to minimise the dimension of the fixed locus. This is reminiscent of Boardman's family of manifolds with involution \cite[Proof of Lemma 16]{Boardman-revisited}, but more care is required, since we must compute characteristic numbers with integral values, while Boardman only needed to consider their parities. 

The study of the ring $\inv$ can be reduced to that of its quotient $\invh$ modulo the ideal generated by involutions without fixed points; in other words $\invh$ is the image of the morphism $\nu$ above. Denoting by $x_n \in \invh$ the class of $\X_n$, we prove in \rref{prop:bound} that
\begin{equation}
\label{eq:bound}
\tag{$*$}
\Zz[x_1,\ldots] \subset \invh \subset \Zz[t,x_1,\ldots]/(x_1t-2).
\end{equation}
These inclusions are strict and do not bound $\invh$ very tightly. We may however view \eqref{eq:bound} as a stable description of $\invh$, asserting that this ring becomes polynomial in the variables $x_n$ for $n\geq 2$ after to a certain stabilisation procedure.

The second inclusion of \eqref{eq:bound} permits to write the class in $\invh$ of any involution as a sum of elements $t^iA_i(x_1,\ldots)$ for $i=0,\ldots,m$, where each $A_i$ is a polynomial with integral coefficients. Each monomial appearing in one of the $A_i$ provides a lower bound for the dimension of the fixed locus, which depends on the variables that the monomial involves. This yields unstable complements to \eqref{eq:bound}. In particular, the integer $m$ can be controlled in terms of the dimensions of the ambient variety and fixed locus. Worthy of note is the following (see \rref{th:small_dim})
\begin{theorem}
In $\invh$, the class of any involution whose fixed locus dimension does not exceed one third of the dimension of the ambient variety is a polynomial in the variables $x_n$.
\end{theorem}
Thus we obtain a simple description of the ring $\invh$ in a certain range of dimensions, where in particular involutions extend (from the point of view of cobordism) to $\Gm$-actions \rref{cor:Gm}. Note that this range was completely absent from the picture in unoriented cobordism, where all classes in this range are trivial by an equivariant version of the five halves theorem \cite[\S5]{KS}.

Many concrete applications can be derived simply by observing the polynomials $A_i$. We list some of them in \S\ref{sect:applications}, including an algebraic version \rref{cor:Boardman} of Boardman's five-halves theorem, which we generalise in \rref{th:ideal_power}. Roughly speaking, for each $s \in \Nn$, the ratio $5/2$ may be improved to $2+s^{-1}$ if certain Chern numbers are not divisible by too large a power of two. We also obtain results of a slightly different nature, involving Chern numbers of the fixed locus. This type of results are again not visible by unoriented cobordism. Here is a sample (see \rref{cor:n-5d}):
\begin{proposition}
Let $n$ be the dimension of the ambient variety, and $d$ the dimension of the fixed locus. If the fixed locus has a Chern number not divisible by $2^{q+1}$, then $n \leq q + 5d/2$.
\end{proposition}

The variety $\X_1$ is the projective line $\Pp^1$, equipped with an involution fixing two points. Its class $x_1 \in \invh$ generates the subring of classes of virtual involutions fixing only finitely many points (a comparatively easy fact, see \rref{prop:trivial_normal}):
\begin{proposition}
Assume that the number $q$ of fixed points (over an algebraic closure) of an involution is finite. Let $n$ be the dimension of the ambient variety. Then $q=a 2^n$ for some $a \in \Nn$, and the class of the involution in $\invh$ is $ax_1^n$.
\end{proposition}

A similar (but more complicated) description of the classes of virtual involutions with one-dimensional fixed locus is obtained in \S\ref{sect:fixed_curves}. Conversely, we determine whether the cobordism class of a vector bundle over a one-dimensional base arises as the normal bundle to the fixed locus of some virtual involution, in terms of basic invariants (degree and rank of the bundle, genus and number of isolated points of the base). The case of higher dimensional bases is in principle covered by the integrality theorem \rref{th:integrality}, but as mentioned above, this does not readily yield conditions expressible in a simple way.\\

Some of our results (until \S\ref{sect:injectivity}) are actually valid when $\carac k=2$, if instead of involutions one considers $\mud$-actions (that is, idempotent global derivations). Starting from \S\ref{sect:structure}, the theory breaks down, apparently because the supply of smooth varieties with $\mud$-action is simply insufficient. This problem might perhaps be solved by considering lci varieties instead of smooth ones when defining the cobordism ring, but we did not pursue that idea here.

Finally, let us mention that the current paper improves certain results of our previous paper \cite{inv}, and provides more conceptual and systematic proofs. The new elements here are the introduction of the equivariant cobordism ring, and a stronger emphasis on structural considerations. In particular, the analog of Boardman's theorem seemed to us completely out of reach of the methods of \cite{inv}. We will nonetheless initially rely on the setting and notation of that paper concerning oriented cohomology theories, so as to limit the length of the current paper.\\

{\bf Acknowledgements.} I warmly thank the referees for the time they devoted to an extremely thorough reading of this paper, and for pointing certain hard-to-spot inaccuracies. Their suggestions and feedback  certainly led to improvements in the paper.

\section{Fundamental polynomials of vector bundles}
\numberwithin{theorem}{subsection}

\subsection{Oriented cohomology theories}
We fix a base field $k$, and also denote by $\Speck$ its spectrum. We denote by $\Sm_k$ the category of smooth quasi-projective $k$-schemes, and by $\Tan_X$ the tangent bundle to $X \in \Sm_k$. We will use the notation and definitions of \cite[\S2.1--2.3,\S3.1]{inv} concerning oriented cohomology theories (a notion introduced in \cite{LM-Al-07}), and recall some of those below. The letter $\Hh$ will always denote an oriented cohomology theory. Thus for each $X \in \Sm_k$, we have a graded ring $\Hh(X)$, together with pullbacks along arbitrary morphisms in $\Sm_k$, and pushforwards along projective morphisms in $\Sm_k$. Rings and algebras will be assumed to be unital, associative and commutative. Unless otherwise specified, a grading will mean a $\Zz$-grading. 

\begin{para}
\label{p:K}
The Chow ring $\CH$ is an oriented cohomology theory. Another example is the functor $\K$ defined using $K$-theory, as follows. For $X \in \Sm_k$, we set $\K(X) = K_0'(X)[u,u^{-1}] = K_0(X)[u,u^{-1}]$, where $u$ has degree $-1$. Pullbacks are induced by pullbacks in $K_0$. If $X,Y \in \Sm_k$ are connected, and $f\colon Y \to X$ is a projective morphism, then $f_*^{\K}(xu^i) = u^{i+\dim Y-\dim X} f_*^{K_0'}(x)$ for all $x\in K_0'(Y)$. We refer to \cite[Example 1.1.5]{LM-Al-07} (where $u=\beta$) and \cite[\S7]{Qui-72} for more details.
\end{para}

\begin{para}
\label{p:lc_rc}
Each smooth projective $k$-scheme $X$ has a class $\lc X \rc \in \Hh(\Speck)$, defined as $p_*(1)$, where $p\colon X \to \Speck$ is the structural morphism. The subgroup generated by such classes is denoted by $\fund{\Hh}$. It is a graded subring of $\Hh(\Speck)$: if $X$ has pure dimension $n$, then $\lc X \rc$ is homogeneous of degree $-n$. In particular $\fund{\Hh}^i=0$ for $i >0$. If $u\in \Hh(X)$, we write $\lc u \rc =p_*(u) \in \Hh(\Speck)$. 
\end{para}

\begin{example}
When $X$ is a smooth projective $k$-scheme, the Euler characteristic of a coherent $\Oc_X$-module $\Fc$ is defined as
\begin{equation}
\label{eq:def_Euler_char}
\chi(X,\Fc) = \sum_{i=0}^{\dim X} (-1)^i \dim_k H^i(X,\Fc),
\end{equation}
where $H^i$ refers to the sheaf cohomology in the Zariski topology. Then, when $X$ has pure dimension $n$ and $\Hh=\K$, we have $\lc X \rc = \chi(X,\Oc_X)u^n$.
\end{example}

\begin{para}
\label{p:projective_space}
Let $X \in \Sm_k$. For any $n\in \Nn$, it follows from the transversality axiom \cite[(2.1.3.v)]{inv} and \cite[(2.1.9)]{inv} that $c_1(\Oc(1)) \in \Hh(\Pp^n \times X)$ is the class of a linearly embedded $\Pp^{n-1} \times X$ (we write $\Pp^{-1}=\varnothing$). In particular $c_1(\Oc(1))^{n+1} = 0$, and by the projective bundle axiom \cite[(2.1.3.vi)]{inv}, mapping $x$ to $c_1(\Oc(1))$ induces a ring isomorphism
\[
\Hh(X)[x]/x^{n+1} \simeq \Hh(\Pp^n \times X).
\]
\end{para}

\begin{para}
\label{def:loc}
We say that $\Hh$ satisfies the \emph{localisation axiom} if for any closed immersion $i\colon Y \to X$ in $\Sm_k$, with open complement $u\colon U \to X$, the following sequence is exact:
\[
\Hh(Y) \xrightarrow{i_*} \Hh(X) \xrightarrow{u^*} \Hh(U) \to 0.
\]
\end{para}

\begin{para}
The theories $\CH$ and $\K$ satisfy the localisation axiom \cite[Proposition~1.8]{Ful-In-98}, \cite[\S7, Proposition~3.2]{Qui-72}.
\end{para}

\begin{lemma}
\label{lemm:zero-locus}
Let $X \in \Sm_k$ and $E \to X$ a rank $r$ vector bundle.
\begin{enumerate}[(i)]
\item \label{lemm:zero-locus:1} If $s \colon X \to E$ is the zero-section, then $c_r(E) = s^* \circ s_*(1) \in \Hh(X)$.

\item \label{lemm:zero-locus:2} If $Z \in \Sm_k$ is the zero-locus of a regular section of $E$ (see \cite[B.3.4]{Ful-In-98}), then $[Z] = c_r(E) \in \Hh(X)$.
\end{enumerate}
\end{lemma}
\begin{proof}
The first statement is true by definition when $r=1$ (see \cite[(2.1.9)]{inv}), and follows from the splitting principle \cite[(2.1.16)]{inv} and the transversality axiom \cite[(2.1.3.v)]{inv} in general. Let now $t \colon X \to E$ be a regular section whose zero-locus is $Z$. By homotopy invariance we have $t^*=s^*$, and the second statement follows from the first and the transversality axiom \cite[(2.1.3.vii)]{inv}.
\end{proof}

\subsection{Conner--Floyd Chern classes}
\begin{para}
A \emph{partition} $\alpha$ is a sequence of integers $(\alpha_1,\ldots,\alpha_m)$ with $m\in \Nn$ such that $\alpha_1 \geq \alpha_2 \geq \cdots \geq \alpha_m>0$ (the case $m=0$ corresponds to the partition denoted $\varnothing$). The \emph{length} of the partition $\alpha$ is the integer $\length(\alpha) = m \in \Nn$, and the \emph{weight} of $\alpha$ is the integer $|\alpha| = \alpha_1 +\cdots+ \alpha_m \in \Nn$. When $\alpha=(\alpha_1,\ldots,\alpha_m)$ and $\alpha'=(\alpha'_1,\ldots,\alpha'_{m'})$ are partitions, their union $\alpha \cup \alpha'$ is the partition obtained by reordering the elements $\alpha_1,\ldots,\alpha_m,\alpha'_1,\ldots,\alpha'_{m'}$.
\end{para}

\begin{para}
When $R$ is a ring, we denote by $R[\bb]$ the polynomial ring over $R$ in the variables $b_i$ for $i \in \Nn\smallsetminus\{0\}$. We will also write $b_0 =1$. To a partition $\alpha = (\alpha_1,\ldots,\alpha_m)$ corresponds the monomial $b_\alpha = b_{\alpha_1} \cdots b_{\alpha_m} \in \Zz[\bb]$. Note that $b_{\alpha \cup \alpha'} = b_\alpha b_{\alpha'}$. A similar notation will be used with other sets of variables (such as $a_i$, $X_i$, or $x_i$).
\end{para}

\begin{para}
\label{p:P}
(See \cite[(3.1.2), (3.1.4)]{inv}.) For each vector bundle $E$ over $X \in \Sm_k$, there is a class $P^{\Hh}(E) \in \Hh(X)[\bb]$ compatible with pullbacks, and such that:
\begin{enumerate}[(i)]
\item \label{p:P_add} If $0 \to E' \to E \to E'' \to 0$ is an exact sequence of vector bundles, we have $P^{\Hh}(E) = P^{\Hh}(E') P^{\Hh}(E'')$.

\item \label{p:P_line} If $L \to X$ is a line bundle, then $P^{\Hh}(L) = \displaystyle{\sum_{i\in \Nn}}  c_1(L)^i b_i$ (recall that $b_0=1$).
\end{enumerate}
If $\alpha$ is a partition, the \emph{Conner--Floyd Chern class} $c_\alpha(E)\in \Hh(X)$ is defined as the $b_\alpha$-coefficient of $P^{\Hh}(E)$.

These definitions extend to virtual vector bundles, yielding for each $E \in K_0(X)$ classes $P^{\Hh}(E) \in \Hh(X)[\bb]$ and $c_\alpha(E) \in \Hh(X)$.
\end{para}

\begin{lemma}
\label{lemm:length_rank}
Let $E$ be a vector bundle of rank $r$ over $X\in \Sm_k$. Then $c_\alpha(E) =0\in \Hh(X)$ whenever $\length(\alpha) >r$.
\end{lemma}
\begin{proof}
The conclusion of the lemma means that $P^{\Hh}(E)$ has vanishing $b_\alpha$-coefficient when $\alpha$ has length $>r$. It follows from \dref{p:P}{p:P_add} that whenever $0 \to E_1 \to E_2 \to E_3\to 0$ is an exact sequence of constant rank vector bundles, the lemma holds for $E=E_2$ if it holds for both $E=E_1$ and $E=E_3$. By the splitting principle (see \cite[(2.1.16)]{inv}), we may thus assume that $r=1$, in which case the statement follows from \dref{p:P}{p:P_line}.
\end{proof}

\begin{para}
\label{p:twist}
(See \cite[\S3.1]{inv}.) The oriented cohomology theory $\Hh$ gives rise to an oriented cohomology $\underline{\Hh}$ as follows. For $X \in \Sm_k$, we set $\underline{\Hh}(X)=\Hh(X)[\bb]$. For a morphism $f\colon Y \to X$ in $\Sm_k$ we set $f^*_{\underline{\Hh}} = f^*_{\Hh} \otimes \id$, and if $f$ is projective with virtual tangent bundle $\Tan_f \in K_0(Y)$, for any $a\in \underline{\Hh}(Y)$ we set
\[
f_*^{\underline{\Hh}}(a)= (f_*^{\Hh}\otimes \id)(P^{\Hh}(-\Tan_f)a) \in \underline{\Hh}(X).
\]
In particular, we have (here $\alpha$ runs over the partitions)
\begin{equation}
\label{eq:pushforward_1}
f_*^{\underline{\Hh}}(1) = \sum_\alpha f_*^{\Hh}(c_\alpha(-\Tan_f)) b_\alpha.
\end{equation}
\end{para}

\begin{para}
If $\Hh$ satisfies the localisation axiom \rref{def:loc}, then so does $\underline{\Hh}$. Indeed if $i\colon Y \to X$ is a closed immersion in $\Sm_k$ with normal bundle $N$, then the induced morphism $i_*^{\Hh} \otimes \id \colon \underline{\Hh}(Y) \to \underline{\Hh}(X)$ has the same image as $i_*^{\underline{\Hh}}$, since for any $y\in \underline{\Hh}(Y)$ we have
\[
i_*^{\underline{\Hh}}(y) = (i_*^{\Hh}\otimes \id)(y P^{\Hh}(N))  \quad \text{and} \quad (i_*^{\Hh}\otimes \id)(y) = i_*^{\underline{\Hh}}(y P^{\Hh}(-N)).
\]
\end{para}

\begin{para}
\label{p:Chern_fund}
Let $X$ be a smooth projective $k$-scheme. For any vector bundle $E\to X$ and partition $\alpha$, we have $\lc c_\alpha(E) \rc \in \fund{\Hh}$ by \cite[Corollary 9.10]{Mer-Ori} (see \rref{p:lc_rc}).
\end{para}

\subsection{The Lazard ring and Chern numbers}
\begin{definition}
\label{def:graded_power_series}
When $R$ is graded (i.e.\ $\Zz$-graded) ring, we will denote by $R[[t]]$ the \emph{graded power series algebra}, defined as the limit of the graded $R$-algebras $R[t]/t^m$ for $m\in \Nn$, where $t$ has degree $1$. Thus $R[[t]]$ is additively generated by those power series $\sum_{j \in \Nn} r_j t^{i-j}$ for $i\in \Zz$, where each $r_j \in R$ is homogeneous of degree $j$ (it is a subset of the usual power series ring). Iterating this construction, we will write $R[[t_1,\ldots,t_p]]$ instead of $R[[t_1]]\cdots[[t_p]]$.
\end{definition}

\begin{para}
The formal group law of the theory $\Hh$ will be denoted by $(x,y) \mapsto x+_{\Hh} y$. When $n \in \Zz$, we will denote by $\fgl{n}(x) \in \fund{\Hh}[[x]]$ the formal multiplication by $n$ (see \cite[\S2.3]{inv} where the notation $[n]_{\Hh}(x)$ is used instead). It is homogeneous of degree $1$, so that there exist unique elements $u_i \in \fund{\Hh}^{1-i}$ for $i \in \Zz$ vanishing for $i \leq 0$, and such that $u_1 =n$ and
\begin{equation}
\label{eq:fgl}
\fgl{n}(x) = \sum_{i \in \Zz} u_i x^i = nx + u_2x^2 +\cdots \in \fund{\Hh}[[x]].
\end{equation}
\end{para}

\begin{para}
\label{p:K_fund}
We will denote by $\Laz$ the Lazard ring. It is endowed with a (commutative, one-dimensional) formal group law, which is universal among such laws. There is a natural morphism of graded rings $\Laz \to \fund{\Hh}$ (see \cite[\S2.3]{inv}). When $\Hh=\Kt$ or $\Hh=\underline{\CH}$, this morphism is bijective \cite[Proposition 6.2 (2), Theorem 8.2]{Mer-Ori}, and we will identify $\fund{\Hh}$ with $\Laz$. Note that $\Laz$ is torsion-free (being contained in $\underline{\CH}(\Speck) = \Zz[\bb]$).
\end{para}

\begin{para}
\label{p:Chern_number}
Let $X$ be a smooth projective $k$-scheme. The \emph{Chern number} corresponding to a partition $\alpha$ is the integer $c_\alpha(X) = \deg c^{\CH}_\alpha(-\Tan_X)$.
\end{para}

\begin{para}
\label{p:class_CHt}
When $\Hh = \underline{\CH}$, it follows from \rref{eq:pushforward_1} that, when $X$ is a smooth projective $k$-scheme, we have
\[
\lc X \rc = \sum_{\alpha} c_\alpha(X) b_{\alpha} \in \Laz \subset \Zz[\bb].
\]
\end{para}

\begin{para}
\label{p:c_alpha}
For any partition $\alpha$, taking the $b_\alpha$-coefficient yields a group morphism $c_\alpha \colon \Laz \subset \Zz[\bb] \to \Zz$ such that $c_\alpha(X) = c_\alpha(\lc X \rc)$ for any smooth projective $k$-scheme $X$. Observe that for any $x,y \in \Laz$,
\begin{equation}
\label{eq:c_alpha_prod}
c_{\alpha}(xy) = \sum_{\alpha = \alpha' \cup \alpha''} c_{\alpha'}(x) c_{\alpha''}(y).
\end{equation}
\end{para}

\begin{para}
\label{p:omega}
For $m \in \Nn\smallsetminus\{0\}$, let us define the following integer:
\[
\indic_m =
\begin{cases}
1 &\text{ if $m+1$ is not a prime power,}\\
p &\text{ if $m+1$ is a power of the prime $p$.}
\end{cases}
\]
An elementary computation shows that
\begin{equation}
\label{eq:omega_gcd}
\indic_m = \gcd_{1 \leq i\leq \lfloor (m+1)/2 \rfloor} \binom{m+1}{i}.
\end{equation}
\end{para}

\begin{para}
\label{p:gen_L_exist}
The ring $\Laz$ is polynomial on generators $y_i \in \Laz^{-i}$ for $i \in \Nn\smallsetminus\{0\}$ satisfying $c_{(i)}(y_i) = \indic_i$ (see \cite[II, \S7]{Adams-stable}).
\end{para}

\begin{para}
\label{p:decomposable}
An element of the ring $\Laz$ (resp.\ $\Laz/2$) is called \emph{decomposable} if it belongs to the square of the ideal generated by the homogeneous elements of nonzero degrees, and \emph{indecomposable} otherwise. 
\end{para}

\begin{para}
\label{p:gen_L}
For $i\in \Nn\smallsetminus\{0\}$ the function $c_{(i)} \colon \Laz \to \Zz$ vanishes on decomposable elements by \eqref{eq:c_alpha_prod}, and on $\Laz^{-j}$ for $j \neq i$. Therefore by \rref{p:gen_L_exist} a family $\ell_i \in \Laz^0 + \cdots + \Laz^{-i}$ for $i \in \Nn\smallsetminus\{0\}$ constitutes a set of polynomial generators of the ring $\Laz$ if and only if $c_{(i)}(\ell_i) = \pm \indic_i$ for each $i \in \Nn\smallsetminus\{0\}$. The elements $\ell_i$ reduce to polynomial generators of the $\Fd$-algebra $\Laz/2$ if and only if $c_{(i)}(\ell_i)$ is odd when $i+1$ is not a power of two, and $c_{(i)}(\ell_i) =2 \mod 4$ when $i +1$ is a power of two.
\end{para}

\subsection{Classes of vector bundles}

\label{sect:M}
\begin{definition}
Let $E$ be a vector bundle over a smooth projective $k$-scheme $S$. We may decompose $S$ into a disjoint union $S_0 \sqcup \cdots \sqcup S_n$ in such a way that $E|_{S_r}$ has constant rank $r$, for $r=0,\dots,n$. We set
\begin{equation}
\label{eq:def_Mcc}
\lc E \to S \rc = \sum_{r=0}^nv^r \sum_{\alpha} \ba_\alpha\lc c_\alpha(E|_{S_r}) \rc\in \fund{\Hh}[v][\bba],
\end{equation}
where $\alpha$ runs over the partitions. The set of such classes $\lc E \to S\rc$ forms an abelian monoid $\Mcc^{\eff}$, and we denote by $\Mcc$ the associated group.
\end{definition}

\begin{para}
In $\Mcc$, the class of a line bundle $L$ over a smooth projective $k$-scheme $S$ is
\begin{equation}
\label{eq:class_P}
\lc L \to S \rc = v \sum_{i \in \Nn} \lc c_1(L)^i \rc \ba_i.
\end{equation}
\end{para}

\begin{para}
\label{p:Mcc_ring}
The subgroup $\Mcc \subset \fund{\Hh}[v][\bba]$ is a $\fund{\Hh}$-subalgebra. Indeed first observe that $\lc S \to S \rc = \lc S \rc$. Moreover, for $i=1,2$, let $E_i$ be a vector bundle over a smooth projective $k$-scheme $S_i$, and denote by $p_i \colon S_1 \times S_2 \to S_i$ the projection. Since $P^{\Hh}(p_1^*E_1 + p_2^*E_2) = p_1^*P^{\Hh}(E_1) \cdot p_2^*P^{\Hh}(E_2)$, it follows from the projection formula \cite[(2.1.3.iii)]{inv} that $\lc E_1 \to S_1 \rc \cdot \lc E_2 \to S_2 \rc = \lc p_1^*E_1 \oplus p_2^*E_2 \to S_1 \times S_2 \rc$. 
\end{para}

\begin{para}
\label{p:Mcc_base}
The $\fund{\Hh}$-algebra morphism $\fund{\Hh}[v][\bba]\to \fund{\Hh}$ given by $v \mapsto 1$ and $\ba_i\mapsto 0$ for $i>0$ induces a morphism $\varphi \colon \Mcc \to \fund{\Hh}$ which sends $\lc E \to S \rc$ to $\lc S \rc$. This extends to a morphism of $\fund{\Hh}$-algebras $\varphi \colon \Mcc[v^{-1}] \to \fund{\Hh}$.
\end{para}

\begin{para}
\label{p:p_i}
For $i\in \Nn$, let us write $p_i = \lc \Oc(1) \to \Pp^i\rc \in \Mcc$. It follows from \rref{eq:class_P} that (recall that $\ba_0=1$),
\begin{equation}
\label{eq:pi_i}
p_i =\sum_{j=0}^i  v\ba_j\lc \Pp^{i-j} \rc.
\end{equation}
\end{para}

\begin{proposition}
\label{prop:M_pol}
Each of the sets $\{v\ba_i, i \in \Nn\}$ and $\{p_i, i\in \Nn\}$ generates $\Mcc$ as a polynomial algebra over $\fund{\Hh}$.
\end{proposition}
\begin{proof}
Let $R$ be the $\fund{\Hh}$-subalgebra of $\fund{\Hh}[v][\bba]$ generated by $v\ba_i$ for $i \in \Nn$ (recall that $\ba_0=1$). The $\fund{\Hh}$-algebra $R$ is polynomial in the variables $v\ba_i$. Moreover \eqref{eq:pi_i} implies by induction that the elements $p_i \in \Mcc$ for $i\in \Nn$ constitute another set of polynomial generators of that algebra. In particular $R \subset \Mcc$.

Let now $E$ be a vector bundle of rank $r$ over a smooth projective $k$-scheme $S$. By \rref{lemm:length_rank} the element $\lc E \to S \rc$ is an $\fund{\Hh}$-linear combination of monomials $v^r \ba_\alpha$ with $\length(\alpha) \leq r$. Such monomials are products of $v\ba_i$ for $i \in \Nn$ (including $v\ba_0 =v$). This proves that $\lc E \to S \rc \in R$, and $\Mcc \subset R$.
\end{proof}

\begin{para}
\label{p:grading}
Unless otherwise stated, the grading on the ring $\Mcc$ will be the following. Using the grading on $\fund{\Hh}$, and letting $\ba_i$ have degree $-i$ and $v$ have degree $1$, we obtain a grading on the ring $\fund{\Hh}[v][\bba]$. If $E$ is a rank $r$ vector bundle over a smooth projective $k$-scheme $S$ of pure dimension $d$, then $\lc E \to S\rc$ is homogeneous of degree $-d-r$. It follows that $\Mcc$ is a graded $\fund{\Hh}$-subalgebra of $\fund{\Hh}[v][\bba]$. Since $v$ is homogeneous, this grading extends to $\Mcc[v^{-1}]$.
\end{para}

\begin{para}
\label{p:base-grading}
It will be useful to consider another grading of $\Mcc$, that will be called the \emph{base-grading}. We proceed as in \rref{p:grading}, but let $v$ have degree $0$ instead of $1$. The element $\lc E \to S\rc$ is now homogeneous of degree $-d$.
\end{para}

\begin{para}
\label{p:dim}
Let $R$ be a graded ring. When $x \in R$, we will denote by $\dim(x) \in \Zz\cup \{-\infty\}$ the supremum of those $d\in \Zz$ such that $x$ has a nonzero component of degree $-d$. We have $\dim(x+y) \leq \max\{\dim(x),\dim(y)\}$. If $R$ is an integral domain, then $\dim(xy) = \dim(x) + \dim(y)$. When $R=\Mcc$ or $\Mcc[v^{-1}]$ with the base-grading \rref{p:base-grading}, we will use the notation $\fdim(x)$ instead of $\dim(x)$, reserving the notation $\dim(x)$ for the quantity computed with the usual grading \rref{p:grading}.
\end{para}

\begin{para}
\label{p:dim_bdim_dim}
If $S$ is a smooth projective $k$-scheme and $E \to S$ a vector bundle, we claim that
\[
\dim S \geq \fdim(\lc E \to S\rc) \geq \dim(\lc S \rc).
\]
The second inequality holds because the morphism $\varphi \colon \Mcc \to \fund{\Hh}$ of \rref{p:Mcc_base} is graded (for the base-grading \rref{p:base-grading} of $\Mcc$). To prove the first inequality, we may assume that $S$ has pure dimension $d$. Then $\lc E \to S\rc \in \Mcc$ is homogeneous of degree $-d$ (with respect to the base-grading, see \rref{p:base-grading}), and the first inequality follows.
\end{para}

\begin{definition}
The $\fund{\Hh}$-submodule of $\Mcc$ generated by the classes of line bundles $\lc L \to S \rc$ coincides with $\Mcc \cap v\fund{\Hh}[\bba]$, and will be denoted by $\Pc$. 
\end{definition}

\begin{proposition}
\label{prop:P_basis}
Each of the sets $\{v\ba_i, i \in \Nn\}$ and $\{p_i, i\in \Nn\}$ freely generates the $\fund{\Hh}$-module $\Pc$.
\end{proposition}
\begin{proof}
This follows from \rref{prop:M_pol}.
\end{proof}

\section{Equivariant theory}
\numberwithin{theorem}{subsection}
\label{section:cyclic}

In this section $\Hh$ is an oriented cohomology theory, and $G$ a subgroup of $\Gm$, as algebraic groups over $k$. Note that $G=\Gm$ or $G=\mun$ for some $n \in \Nn\smallsetminus\{0\}$. Observe that any character of an algebraic group $A$ over $k$ gives rise to such a subgroup, by taking the image of the associated morphism $A \to \Gm$ (see e.g.\ \cite[(1.73), (5.39)]{Milne-AG}). We refer e.g.\ to \cite[I]{SGA3-1} for the notions related to group actions on schemes. Actions on $k$-schemes will always be over $k$.

\subsection{Equivariant cohomology}

\begin{para}
We will denote by $\sigma$ the one-dimensional $G$-representation induced by the embedding $G \subset \Gm$, and think of it as a $G$-equivariant line bundle over $\Speck$. For $m\in \Nn$, we write $m\sigma = \sigma^{\oplus m}$, and $m\sigma -0$ for the complement of the zero-section. When $X$ is a $k$-scheme with trivial $G$-action and $E$ is a vector bundle over $X$, we will denote by $E\sigma$ the $G$-equivariant vector bundle $E \otimes \sigma$ over $X$. 
\end{para}

\begin{lemma}
\label{lemm:X_m}
Let $X \in \Sm_k$ with a $G$-action, and $m\in \Nn$. Let $X_m' = X \times (m\sigma - 0)$. Then
\begin{enumerate}[(i)]
\item \label{lemm:X_m:1}
the categorical quotient $X_m = X_m'/G$ exists in $\Sm_k$,
\item \label{lemm:X_m:2}
the morphism $X_m' \to X_m$ is faithfully flat,
\item the morphism $G \times X_m' \to X'_m\times_{X_m} X'_m$ given by $(g,x) \mapsto (g\cdot x,x)$ is an isomorphism.
\end{enumerate}
\end{lemma}
\begin{proof}
First observe that in \eqref{lemm:X_m:1} it suffices to show that $X_m$ exists as a quasi-projective $k$-scheme. Indeed, since $X'_m$ is smooth over $k$, it then follows from \eqref{lemm:X_m:2} that $X_m$ is smooth over $k$ by \cite[(17.7.7)]{ega-4-4}.

Consider now the case when $G=\mun$ with $n \in \Nn\smallsetminus\{0\}$, so that $G$ is finite. Since $X$ is quasi-projective over $k$, every finite subset of $X \times (m\sigma-0)$ is contained in an open subscheme, and the statements are proved in \cite[V, Th\'eor\`eme~4.1 and Remarque~5.1]{SGA3-1} (note that $G$ acts freely on  $m\sigma -0$, hence also on $X_m'$).

Finally consider the case $G=\Gm$. The statements certainly hold in the special case $X=\Speck$, since $X_m=\Pp^{m-1}$ (where $\Pp^{-1}=\varnothing$) and $X_m'$ is the complement of the zero-section in $\Oc_{\Pp^{m-1}}(1)$. As $G$ is connected, the case of an arbitrary $X \in \Sm_k$ is deduced using \cite[Theorem~1.6]{Sumihiro} and \cite[Proposition~7.1]{GIT}.
\end{proof}

\begin{definition}
\label{def:H_G}
For any $m \in \Nn$, we view $m\sigma$ as the $G$-equivariant subbundle (over $\Speck$) of $(m+1)\sigma$ given by the vanishing of the last coordinate. When $X \in \Sm_k$, we write $X_m = (X \times (m\sigma-0))/G \in \Sm_k$ (see \rref{lemm:X_m}). Using the pullbacks along the induced morphisms $X_m \to X_{m+1}$, we consider the limit in the category of graded $\Hh(\Speck)$-algebras
\[
\Hh_G(X) = \lim_m \Hh(X_m) = \lim_m \Hh((X \times (m\sigma-0))/G).
\]
Its homogeneous component $\Hh_G^i(X)$ of degree $i \in \Zz$ is the limit of the groups $\Hh^i(X_m)$.
\end{definition}

\begin{para}
The theory $\Hh_G$ is endowed with pushforwards along projective $G$-equivariant morphisms and pullbacks along $G$-equivariant morphisms.
\end{para}

\begin{para}
\label{p:Chern_colim}
Assume that $G$ acts on $X \in \Sm_k$, and let $E \to X$ be a $G$-equivariant vector bundle. For any $m\in \Nn$, the scheme $E_m = (E \times (m\sigma-0))/G$ is naturally a vector bundle over $X_m = (X \times (m\sigma -0))/G$, and for any $n\leq m$, we have $E_m \times_{X_m} X_n=E_n$. For any $i \in \Nn$, we set
\[
c_i(E) = \lim_m c_i(E_m) \in \Hh_G(X).
\]
\end{para}

\begin{lemma}
\label{lemm:power_series}
Let $X \in \Sm_k$ with a $G$-action, and let $L_1,\dots,L_p$ be $G$-equivariant line bundles over $X$. Then there exists a unique morphism of graded $\Hh_G(X)$-algebras
\[
\lambda \colon \Hh_G(X)[[t_1,\dots,t_p]] \to \Hh_G(X)
\]
mapping $t_i$ to $c_1(L_i)$ for $i=1,\dots,p$ (recall our convention \rref{def:graded_power_series} on power series).
\end{lemma}
\begin{proof}
Let $I$ be the ideal of $\Hh_G(X)$ generated by $c_1(L_1),\dots,c_1(L_p)$. For a fixed $m \in \Nn$, the image of $I^n$ in $\Hh(X_m)$ vanishes for $n$ large enough by \cite[(2.1.13)]{inv}. Thus for any $f \in \Hh_G(X)[[t_1,\dots,t_p]]$, having image $f_m \in \Hh(X_m)[[t_1,\dots,t_p]]$, we may consider the element $y_m = f_m(c_1(L_1),\dots,c_1(L_p)) \in \Hh(X_m)$. The image of $\lambda(f) \in \Hh_G(X)$ must coincide with $y_m$, proving the unicity of $\lambda$. Conversely, the elements $y_m$ for $m\in \Nn$ define an element $\lambda(f) \in \Hh_G(X)$, and the induced map $\lambda$ has the required properties. 
\end{proof}

\begin{para}
\label{p:fgl_equ}
The assignment $E \mapsto E_m$ in \rref{p:Chern_colim} is compatible with tensor products. Thus, if $L,M$ are $G$-equivariant line bundles over $X \in \Sm_k$, we have, taking \rref{lemm:power_series} into account,
\[
c_1(L\otimes M) = c_1(L) +_{\Hh} c_1(M) \in \Hh_G(X).
\]
\end{para}

\begin{para}
\label{p:zero-locus-eq}
Let $E$ be a $G$-equivariant vector bundle over $X \in \Sm_k$. If $Z \in \Sm_k$ is the zero-locus of a $G$-equivariant regular section of $E$, then $[Z] = c_r(E) \in \Hh_G(X)$. This follows from \dref{lemm:zero-locus}{lemm:zero-locus:2} by taking the limit.
\end{para}

\begin{para}
\label{p:trivial_action}
Assume that $G$ acts trivially on $X\in \Sm_k$. Then $\Hh_G(X)$ is endowed with a structure of graded $\Hh(X)$-algebra, via the pullbacks along the morphisms 
\[
X_m = (X \times (m\sigma -0))/G = X \times ((m\sigma -0)/G) \to X.
\]
\end{para}

\begin{para}
\label{p:t}
Set $t=c_1(\sigma) \in \Hh_G(\Speck)$. In view of \rref{lemm:power_series} and \rref{p:trivial_action}, for any $X\in \Sm_k$ with trivial $G$-action, there exists an induced morphism of graded $\Hh(X)$-algebras (see \rref{def:graded_power_series})
\begin{equation}
\label{eq:Ht}
\Hh(X)[[t]] \to \Hh_G(X).
\end{equation}
\end{para}

\begin{para}
\label{p:trivial_group}
Assume that $G=1$ and that $\Hh$ satisfies the localisation axiom \rref{def:loc}. It then follows from the homotopy invariance and transversality axioms \cite[(2.1.3)]{inv} that the morphisms $\Hh(X) \to \Hh(X \times (m\sigma -0))$ are bijective for all $m$, hence $\Hh(X) = \Hh_G(X)$ (under the algebra structure of \rref{p:trivial_action}).
\end{para}

\begin{para}
\label{p:change_of_groups}
Let $G' \subset G$ be a subgroup. The pullbacks along $(X \times (m\sigma-0))/G' \to (X \times (m\sigma-0))/G$ induce a morphism of graded $\Hh(\Speck)$-algebras $\Hh_G(X) \to \Hh_{G'}(X)$. These morphisms are compatible with pullbacks along morphisms $Y \to X$ in $\Sm_k$, as well as with pushforwards along projective morphisms in $\Sm_k$ (this follows from the transversality axiom \cite[(2.1.3.v)]{inv}). Taking $G'=1$ and $X=\Speck$, in view of \rref{p:trivial_group}, if $\Hh$ satisfies the localisation axiom \rref{def:loc} we obtain a morphism of graded $\Hh(\Speck)$-algebras $\varepsilon \colon \Hh_G(\Speck) \to \Hh(\Speck)$. In the notation of \rref{p:t}, we have $\varepsilon(t) = c_1(1)=0$.
\end{para}

\begin{lemma}
\label{lemm:Gm-trivial}
Let $G=\Gm$. If $G$ acts trivially on $X\in \Sm_k$, then the morphism \eqref{eq:Ht} is bijective.
\end{lemma}
\begin{proof}
When $m \geq 1$, we have $(m\sigma -0)/G = \Pp^{m-1}$, and $t \in \Hh_G(\Speck)$ is mapped to $c_1(\Oc(1)) \in \Hh(\Pp^{m-1})$. The lemma follows from the isomorphism of \rref{p:projective_space}.
\end{proof}

\begin{lemma}
\label{lemm:mun-trivial}
Let $G=\mun$ with $n\in \Nn\smallsetminus\{0\}$, and $X \in \Sm_k$ with trivial $G$-action. Assume that $\Hh$ satisfies the localisation axiom \rref{def:loc}, and that $\Hh(X)$ has no $n$-torsion. Then the morphism \eqref{eq:Ht} induces an isomorphism of graded $\Hh(X)$-algebras
\[
\Hh(X)[[t]]/\fgl{n}(t) \simeq \Hh_G(X). 
\]
\end{lemma}
\begin{proof}
Let $m \in \Nn\smallsetminus\{0\}$. Observe that $(m\sigma-0)/G$ is isomorphic to the complement $B_m$ of the zero-section in the line bundle $\Oc_{\Pp^{m-1}}(n)$ over $\Pp^{m-1}$. Consider the line bundle $L=\Oc_{\Pp^{m-1}}(1) \times X$ over $\Pp^{m-1} \times X$. Let $j\colon \Pp^{m-1} \times X \to L^{\otimes n}$ be the zero-section and $p \colon L^{\otimes n} \to \Pp^{m-1} \times X$ the projection. By homotopy invariance \cite[(2.1.3.vii)]{inv}, the morphism $p^* \colon \Hh(\Pp^{m-1}\times X) \to \Hh(L^{\otimes n})$ is bijective, hence by \rref{p:projective_space} we have $\Hh(L^{\otimes n}) = \Hh(X)[x]/x^m$ where $x =p^*c_1(L)$. The localisation axiom identifies $\Hh(B_m \times X)$ with the cokernel of $j_* \colon \Hh(\Pp^{m-1} \times X) \to \Hh(L^{\otimes n})$. By the projection formula \cite[(2.1.3.iii)]{inv}, the image of $j_*$ is the ideal generated by $j_*(1)$. Now, by homotopy invariance
\[
j_*(1) = p^* \circ (p^*)^{-1} \circ j_*(1) = p^* \circ j^* \circ j_*(1) = p^*  c_1(L^{\otimes n}) = \fgl{n}(x).
\]
This yields an exact sequence of graded $\Hh(X)$-modules
\begin{equation}
\label{seq:B_m}
\Hh(X)[x]/x^m \xrightarrow{\fgl{n}(x)}\Hh(X)[x]/x^m \to \Hh(B_m \times X) \to 0
\end{equation}
Since by \eqref{eq:fgl} the element $\fgl{n}(x) \in  \Hh(X)[x]/x^m$ is the product of $x$ with a nonzerodivisor (by assumption $\Hh(X)$ has no $n$-torsion), we deduce an exact sequence
\[
0 \to \Hh(X)[x]/x^{m-1} \xrightarrow{\fgl{n}(x)}\Hh(X)[x]/x^m \to \Hh(B_m \times X) \to 0.
\]
As $t \in \Hh_G(X)$ is mapped in $\Hh(B_m \times X)$ to the image of $x$, the statement follows by taking the limit over $m\geq 1$ (note that $\lim^1_m \Hh(X)[x]/x^{m-1}=0$).
\end{proof}

\subsection{Geometric elements}
In this section, we will use the theories $\K$ and $\Kt$ (see \rref{p:K} and \rref{p:twist}).

\begin{para}
\label{p:class_in_inv}
Let $X$ be a smooth projective $k$-scheme with a $G$-action. Denote by $p\colon X\to \Speck$ its structural morphism. We will write in $\Hh_G(\Speck)$
\[
\lc X \rc_G = p_*(1) \quad \text{ and } \quad \lc u \rc_G = p_*(u) \text{ when $u\in \Hh_G(X)$}.
\]
\end{para}

\begin{definition}
The subset of classes $\lc X \rc \in \Hh_G(\Speck)$, where $X$ runs over the smooth projective $k$-schemes with a $G$-action, forms an abelian monoid  denoted $\Inv_{\Hh}(G)^{\eff}$, or simply $\Inv(G)^{\eff}$. We denote by $\Inv_{\Hh}(G)$, or simply $\Inv(G)$, the associated group. This is a $\fund{\Hh}$-subalgebra of $\Hh_G(\Speck)$. If $X$ has pure dimension $n$, then $\lc X \rc_G \in \Hh^{-n}_G(X)$, so that $\Inv_{\Hh}(G)$ is a graded $\fund{\Hh}$-algebra.
\end{definition}

\begin{para}
\label{p:varepsilon}
Assume that $\Hh$ satisfies the localisation axiom \rref{def:loc}. The morphism $\varepsilon$ of \rref{p:change_of_groups} maps $\lc X \rc_G$ to $\lc X \rc$, and therefore induces a morphism $\varepsilon \colon \Inv_{\Hh}(G) \to \fund{\Hh}$.
\end{para}

\begin{lemma}
\label{lemm:Gm_fund}
Let $G=\Gm$. For any smooth projective $k$-scheme $X$ with a $G$-action, we have $\lc X \rc_G \in \fund{\Hh}[[t]] \subset \Hh(\Speck)[[t]]=\Hh_G(\Speck)$.
\end{lemma}
\begin{proof}
For $i\in \Nn$, let $s_i\in \Hh(\Speck)$ be the $t^i$-coefficient of $\lc X \rc_G \in \Hh(\Speck)[[t]]$. Let $m \in \Nn$. The morphism $f \colon X_{m+1} = (X \times ((m+1)\sigma -0))/G \to \Pp^m$ is projective, and we have $f_*(1) = s_0 + s_1 x +\cdots +s_m x^m$ in $\Hh(\Pp^m)$, where  $x = c_1(\Oc_{\Pp^m}(1))$. Applying the pushforward along $\Pp^m \to \Speck$, we obtain in $\Hh(\Speck)$
\[
\lc X_{m+1} \rc = s_0 \lc \Pp^m\rc + \cdots +  s_{m-1} \lc \Pp^1 \rc+ s_m.
\]
If $m=0$, then $s_0 = \lc X_1\rc \in \fund{\Hh}$. If $m>0$, we see that $s_m \in \fund{\Hh}$ by induction on $m$.
\end{proof}

\begin{lemma}
\label{lemm:inj_tau}
Let $G= \mu_n$ and $\Hh = \Kt$. Then $\Laz[[t]]/\fgl{n}(t) \to \Hh_G(\Speck)$ is injective.
\end{lemma}
\begin{proof}
Since $\Hh(\Speck)=\Zz[u,u^{-1}][\bb]$ is torsion-free, in view of \rref{lemm:mun-trivial}, it will suffice to prove that $\Laz[[t]]/\fgl{n}(t) \to \Hh(\Speck)[[t]]/\fgl{n}(t)$ is injective. By the Hattori--Stong theorem \cite[Proposition 7.16 (1)]{Mer-Ori}, the  morphism of abelian groups $\Laz \to \Hh(\Speck)$ admits a retraction. Thus $\Hh(\Speck)/\Laz$ is torsion-free, being a direct summand of $\Hh(\Speck)$. Let $f \in \Hh(\Speck)[[t]]$ be such that $\fgl{n}(t)f \in \Laz[[t]] \subset \Hh(\Speck)[[t]]$. For $i\in \Nn$, let $f_i \in \Hh(\Speck)$ be the $t^i$-coefficient of $f$. For all $m\in \Nn$, we have $nf_m + u_2f_{m-1} +\cdots +u_{m+1}f_0 \in \Laz \subset \Hh(\Speck)$, where $u_i \in \Laz$ are the elements of \eqref{eq:fgl}. Since  $\Hh(\Speck)/\Laz$ has no $n$-torsion, we deduce by induction that $f_i \in \Laz$ for all $i$.
\end{proof}

The next proposition is reminiscent of a result of tom Dieck in topology \cite{tomDieck-Characteristic-II}.
\begin{proposition}
Let $G= \mu_n$ and $\Hh=\Kt$. Assume that $\carac k=0$, and denote by $\Omega$ the algebraic cobordism theory of \cite{LM-Al-07}. Then $\Inv_{\Omega}(G) \to \Inv_{\Hh}(G)$ is bijective.
\end{proposition}
\begin{proof}
The theory $\Omega$ satisfies the localisation axiom \cite[Theorem~3.2.7]{LM-Al-07}. Also $\Omega(\Speck) = \Laz$ by \cite[Theorem~4.3.7]{LM-Al-07}, which is torsion-free \rref{p:K_fund}. In view of \rref{lemm:mun-trivial} and \rref{lemm:inj_tau}, the morphism
\[
\Omega_G(\Speck) = \Laz[[t]]/\fgl{n}(t) \to \Hh(\Speck)[[t]]/\fgl{n}(t) = \Hh_G(\Speck)
\]
is injective, hence so is its restriction $\Inv_{\Omega}(G) \to \Inv_{\Hh}(G)$. Certainly, the latter is also surjective.
\end{proof}

\begin{para}
\label{p:Rep_K}
Denote by $\Rep(G)$ the representation ring of $G$, and by $I_G$ its augmentation ideal, i.e.\ the kernel of the rank morphism $\Rep(G) \to \Zz$. Mapping a representation $V$ to the class of the vector bundle $(V \times (m\sigma -0))/G$ over $(m\sigma -0)/G$ induces a morphism
\begin{equation}
\label{eq:Rep_K}
\Rep(G)[u,u^{-1}] \to \K_G(\Speck).
\end{equation}
\end{para}

\begin{lemma}
\label{lemm:separated}
Let $G=\mu_n$ with $n$ a prime power. Then  $\bigcap_{m\in \Nn} I_G^m =0$.
\end{lemma}
\begin{proof}
Let $R=\Rep(G)$. We have $R= \Zz[\sigma]/(1-\sigma^n)$, and the ideal $I_G$ is generated by $y=1-\sigma$. If $n$ is a power of the prime $p$, we have in $R$
\[
0 = 1 - \sigma^n = 1 -(1-y)^n = y^n \mod p.
\]
It follows that $I_G^n \subset pR$, hence $\bigcap_{m\in \Nn} I_G^m \subset \bigcap_{s \in \Nn} p^s R$. But $R$ is a free abelian group, hence contains no nonzero element divisible by every power of $p$.
\end{proof}

\begin{lemma}
\label{lemm:Rep_K}
Let $G=\mu_n$ with $n$ a prime power. Then the morphism \eqref{eq:Rep_K} is injective.
\end{lemma}
\begin{proof}
For a $X \in \Sm_k$ with a $G$-action, let us denote by $K'_0(X;G)$ the Grothendieck group of the category of $G$-equivariant coherent $\Oc_X$-modules \cite[\S1.4]{Thomason-group}. For $m\in \Nn$, consider the $G$-equivariant vector bundle $p_m \colon m\sigma \to \Speck$. Let $i_m \colon \Speck \to m\sigma$ be the zero-section and by $v_m \colon m\sigma -0 \to m\sigma$ its open complement. Consider that composite
\[
r_m \colon \Rep(G) = K_0'(\Speck;G) \xrightarrow{p_m^*} K_0'(m\sigma;G) \xrightarrow{v_m^*} K_0'(m\sigma-0;G) =K_0'((m\sigma-0)/G)
\]
(the last equality is obtained by identifying the category of $G$-equivariant coherent modules on $m\sigma-0$ with that of coherent modules on $(m\sigma-0)/G$, see e.g.\ \cite[Proposition~2.4.9]{fpt}). Using the localisation sequence and homotopy invariance \cite[Theorems 4.1 and 2.7]{Thomason-group}, we see that $\ker r_m$ is the image of the composite
\[
\Rep(G) = K_0'(\Speck;G) \xrightarrow{(i_m)_*} K_0'(m\sigma;G) \xrightarrow{(p_m^*)^{-1}} K_0'(\Speck;G) = \Rep(G).
\]
Now the tensor product induces an action of $\Rep(G)$ on $K_0'(Y,G)$ for any $Y \in \Sm_k$ with a $G$-action, and pushforwards as well as pullbacks are $\Rep(G)$-linear. It follows that $\ker r_m$ is the ideal of $\Rep(G)$ generated by $e_m = (i_m)_* \circ (p_m^*)^{-1}(1)$. If $m \geq 1$, consider the diagrams:
\[ \xymatrix{
(m-1)\sigma\ar[rr]^{j} \ar[d]_{p_{m-1}} && m\sigma \ar[d]^q && K_0'((m-1)\sigma;G) \ar[rr]^{j_*} && K_0'(m\sigma;G)\\ 
\Speck \ar[rr]^{i_1} && \sigma && K_0'(\Speck;G) \ar[rr]^{(i_1)_*} \ar[u]^{p_{m-1}^*}&&K_0'(\sigma;G)\ar[u]_{q^*}
}\]
Here the morphism $q$ is given by projecting onto the last factor, and $j$ by the vanishing of the last coordinate. The right-hand square is induced by the left-hand one. To see that the former commutes, observe that the latter commutes and is cartesian, with $q$ flat and $i_1$ finite, and that for any $G$-equivariant coherent $\Speck$-module $V$ the natural $G$-equivariant morphism $q^* \circ (i_1)_*V \to j_* \circ p_{m-1}^*V$ is bijective. It follows that
\[
(p_m^*)^{-1} \circ (i_m)_* = (p_1^*)^{-1} \circ (q^*)^{-1} \circ j_* \circ (i_{m-1})_* = (p_1^*)^{-1} \circ (i_1)_* \circ (p_{m-1}^*)^{-1} \circ (i_{m-1})_*,
\]
so that $e_m = e_1 e_{m-1}$. By induction on $m$, we deduce that $e_m = e_1^m$. Now $e_1$ belongs to the augmentation ideal $I_G$ of $\Rep(G)$ (because its image in $K_0'(\Speck)$ is $c_1(1)=0$), hence $e_m \in I_G^m$. We have proved that $\ker r_m \subset I_G^m$. Taking the limit over $m$ and using \rref{lemm:separated}, we deduce that $\Rep(G) \to \lim_m  K_0'((m\sigma-0)/G)$ is injective, and the statement follows after tensoring with $\Zz[u,u^{-1}]$ over $\Zz$.
\end{proof}

\begin{proposition}
Let $G=\mu_n$, with $n$ a prime power not divisible by $\carac k$. Let $\Hh=\Kt$ or $\Hh=\K$. Then the degree zero component of $\Inv(G)$ may be canonically identified with the Burnside ring of $\Zz/n$.
\end{proposition}
\begin{proof}
When $m \in \Nn$ divides $n$, there is a unique isomorphism class $\Sigma_m$ of sets of cardinality $m$ with a transitive $\Zz/n$-action. Denote by $A$ the Burnside ring of $\Zz/n$. The abelian group $A$ is freely generated by the classes $\Sigma_m$ for $m$ dividing $n$. Denote by $\Inv(G)^0$ the degree zero component of $\Inv(G)$. We define a group morphism $f \colon A \to \Inv(G)^0$ by mapping each $\Sigma_m$ to the class of the finite smooth $k$-scheme $\mun/\mu_{n/m} = \Spec (k[x]/(x^m-1))$, with the $G$-action induced by letting $x$ have degree $n/m$ (recall that $\Zz/n$-grading on a $k$-algebra $R$ is the same thing as a $\mun$-action on the $k$-scheme $\Spec R$). If $k'/k$ is a field extension, then $\Hh(\Speck) \to \Hh(k')$ is bijective, hence so is $\Hh_G(\Speck) \to \Hh_{G_{k'}}(k')$ by \rref{lemm:mun-trivial}, and thus $\Inv(G) \to \Inv(G_{k'})$ is injective. Therefore to prove that $f$ is a bijective ring morphism, we may assume that $k$ is algebraically closed. Then the category of smooth finite $k$-schemes with a $\Zz/n$-action is equivalent to the category of finite sets with a $\Zz/n$-action. Fixing an isomorphism $G \simeq \Zz/n$, each finite set $E$ with a $\Zz/n$-action thus corresponds to a smooth finite $k$-scheme $S_E$ with a $\mun$-action. Then the morphism $f$ is induced by $E \mapsto S_E$, and in particular is a surjective ring morphism. To prove its injectivity, we may assume that $\Hh=\K$ (because $\Kt$ maps to $\K$). Consider the linearisation morphism $\lambda \colon A \to \Rep(G)$, induced by mapping a finite set $E$ with a $\Zz/n$-action to the $G$-representation $V_E$ having $E$ as a $k$-basis using the isomorphism $G\simeq \mun$ (in fact $\lambda$ does not depend on the choice of that isomorphism, because $G$ is commutative). Since the elements $1,\sigma,\ldots,\sigma^{n-1}$ are $\Zz$-linearly independent in $\Rep(G) = \Zz[\sigma]/(1-\sigma^n)$, so are the elements $\lambda(\Sigma_m) = 1+\sigma^{n/m} +\cdots + \sigma^{n(m-1)/m}$. Therefore $\lambda$ is injective. Finally, we claim that the composites $\alpha \colon A \xrightarrow{f} \Inv(G)^0 \subset \Hh_G(\Speck)$ and $\beta \colon A \xrightarrow{\lambda} \Rep(G) \subset \Rep(G)[u,u^{-1}] \xrightarrow{\rref{eq:Rep_K}} \Hh_G(\Speck)$ coincide; in view of \rref{lemm:Rep_K} this will conclude the proof of the proposition. 

The morphisms $\alpha,\beta$ are obtained as limits of morphisms $\alpha_m,\beta_m \colon A \to \Hh(U_m/G)$, where $U_m=m\sigma -0$. For a finite set $E$ with a $\Zz/n$-action, the element $\alpha_m(E)$ is the image of $1$ under the pushforward along $q \colon (S_E \times U_m)/G \to U_m/G$, i.e.\ the class of the vector bundle $q_*\Oc_{(S_E \times U_m)/G}$ over $U_m/G$. On the hand $\beta_m(E)$ is the class of the vector bundle $(V_E \times U_m)/G$ over $U_m/G$. Now the $G$-representation $H^0(S_E,\Oc_{S_E})$ is the $k$-vector space of maps $E \to k$, with the $G\simeq \Zz/n$-action given by $(g \cdot f)(x) = f(g^{-1}x)$, for a map $f\colon E \to k$ and elements $g\in \Zz/n$ and $x \in E$. This representation is isomorphic to $V_E$. For $m\in \Nn$, denoting by  $p\colon S_E \times U_m \to U_m$ the projection, we thus have $p_*\Oc_{S_E \times U_m} \simeq V_E \times U_m$ as $G$-equivariant vector bundles over $U_m$, and therefore $q_*\Oc_{(S_E \times U_m)/G} \simeq (V_E \times U_m)/G$ as vector bundles over $U_m/G$. Therefore $\alpha_m = \beta_m$, and $\alpha = \beta$ as required.
\end{proof}

\begin{proposition}
Let $\Hh=\Kt$ and $G=\mu_n$. Then the class $\lc X \rc_G \in \Inv(G)$ of a smooth projective $k$-scheme $X$ with a $G$-action is determined by the elements (see \rref{eq:def_Euler_char})
\[
\chi(X,\Lambda^{\alpha_1}\Tan_X \otimes \cdots \otimes \Lambda^{\alpha_p}\Tan_X) \in \Rep(G),
\]
where $\alpha = (\alpha_1,\ldots,\alpha_p)$ runs over the partitions.

If $n$ is a prime power, then $\lc X \rc_G \in \Inv(G)$ determines the above elements.
\end{proposition}
\begin{proof}
We may assume that $X$ has pure dimension $r$. Let $m\in \Nn$ and $U_m = m\sigma -0$. Let $q\colon (X\times U_m)/G \to U_m/G$ be the natural morphism, with tangent bundle $\Tan_q$ (of rank $r$). It follows from \rref{eq:pushforward_1} that $q_*^{\Kt}(1) \in \Kt(U_m/G)$ determines and is determined by the elements $q_*^{\K}(c_\alpha(-\Tan_q)) \in \K(U_m/G)$, or equivalently $q_*^{\K}(c_\alpha(\Tan_q))$ in view of \cite[(3.1.10)]{inv}, where $\alpha$ runs over the partitions. Now the polynomial ring $\Zz[Y_1,\ldots,Y_r]$ is additively generated by the polynomials $Q_\alpha$ indexed by the partitions $\alpha$, such that $Q_\alpha=0$ when $\length(\alpha)>r$, and $c_\alpha(\Tan_q) = Q_\alpha(c_1(\Tan_q),\ldots,c_r(\Tan_q))$ for all $\alpha$ (see \cite[(3.1.7),(3.1.8),(3.1.9)]{inv}). Thus it follows from \rref{lemm:P_R} below that $q_*^{\Kt}(1)$ determines and is determined by the elements
\[
q_*^{\K}([\Lambda^{\alpha_1}\Tan_q]\cdots[\Lambda^{\alpha_p}\Tan_q]) \in \K(U_m/G),
\]
where $\alpha = (\alpha_1,\ldots,\alpha_p)$ runs over the partitions. Let $p \colon X \to \Speck$ be the structural morphism. Taking the limit over $m\in \Nn$, we deduce that $\lc X \rc_G = p_*^{\Kt_G}(1) \in \Kt_G(\Speck)$ determines and is determined by the elements
\[
p_*^{\K_G}([\Lambda^{\alpha_1}\Tan_X]\cdots[\Lambda^{\alpha_p}\Tan_X]) \in \K_G(\Speck),
\]
each of which is the image under the morphism \eqref{eq:Rep_K} of 
\[
\chi(X,\Lambda^{\alpha_1}\Tan_X \otimes\cdots\otimes \Lambda^{\alpha_p}\Tan_X)u^r \in \Rep(G)[u,u^{-1}].
\]
The proposition follows, in view of \rref{lemm:Rep_K}.
\end{proof}

\begin{lemma}
\label{lemm:P_R}
Let $n,r,d \in \Nn$. There are polynomials $P,R$ in $\Zz[u,u^{-1}][X_1,\ldots,X_r]$ such that for any rank $r$ vector bundle $E \to X$ with $X\in \Sm_k$ and $\dim X \leq d$, we have in $\K(X)$
\[
c_n(E) = P([\Lambda^1E],\ldots,[\Lambda^rE]) \quad \text{ and } \quad [\Lambda^nE] = R(c_1(E),\ldots,c_r(E)).
\]
\end{lemma}
\begin{proof}
By the splitting principle (see \cite[(2.1.16)]{inv}), after replacing $d$ with $d +1 +\cdots+(r-1)$, we may assume that $E$ has a filtration by subbundles with successive quotient line bundles $L_1,\ldots,L_r$. Then for each $i = 1,\ldots,r$ we have in $\K(X)$
\begin{equation}
\label{eq:c1_L}
c_1(L_i)= \fgl{-1}(c_1(L_i^{\vee})) = \fgl{-1}(u^{-1}(1-[L_i])).
\end{equation}
As $c_1(L_i^{\vee})^{d+1}=0$ (see \cite[(2.1.13)]{inv}) we see that $c_1(L_i) \in \K(X)$ can be expressed as a polynomial in $[L_i]$ having coefficients in $\Zz[u,u^{-1}]$ and depending only on $r,n,d$. Denote by $\sigma_1,\ldots,\sigma_r$ the elementary symmetric functions in $r$ variables, and set $\sigma_j =0$ for $j>r$. Then $c_n(E) = \sigma_n(c_1(L_1),\ldots,c_1(L_r))$, so that $c_n(E)$ is a symmetric polynomial in the $[L_i]$ for $i=1,\dots,r$, hence can be expressed as a polynomial in the $\sigma_j([L_1],\ldots,[L_r]) = [\Lambda^jE]$ for $j=1,\dots,r$, which depends only on $r,n,d$. Conversely, in view of \eqref{eq:c1_L},
\[
[\Lambda^nE] = \sigma_n\Big(1-u\fgl{-1}(c_1(L_1)),\ldots,1-u\fgl{-1}(c_1(L_r))\Big) \in \K(X)
\]
is a symmetric polynomial in the $c_1(L_i)$ for $i=1,\dots,r$, hence a polynomial in the $\sigma_j(c_1(L_1),\dots,c_1(L_r)) = c_j(E)$ for $j=1,\dots,r$, which depends only on $r,n,d$.
\end{proof}

\section{The characteristic class \texorpdfstring{$\g$}{\textgamma}}
\label{section:gamma}
In this section $\Hh$ is an oriented cohomology theory.

\subsection{The class \texorpdfstring{$\g$}{\textgamma}}
We recall that the notation $R[[x]]$, where $R$ is a graded ring, stands for the \emph{graded} power series algebra \rref{def:graded_power_series}, where $x$ has degree one.

\begin{definition}
\label{def:gamma}
By the splitting principle \cite[(2.1.16)]{inv} and the nilpotence of first Chern classes \cite[(2.1.13)]{inv}, there is a unique way to assign to each vector bundle $E \to S$ in $\Sm_k$ an element $\g(E)(x) \in \Hh(S)[[x]]$ so that:
\begin{enumerate}[(i)]
\item If $f\colon T \to S$ is a morphism in $\Sm_k$, then $\g(f^*E)(x) = f^*(\g(E)(x))$.
\item If $L$ is a line bundle, then $\g(L)(x) = x +_{\Hh} c_1(L)$. 

\item \label{cond:es} If $0 \to E' \to E \to E'' \to 0$ is an exact sequence of vector bundles, then $\g(E)(x) = \g(E')(x) \cdot \g(E'')(x)$.
\end{enumerate}
\end{definition}

\begin{para}
There are $w_i(x) \in \fund{\Hh}[[x]]$ for $i\in \Nn$ such that $x +_{\Hh} y = x+ y \sum_{i \in \Nn} y^i w_i(x)$ in $\fund{\Hh}[[x,y]]$. It follows that there are $v_i(x) \in \fund{\Hh}[[x]][x^{-1}]$ for $i\in \Nn$ such that
\begin{equation}
\label{eq:inverse_fgl}
1 = (x +_{\Hh} y) \Big(\sum_{i\in \Nn} v_i(x) y^i\Big) \in \fund{\Hh}[[x]][x^{-1}][[y]].
\end{equation}
Thus for any line bundle $L$ over $S\in \Sm_k$, the element $\g(L)(x)$ admits an inverse in $\Hh(S)[[x]][x^{-1}]$, expressible as a universal power series  in $c_1(L)$, with coefficients in $\fund{\Hh}[[x]][x^{-1}]$. By the splitting principle, it follows that $\g(E)(x) \in \Hh(S)[[x]][x^{-1}]^\times$ for any vector bundle $E$ over $S$. This allows us to define a Laurent series $\g(E)(x) \in \Hh(S)[[x]][x^{-1}]$ for any element $E \in K_0(S)$, compatibly with \rref{def:gamma} when $E$ is a vector bundle.
\end{para}

\begin{para}
When $E\in K_0(S)$, we define elements $\g_i(E) \in \Hh(S)$ by
\[
\g(E)(x) = \sum_{i \in \Zz} \g_i(E) x^i \in \Hh(S)[[x]][x^{-1}].
\]
When $E$ has rank $r \in \Zz$, the element $\g(E)(x)$ is homogeneous of degree $r$ by the splitting principle, so that $\g_i(E) \in \Hh^{r-i}(S)$ for each $i \in \Zz$.
\end{para}

\begin{para}
\label{p:gamma_tens}
Let $S\in \Sm_k$ and $E \to S$ a vector bundle of rank $r$. It follows from the splitting principle that for any line bundle $L\to S$,
\[
\gamma(E)(c_1(L)) = c_r(E\otimes L) \in \Hh(S).
\]
The case $L=1$ yields the formula $\gamma_0(E) = c_r(E)$.
\end{para}

\begin{para}
\label{p:gamma_tens_eq}
Let $G$ be a subgroup of $\Gm$, acting on $X \in \Sm_k$. Let $E\to X$ be a $G$-equivariant vector bundle of rank $r$. Then, using the notation of \rref{p:Chern_colim}, the elements $\gamma(E_m) \in \Hh(X_m)[[x]]$ for $m\in \Nn$ define an element $\gamma(E) \in \Hh_G(X)[[x]]$. Applying \rref{p:gamma_tens} 
over each $X_m$ yields the formula $\gamma_0(E) = c_r(E)$, as well as, for any $G$-equivariant line bundle $L \to X$, in view of \rref{lemm:power_series},
\[
\gamma(E)(c_1(L)) = c_r(E\otimes L) \in \Hh_G(X).
\]
\end{para}

\begin{proposition}
\label{prop:gamma_determined}
There exists a unique morphism of abelian groups
\[
\J\colon \Mcc \to \fund{\Hh}[[x]][x^{-1}] \quad ; \quad \lc E \to S \rc \mapsto \sum_{i \in \Zz} \lc \g_i(-E) \rc x^i.
\]
This is a morphism of graded $\fund{\Hh}$-algebras.
\end{proposition}
\begin{proof}
Let $S\in \Sm_k$ and $E \to S$ a vector bundle of rank $r$. By construction each coefficient $\g_i(E)$ may be expressed as a polynomial in the Chern classes of $E$ with coefficients in $\fund{\Hh}$, depending only on $i,r$. Therefore the same is true for each $\g_i(-E)$ (being the coefficients of the inverse power series). Any such polynomial is an $\fund{\Hh}$-linear combination of the Conner--Floyd Chern classes $c_\alpha(E)$ (see \cite[(3.1.7),(3.1.8),(3.1.9)]{inv}), depending only on $i,r$. It follows that there exist coefficients $\lambda_\alpha^{i,r} \in \fund{\Hh}$ indexed by the partitions $\alpha$ and depending only on the integers $i,r$, such that for every vector bundle $E \to S$ with $S \in \Sm_k$
\[
\gamma_i(-E) = \sum_{r \in \Nn} \sum_\alpha \lambda_\alpha^{i,r} c_\alpha(E|_r) \in \Hh(S) \quad \text{for $r\in \Nn$},
\]
where $S_r$ is the (open and closed) locus of $S$ where $E|_{S_r}$ has constant rank $r$. When $S$ is additionally projective over $k$, we deduce that
\[
\lc \gamma_i(-E) \rc= \sum_{r \in \Nn} \sum_\alpha \lambda_\alpha^{i,r} \lc c_\alpha(E|_{S_r}) \rc \in \Hh(\Speck).
\]
In view of \rref{p:Chern_fund}, it follows that $\lc \gamma_i(-E)\rc \in \fund{\Hh}\subset \Hh(\Speck)$. Moreover, by the very definition \rref{eq:def_Mcc}, the elements $\lc c_\alpha(E|_{S_r}) \rc \in \fund{\Hh}$ are determined by the class $\lc E \to S \rc \in \Mcc^{\eff}$. Thus the formula of the proposition defines a map $\Mcc^{\eff} \to \fund{\Hh}[[x]][x^{-1}]$. This map is additive, hence extends to a unique group morphism $\Gamma \colon \Mcc \to \fund{\Hh}[[x]][x^{-1}]$.

As in \rref{p:Mcc_ring}, it follows from the projection formula and the definition of $\g$ that $\J$ is a morphism of $\fund{\Hh}$-algebras.
\end{proof}

\subsection{The class \texorpdfstring{$\rr$}{\textrho}}

\begin{para}
\label{p:def_rho}
Let $E$ be a vector bundle over $S \in \Sm_k$. For $j \in \Nn$, denote by $q \colon \Pp(E \oplus 1\oplus j) \to S$ the projective bundle, and write $\xi = c_1(\Oc(1)) \in \Hh(\Pp(E \oplus 1 \oplus j))$. For a given $i\in \Zz$, the element $\rr_i(E) = q_*(\xi^{j-i}) \in \Hh(S)$ does not depend on the choice of the integer $j \geq \max\{0,i\}$, and vanishes for $i < -\dim \Pp(E\oplus 1)$. We set
\[
\rr(E) = \sum_{i \in \Zz} \rr_i(E) x^i \in \Hh(S)[[x]][x^{-1}].
\]
\end{para}

\begin{para}
\label{p:rho_0}
We have $\rr(0) = \displaystyle{\sum_{i \in \Nn} \lc \Pp^i \rc x^i}\in \Hh(\Speck)[[x]][x^{-1}]$.
\end{para}

\begin{proposition}
\label{prop:rho_gamma}
Let $E$ be a vector bundle over $S \in \Sm_k$. Then
\[
\rr(E) =\g(-E)\sum_{i \in \Nn} \lc \Pp^i \rc x^i.
\]
\end{proposition}
\begin{proof}
We may assume that $E$ has constant rank $r$. Let us use the notation of \rref{p:def_rho}. For any $j \in \Nn$, the closed subscheme $\Pp_S(1\oplus j) \subset \Pp_S(E \oplus 1 \oplus j)$ is the zero-locus of a regular section of $q^*E(1)$ (see \cite[B.5.6]{Ful-In-98}), so that by \dref{lemm:zero-locus}{lemm:zero-locus:2} and \rref{p:gamma_tens}, we have in $\Hh(\Pp_S(E\oplus 1\oplus j))$
\[
[\Pp_S(1\oplus j)] = c_r(q^*E(1)) = \g(q^*E)(\xi) = \sum_{k \in \Nn} \g_k(q^*E) \xi^k.
\]
Let $i\in \Zz$, and choose $j \geq \max\{0,i\}$. Using the projection formula  \cite[(2.1.3.iii)]{inv}, we have in $\Hh(S)$
\[
\rr_i(0) = q_*(\xi^{j-i}[\Pp_S(1 \oplus j)]) =  \sum_{k \in \Nn} \g_k(E) q_*(\xi^{k+j-i}) = \sum_{k \in \Nn} \g_k(E) \rr_{i-k}(E).
\]
Summing over $i \in \Zz$, we obtain in $\Hh(S)$
\[
\sum_{i \in \Zz} \rr_i(0) x^i = \sum_{i\in \Zz} \sum_{k \in \Nn} \g_k(E) x^k \rr_{i-k}(E) x^{i-k} = \Big( \sum_{k \in \Nn} \g_k(E) x^k \Big) \Big( \sum_{m \in \Zz} \rr_m(E) x^m\Big).
\]
This gives the required formula, in view of \rref{p:rho_0}.
\end{proof}

Taking $x^{-m-1}$-coefficient in \rref{prop:rho_gamma}, we deduce:
\begin{corollary}[Quillen's formula]
\label{prop:Quillen}
Let $S \in \Sm_k$ and $E \to S$ be a vector bundle. Denote by $p \colon \Pp(E) \to S$ the projective bundle. Then for any $m \in \Nn$
\[
p_*\big(c_1(\Oc(1))^m\big) = \sum_{i \in \Nn}  \lc \Pp^i\rc \cdot \g_{-1-m-i}(-E) \in \Hh(S).
\]
\end{corollary}
\begin{proof}
Consider the $x^{-m-1}$-coefficient in \rref{prop:rho_gamma}. Then, in the notation of \rref{p:def_rho}, for $j\geq 0$ we have $\xi^{j+1} = [\Pp(E)]$ in $\Hh(\Pp(E\oplus 1 \oplus j))$, hence $q_*(\xi^{j+1+m}) = p_*(c_1(\Oc(1))^m)$.
\end{proof}

\begin{corollary}
\label{cor:PE_M}
For any vector bundle $E$ over a smooth projective $k$-scheme $S$, the class $\lc \Oc(1) \to \Pp(E) \rc \in \Pc$ is determined by $\lc E \to S \rc \in \Mcc$.
\end{corollary}
\begin{proof}
In view \rref{eq:class_P}, this follows from \rref{prop:Quillen} and \rref{prop:gamma_determined}.
\end{proof}

\subsection{\texorpdfstring{$\Gm$}{Gm}-equivariant interpretation}

\begin{para}
\label{def:gamma_Gm}
Let $G$ be a subgroup of $\Gm$ and $S \in \Sm_k$ with trivial $G$-action. The map \eqref{eq:Ht} allows us to consider the element $\g(E)(t)$ of $\Hh_G(S)$ (resp.\ $\Hh_G(S)[t^{-1}]$) when $E \to S$ is a vector bundle (resp.\ $E \in K_0(S)$). If $E$ is a vector bundle of rank $r$, it follows from \rref{p:gamma_tens_eq} that
\[
\g(E)(t) = c_r(E\sigma) \in \Hh_G(S).
\]
\end{para}

\begin{proposition}
\label{prop:gamma_PE}
Let $E$ be a vector bundle over $S \in \Sm_k$. Let $G$ be a subgroup of $\Gm$. Denote by $p\colon \Pp(E) \to S$ and $q\colon \Pp(E\sigma \oplus 1) \to S$ the projective bundles. Consider the element 
\[
\theta = c_1(\Oc_{\Pp(E)}(1)\sigma^{\vee}) = \g(\Oc_{\Pp(E)}(1))(\fgl{-1}(t)) \in \Hh_{G}(\Pp(E)).
\]
Then $\theta$ is invertible in $\Hh_{G}(\Pp(E))[t^{-1}]$, and
\[
\g(-E)(t)= q_*(1) - p_*(\theta^{-1}) \in \Hh_{G}(S)[t^{-1}].
\]
\end{proposition}
\begin{proof}
The formula for $\theta$ and its invertibility follow from \rref{def:gamma}, since $\fgl{-1}(t)$ is invertible in $\Hh_G(\Speck)[t^{-1}]$ by \eqref{eq:fgl}. To prove the remaining statement, we may assume that $E$ has constant rank $r$. Under the identification $\Pp(E) \simeq \Pp(E\sigma)$, we have $\Oc_{\Pp(E\sigma)}(1) = \Oc_{\Pp(E)}(1)\sigma^{\vee}$, and thus $\theta = c_1(\Oc_{\Pp(E\sigma)}(1))$ in $\Hh_{G}(\Pp(E\sigma))$. Let $i\colon \Pp(E\sigma) \to \Pp(E\sigma \oplus 1)$ be the closed immersion. Let $\xi = c_1(\Oc_{\Pp(E\sigma \oplus 1)}(1))$ and $\zeta = c_1(\Oc_{\Pp(E\sigma \oplus 1)}(-1))$ in $\Hh_G(\Pp(E\sigma \oplus 1))$. Then by \rref{p:gamma_tens_eq}, we have in $\Hh_G(\Pp(E\sigma \oplus 1))$
\[
c_r(q^*E\sigma) = c_r(q^*E\sigma (1) \otimes \Oc(-1)) = \gamma(q^*E\sigma (1))(\zeta) = c_r(q^*E\sigma(1)) + \zeta u,
\]
for some $u\in \Hh_G(\Pp(E\sigma \oplus 1))$. Since $\zeta = \fgl{-1}(\xi)$ is divisible by $\xi$, we deduce that $c_r(q^*E\sigma) = c_r(q^*E\sigma(1)) + \xi v$ for some $v\in \Hh_G(\Pp(E\sigma \oplus 1))$. Since $\theta = i^*\xi$ and $\xi=i_*(1)$, by the projection formula \cite[(2.1.3.iii)]{inv} we have $\xi(1 - i_*(\theta^{-1}))=0$ in $\Hh_G(\Pp(E\sigma \oplus 1))[t^{-1}]$. We obtain in  $\Hh_G(\Pp(E\sigma \oplus 1))[t^{-1}]$
\begin{equation}
\label{eq:theta}
(1 - i_*(\theta^{-1}))c_r(q^*E\sigma) = (1 - i_*(\theta^{-1}))c_r(q^*E\sigma(1)).
\end{equation}
Now the closed immersion $j\colon S = \Pp(1) \to \Pp(E\sigma \oplus 1)$ is the zero-locus of a $G$-equivariant regular section of $q^*E\sigma(1)$ (see \cite[B.5.6]{Ful-In-98}). Thus $j_*(1) =c_r(q^*E\sigma(1))$ in $\Hh_G(\Pp(E\sigma \oplus 1))$ by \rref{p:zero-locus-eq}. Now $\Pp(E\sigma) \cap \Pp(1) = \varnothing$ in $\Pp(E\sigma \oplus 1)$, hence $i_*(\theta^{-1}) \cdot j_*(1) =0$ in $\Hh_G(\Pp(E\sigma \oplus 1))[t^{-1}]$, by the projection formula and the transversality axiom \cite[(2.1.3.v)]{inv}. Thus \eqref{eq:theta} yields $(1 - i_*(\theta^{-1}))c_r(q^*E\sigma) = j_*(1)$ in $\Hh_G(\Pp(E\sigma \oplus 1))[t^{-1}]$. We conclude by applying $q_*$, using the projection formula, and multiplying with $\gamma(-E)(t)$, in view of \rref{def:gamma_Gm}.
\end{proof}

\begin{corollary}
\label{cor:PE1_M}
Let $G$ be a subgroup of $\Gm$. For any vector bundle $E$ over a smooth projective $k$-scheme $S$, the class $\lc \Pp(E\sigma \oplus 1) \rc_G \in \Inv(G)$ depends only on $\lc E \to S \rc \in \Mcc$.
\end{corollary}
\begin{proof}
In view of \rref{p:change_of_groups}, we may assume that $G=\Gm$. Then by \rref{lemm:Gm-trivial} the morphism $\Inv(G) \subset \Hh_G(\Speck) \to \Hh_G(\Speck)[t^{-1}]$ is injective. Using \rref{prop:gamma_PE} and its notation, we have in $\Hh_G(\Speck)[t^{-1}]$
\[
\lc \Pp(E\sigma \oplus 1) \rc_G = \lc \g(-E)(t) \rc_G + \lc \theta^{-1} \rc_G.
\]
By \rref{prop:gamma_determined}, the element $\lc \g(-E)(t) \rc_G$ depends only on $\lc E \to S \rc \in \Mcc$, and the element
$\lc \theta^{-1} \rc_G = \lc \g(-\Oc_{\Pp(E)}(1))(\fgl{-1}(t)) \rc_G$ depends only on $\lc \Oc(1) \to \Pp(E)\rc \in \Pc \subset \Mcc$. But $\lc \Oc(1) \to \Pp(E)\rc$ depends only on $\lc E \to S \rc \in \Mcc$ by \rref{cor:PE_M}.
\end{proof}

\begin{para}
\label{p:pi}
We claim that there exists a unique morphism of abelian groups
\[
\pi \colon \Mcc \to \Inv(G) \quad ; \quad \lc  E \to S \rc \mapsto \lc \Pp(E\sigma \oplus 1) \rc_G.
\]
Indeed it follows from \rref{cor:PE1_M} that this assignment defines a map $\Mcc^{\eff} \to \Inv(G)$. Since this map is additive, it extends uniquely to the morphism $\pi$. 
\end{para}

\begin{remark}
After the definition given in \rref{def:gamma} (or equivalently \rref{def:gamma_Gm} with $G=\Gm$), we thus have two additional ways of computing $\gamma$: either using the class $\rr$ and \rref{prop:rho_gamma}, or taking $G= \Gm$ in \rref{prop:gamma_PE}.
\end{remark}

\section{The normal bundle to the fixed locus}
\numberwithin{theorem}{section}
\label{sect:fixed}

We now concentrate on the case $G=\mud$. Recall that a $\mud$-action on a quasi-projective $k$-scheme may be identified with an involution when $\carac k\neq 2$, and with an idempotent global derivation when $\carac k=2$ (see \cite[(4.5)]{inv}). In this section $\Hh$ is an oriented cohomology theory.

\begin{para}
Let $R$ be a graded ring and $S$ a graded $R$-algebra. Let $s \in S$ be such that $x \mapsto s$ induces a morphism of $R$-algebras $R[[x]] \to S$, denoted as usual by $f \mapsto f(s)$. Let $f,g \in R[[x]]$ be such that $f$ is divisible by $g$, and $g$ is a nonzerodivisor. We will write $\frac{f(s)}{g(s)}$ to mean $q(s)$, where $q \in R[[x]]$ is such that $f=qg$.
\end{para}

We will need the following equivariant version of Vishik's formula \cite[\S5.4]{Vi-Sym}.
\begin{lemma}
\label{lemm:Vishik}
Let $G$ be a subgroup of $\Gm$. Let $f \colon Y \to Z$ be a $G$-equivariant finite morphism in $\Sm_k$ whose fiber over any generic point of $Z$ is the spectrum of a two-dimensional algebra. Then the $G$-equivariant $\Oc_Z$-module $\Lch = \coker(\Oc_Z \to f_*\Oc_Y)$ is invertible, and
\[
f_*(1) = \frac{\fgl{2}(c_1(\Lch^\vee))}{c_1(\Lch^\vee)} \in \Hh_G(Z).
\]
\end{lemma}
\begin{proof}
In view of \eqref{p:Chern_colim}, this follows by applying Vishik's formula \cite[(2.5.1)]{inv} to each morphism $(Y \times (m\sigma -0))/G \to (Z \times (m\sigma -0))/G$ for $m\in \Nn$.
\end{proof}

\begin{para}
Let $X\in \Sm_k$ with a $\mud$-action. The fixed locus $X^{\mud}$ belongs to $\Sm_k$ (see e.g.\ \cite[Lemma~3.5.2]{isol}). We will denote by $N_X$ the normal bundle to the closed immersion $X^{\mud} \to X$, endowed with the trivial $\mud$-action.
\end{para}

\begin{para}
\label{p:blowup}
Let us recall a few facts from \cite[(4.7), (4.8)]{inv}, and fix some notation for later reference. Let $X\in \Sm_k$ with a $\mud$-action. We will write $N=N_X$ (with the trivial $\mud$-action), and denote by $Y$ the blow-up of $X^{\mud}$ in $X$. The scheme $Y$ inherits a $\mud$-action such that $Y^{\mud}$ is the exceptional divisor $\Pp(N \sigma)$. The categorical quotient $f\colon Y \to Y/\mud=Z$ exists in the category of quasi-projective $k$-schemes by \cite[V, Th\'eor\`eme~4.1 and Remarque~5.1]{SGA3-1}. Moreover $Z \in \Sm_k$, the morphism $f$ is finite flat of degree two, and the $\mud$-equivariant $\Oc_Z$-module $\coker(\Oc_Z \to f_*\Oc_Y)$ is isomorphic to $\Lc \sigma$, where $\Lc$ is an invertible $\Oc_Z$-module. There is a natural $\mud$-equivariant isomorphism $f^*\Lc^{\vee} \sigma \to \Oc_Y(\Pp(N\sigma))$.
\end{para}

\begin{lemma}
\label{lemm:X_Z_PN}
Let $X$ be a smooth projective $k$-scheme with a $\mud$-action. Using the notation \rref{p:blowup}, we have
\begin{align}
\label{lemm:X_Z_PN:1}
&\lc X \rc_{\mud} = \lc \Pp(N\sigma \oplus 1)\rc_{\mud} - \Big \lc \frac{\fgl{2}(c_1(\Lc^{\vee}) +_{\Hh} t)}{\fgl{-1}(c_1(\Lc^{\vee}) +_{\Hh} t)} \Big \rc_{\mud}&&\in\Hh_{\mud}(\Speck),  \\ 
\label{lemm:X_Z_PN:2}
&\lc X \rc_{\mud} = \lc \g(-N)(t) \rc_{\mud} && \in \Hh_{\mud}(\Speck)[t^{-1}],\\
\label{lemm:X_Z_PN:3}
&\lc Z \rc_{\mud} = 2^{-1} \lc X \rc_{\mud} + \Big \lc \frac{\fgl{-1}(c_1(\Oc_{\Pp(N\sigma\oplus 1)}(-1)))}{\fgl{2}(c_1(\Oc_{\Pp(N \sigma \oplus 1)}(-1)))} \Big \rc_{\mud} &&\in \Hh_{\mud}(\Speck)[2^{-1}].
\end{align}
\end{lemma}
\begin{proof}
We let $\mud$ act trivially on $\Pp^1$, and consider the blow-up $D$ of $X^{\mud} \times 0$ in $X \times \Pp^1$, together with its natural $\mud$-action. Denote by $i\colon \Pp(N\sigma) \to Y$ and $j\colon \Pp(N\sigma \oplus 1) \to D$ the immersions of the exceptional divisors. The fiber of $D \to \Pp^1$ over $0$ is the sum of the effective Cartier divisors $\Pp(N\sigma \oplus 1)$ and $Y$. It is linearly equivalent to the fiber over $1$, which is a copy of $X$. Observe that $\eta =  c_1(\Oc_D(\Pp(N\sigma \oplus 1))) \in \Hh_{\mud}(D)$ restricts to $\mu = c_1(\Oc_Y(\Pp(N\sigma))) \in \Hh_{\mud}(Y)$, to $\zeta = c_1(\Oc_{\Pp(N\sigma\oplus 1)}(-1)) \in \Hh_{\mud}(\Pp(N\sigma\oplus 1))$, and to zero on $X$. Since $Y = X - \Pp(N\sigma\oplus 1)$ as $\mud$-invariant Cartier divisors on $D$, and $X \cap \Pp(N\sigma\oplus 1) = \varnothing$ in $D$, taking \rref{p:zero-locus-eq} and \rref{p:fgl_equ} into account, we deduce that (see also \cite[Proposition~2.5.2]{LM-Al-07})
\begin{equation}
\label{eq:defo}
[Y]  = [X] + j_*\Big(\frac{\fgl{-1}(\zeta)}{\zeta}\Big) \in \Hh_{\mud}(D).
\end{equation}
Multiplying \eqref{eq:defo} with $\frac{\eta}{\fgl{-1}(\eta)}$ and projecting to $\Hh_{\mud}(\Speck)$, we obtain
\begin{equation}
\label{eq:X_PN}
\lc X \rc_{\mud} = \lc \Pp(N\sigma \oplus 1)\rc_{\mud} - \Big \lc \frac{\mu}{\fgl{-1}(\mu)} \Big \rc_{\mud}.
\end{equation}
Since $\Oc_Y(\Pp(N\sigma)) = f^*\Lc^{\vee} \sigma$, we have $\mu = f^*(c_1(\Lc^{\vee}) +_{\Hh} t)$. Thus \eqref{eq:X_PN} together with \rref{lemm:Vishik} (where $\Lch=\Lc\sigma$) and the projection formula yields \eqref{lemm:X_Z_PN:1}.

Now $c_1(\Lc) +_{\Hh} t =\gamma(\Lc)(t) \in \Hh_{\mud}(Z)[t^{-1}]$ is invertible, hence so is its pullback $\fgl{-1}(\mu) \in \Hh_{\mud}(Y)[t^{-1}]$. Letting $\theta = c_1(\Oc_{\Pp(N\sigma)}(1)) =i^*\fgl{-1}(\mu) \in \Hh_{\mud}(\Pp(N\sigma))$, we have in $\Hh_{\mud}(Y)[t^{-1}]$,
\[
\frac{\mu}{\fgl{-1}(\mu)} = \fgl{-1}(\mu)^{-1} \mu=  \fgl{-1}(\mu)^{-1} i_*(1) = i_*( (i^*\fgl{-1}(\mu))^{-1}) = i_*(\theta^{-1}).
\]
Plugging this equation into \eqref{eq:X_PN} and using \rref{prop:gamma_PE}, we obtain \eqref{lemm:X_Z_PN:2}.

Finally we invert $2$ in \eqref{eq:defo}, multiply with $\frac{\eta}{\fgl{2}(\eta)}$, and project to $\Hh_{\mud}(\Speck)[2^{-1}]$. Using again \rref{lemm:Vishik}, we obtain
\[
\lc Z \rc_{\mud} = \Big \lc \frac{\mu}{\fgl{2}(\mu)} \Big \rc_{\mud} = 2^{-1} \lc X \rc_{\mud} - \Big \lc \frac{\fgl{-1}(\zeta)}{\fgl{2}(\zeta)} \Big \rc_{\mud} \in \Hh_{\mud}(\Speck)[2^{-1}].\qedhere
\]
\end{proof}

\begin{proposition}
\label{prop:in_Hf}
Let $X$ be a smooth projective $k$-scheme with a $\mud$-action. Then $\lc X \rc_{\mud}$ belongs to the image of $\fund{\Hh}[[t]] \subset \Hh(\Speck)[[t]]\to \Hh_{\mud}(\Speck)$ (see \rref{eq:Ht}).
\end{proposition}
\begin{proof}
Let $I$ be this image. We use the notation of \rref{p:blowup}. The $\mud$-action on $\Pp(N\sigma \oplus 1)$ extends to a $\Gm$-action, hence $\lc \Pp(N\sigma \oplus 1) \rc_{\mud}$ is in the image of $\Inv(\Gm) \to \Inv(\mud)$, which is contained in $I$ by \rref{lemm:Gm_fund}. We may find coefficients $v_{i,j} \in \fund{\Hh}$ such that
\[
\Big \lc \frac{\fgl{2}(c_1(\Lc^\vee) +_{\Hh} t)}{\fgl{-1}(c_1(\Lc^\vee) +_{\Hh} t)}\Big \rc_{\mud}  = \sum_{i,j\in \Nn} v_{i,j} \lc c_1(\Lc^\vee)^i \rc t^j \in \Hh_{\mud}(\Speck),
\]
so that this element belongs to $I$ by \rref{p:Chern_fund}. The proposition thus follows from \rref{lemm:X_Z_PN:1}.
\end{proof}
\numberwithin{theorem}{subsection}

\section{Actions on cobordism generators}

\label{section:actions_gen}

\subsection{An explicit family of actions}
\label{section:explicit}

\begin{para}
When $n\in \Nn\smallsetminus\{0\}$, we will denote by $c_{(n)}$ the Conner--Floyd Chern class \rref{p:P} for the partition $\alpha = (n)$ and the theory $\Hh=\CH$. Thus $c_{(n)}(L) = c_1(L)^n$ for any line bundle $L$ over $X \in \Sm_k$, and $c_{(n)}(E-F) = c_{(n)}(E) - c_{(n)}(F)$ for any $E,F \in K_0(X)$. For a smooth projective $k$-scheme $X$, the associated Chern number \rref{p:Chern_number} is
\[
c_{(n)}(X) = \deg c_{(n)}(-\Tan_X) = - \deg c_{(n)}(\Tan_X) \in \Zz.
\]
\end{para}

\begin{lemma}
\label{lemm:add_proj}
If $n\in \Nn\smallsetminus\{0\}$, we have $c_{(n)}(\Pp^n) = -n -1$.
\end{lemma}
\begin{proof}
Observe that $\Tan_{\Pp^n} = (n+1)\Oc(1)-1 \in K_0(\Pp^n)$ (see \cite[B.5.8]{Ful-In-98}).
\end{proof}

\begin{para}
\label{p:Milnor_hyp}
(Milnor hypersurfaces.)
Let $0 \leq m \leq n$ be integers. Denote by $U_m$ the kernel of the canonical epimorphism $(m+1) \to \Oc(1)$ of vector bundles over $\Pp^m$, and set $H_{m,n} = \Pp_{\Pp^m}(U_m \oplus (n-m))$. Then $\dim H_{m,n} = m+n-1$.
\end{para}

\begin{para}
\label{p:comp_Milnor}
We will denote by $\pi \colon H_{m,n} \to \Pp^m$ the projective bundle, and by $\Oc\{1\} = \Oc_{H_{m,n}}(1)$ the corresponding canonical line bundle over $H_{m,n}$. Since $U_m = m+1 -\Oc(1)$ in $K_0(\Pp^m)$, we have in $\CH(\Pp^m)$ for any $i\in\Nn$ (see \cite[\S3.2]{Ful-In-98})
\begin{equation}
\label{eq:comp_Milnor}
\pi_*(c_1(\Oc\{1\})^i) = c_{i+1-n}(\Oc(1)) =
\begin{cases}
1 & \text{ if $i=n-1$,}\\
c_1(\Oc(1)) & \text{ if $i=n$,}\\
0 & \text{ otherwise.}
\end{cases}
\end{equation}
\end{para}

\begin{lemma}
\label{lemm:add_Milnor}
Let $1 \leq m \leq n$. Then
\[
c_{(m+n-1)}(H_{m,n}) = \begin{cases}\displaystyle{\binom{m+n}{m}} &\text{ if $m \geq 2$,} \\ 0 & \text{ if $m=1$ and $n\geq 2$.} \end{cases}
\]
\end{lemma}
\begin{proof}
By \cite[B.5.8]{Ful-In-98}, we have $\Tan_{H_{m,n}} = (n+1 - \pi^*\Oc(1))\Oc\{1\}+ (m+1)\pi^*\Oc(1) -2\in \K_0(H_{m,n})$. For $n\geq 2$, we deduce using \eqref{eq:comp_Milnor} that $\pi_*c_{(m+n-1)}(-\Tan_{H_{m,n}}) \in \CH(\Pp^m)$ equals
\[
-(n+1)c_m(\Oc(1))+ \left(\binom{m+n-1}{n-1} + \binom{m+n-1}{n}\right)c_1(\Oc(1))^m,
\]
and the statement follows by taking the degree.
\end{proof}

\begin{definition}
Let $X$ be a smooth projective $k$-scheme with a $\mud$-action, and $d \in \Nn\smallsetminus\{0\}$. We define the invariant
\[
\vartheta_d(X) = (c_{(d)}(X^{\mud}),\deg c_{(d)}(N_X)) \in \Zz^2.
\]
\end{definition}

\begin{definition}
Let $a,b \in \Nn$. We will denote by $\Pp(a,b)$ the scheme $\Pp((a+1)\sigma \oplus (b+1))$ with its natural $\mud$-action.
\end{definition}

\begin{para}
\label{p:pij}
The scheme underlying $\Pp(a,b)$ is $\Pp^{a+b+1}$, and $\Pp(a,b)^{\mud} = \Pp^a \sqcup \Pp^b$.
\end{para}

\begin{lemma}
\label{lemm:pij}
Let $a>b$. Then $\dim \Pp(a,b)^{\mud} =a$ and $\vartheta_a(\Pp(a,b)) = (-a-1,b+1)$.
\end{lemma}
\begin{proof}
The $a$-dimensional component of $\Pp(a,b)^{\mud}$ is $\Pp^a$, by \rref{p:pij}. Since the normal bundle to the immersion $\Pp^a \to \Pp^{a+b+1}$ is $(b+1) \Oc(1)$, the lemma follows from \rref{lemm:add_proj}.
\end{proof}

\begin{definition}
Let $1\leq i \leq j$ be integers. Denote by $V_i$ the kernel of the canonical epimorphism $(i+1)\sigma^{\vee} \oplus i \to \Oc(1)$ of $\mud$-equivariant vector bundles over $\Pp(i,i-1)$, and set $H(i,j) = \Pp_{\Pp(i,i-1)}(V_i\oplus (j-i)\sigma^{\vee}\oplus (j-i))$.
\end{definition}

\begin{para}
\label{p:underlying_Hij}
The scheme underlying $H(i,j)$ is $H_{2i,2j}$ (see \rref{p:Milnor_hyp}).\end{para}

\begin{para}
\label{p:fixed_Hij}
The $\mud$-equivariant vector bundle $V_i \oplus (j-i)\sigma^\vee \oplus (j-i)$ over $\Pp(i,i-1)$ restricts to  $j \oplus (U_i\oplus(j-i))\sigma^{\vee}$ on $\Pp^i$, and to $(j+1)\sigma^{\vee} \oplus U_{i-1}\oplus (j-i)$ on $\Pp^{i-1}$. Thus
\[
H(i,j)^{\mud} = (\Pp^i \times \Pp^{j-1}) \sqcup H_{i,j} \sqcup (\Pp^{i-1} \times \Pp^j) \sqcup H_{i-1,j-1}.\qedhere
\]
\end{para}

\begin{lemma}
\label{lemm:H}
Let $d \geq 1$ and $i$ such that $1 \leq i \leq (d+1)/2$. Let $H=H(i,d+1-i)$. Then 
\begin{enumerate}[(i)]
\item \label{lemm:H:1} $H$ has pure dimension $2d+1$.

\item \label{lemm:H:2} Every component of $H^{\mud}$ has dimension $d$ or $d-2$.

\item \label{lemm:H:3} 
$c_{(2d+1)}(H) = \displaystyle{\binom{d+1}{i}}
\begin{cases}
\mod 2 & \text{if $d+1$ is not a power of two,}\\
\mod 4 & \text{if $d+1$ is a power of two.}
\end{cases}$

\item \label{lemm:H:4} 
$\vartheta_d(H) = 
\begin{cases}
\displaystyle{\left(\binom{d+1}{i},-\binom{d+1}{i}\right)} & \text{if $i\neq 1$,}\\
(-d-1,d-1) & \text{if $i=1$ and $d \neq 1$,}\\
(-6,2) & \text{if $i=d=1$.}
\end{cases}$
\end{enumerate}
\end{lemma}
\begin{proof}
Statement \eqref{lemm:H:1} follows from \rref{p:underlying_Hij}. By \rref{lemm:add_Milnor}, we have
\[
c_{(2d+1)}(H) = \binom{2(d+1)}{2i} =  \pm \binom{d+1}{i} \mod 4.
\]
Moreover, this integer is even when $d+1$ is a power of two, and \eqref{lemm:H:3} follows.

Statement \eqref{lemm:H:2} follows from \rref{p:fixed_Hij}. Moreover the $d$-dimensional part of $H^{\mud}$ is $(\Pp^i \times \Pp^{d-i})\sqcup( \Pp^{i-1} \times \Pp^{d+1-i}) \sqcup  H_{i,d+1-i}$. When $d=1=i$, since $H_{1,1} \simeq \Pp^1$, it follows that $c_{(1)}(H^{\mud}) = 3c_{(1)}(\Pp^1) =-6$. When $d>1$ and $i=1$, by \rref{lemm:add_Milnor} we have $c_{(d)}(H_{1,d})=0 = c_{(d)}(\Pp^1 \times \Pp^{d-1})$, thus $c_{(d)}(H^{\mud}) = c_{(d)}(\Pp^0 \times \Pp^d) = -d-1$. When $d>1$ and $i>1$, we have $c_{(d)}(\Pp^i \times \Pp^{d-i})=0 = c_{(d)}(\Pp^{i-1} \times \Pp^{d+1-i})$, so that, by \rref{lemm:add_Milnor},
\[
c_{(d)}(H^{\mud}) = c_{(d)}(H_{i,d+1-i}) = \binom{d+1}{i}.
\]

In order to compute the normal bundle to the fixed locus, we will use the description given in \rref{p:underlying_Hij}, together with the following observation. Let $Y \to X$ be a closed immersion in $\Sm_k$, with normal bundle $N$, and $F \subset E$ is an inclusion of vector bundles over $X$. Then the class of the normal bundle to $\Pp_Y(F|_Y) \to \Pp_X(E)$ in $K_0(\Pp_Y(F|_Y))$ is $(E/F)\Oc(1) + p^*N$, where $p \colon \Pp_Y(F|_Y) \to Y$ is the projection.

The class of the normal bundle $N_1$ to $H_{i,d+1-i} \to H$ in $K_0(H_{i,d+1-i})$ is $(d+1-i)\Oc\{1\} + i\pi^*\Oc_{\Pp^i}(1)$, in the notation of \rref{p:comp_Milnor}. Thus, using \rref{eq:comp_Milnor}
\[
\pi_* c_{(d)}(N_1) = (d+1-i) c_i(\Oc_{\Pp^i}(1)) +i c_1(\Oc_{\Pp^i}(1))^d\pi_*(1) \in \CH(\Pp^i).
\]
Since $\pi_*(1)=0$ if $d>1$ and $\pi_*(1)=1$ if $d=1$, we deduce that
\[
\deg c_{(d)}(N_1) =  \begin{cases}
2 & \text{if $d=i=1$,}\\
d & \text{if $d>1$ and $i=1$,}\\
0 & \text{if $d>1$ and $i>1$.}
\end{cases}
\]

The class of the normal bundle $N_2$ to $\Pp^i \times \Pp^{d-i} \to H$ in $K_0(\Pp^i \times \Pp^{d-i})$ is 
\[
(d+2-i) q^*\Oc_{\Pp^{d-i}}(1) - p^*\Oc_{\Pp^i}(1) q^*\Oc_{\Pp^{d-i}}(1) + ip^*\Oc_{\Pp^i}(1),
\]
where $p\colon \Pp^i \times \Pp^{d-i} \to \Pp^i$ and $q \colon  \Pp^i \times \Pp^{d-i} \to \Pp^{d-i}$ are the projections, so that
\[
\deg c_{(d)}(N_2) = 
\begin{cases}
0 & \text{if $d=i=1$,}\\
-\displaystyle{\binom{d}{i}} & \text{if $d>1$.}
\end{cases}
\]

The class of the normal bundle $N_2$ to $\Pp^{i-1} \times \Pp^{d+1-i} \to H$ in $K_0(\Pp^{i-1} \times \Pp^{d+1-i})$ is 
\[
(d+1-i) q^*\Oc_{\Pp^{d+1-i}}(1) - p^*\Oc_{\Pp^{i-1}}(1) q^*\Oc_{\Pp^{d+1-i}}(1) + (i+1)p^*\Oc_{\Pp^{i-1}}(1),
\]
where $p\colon \Pp^{i-1} \times \Pp^{d+1-i} \to \Pp^{i-1}$ and $q \colon  \Pp^{i-1} \times \Pp^{d+1-i} \to \Pp^{d+1-i}$ are the projections, so that
\[
\deg c_{(d)}(N_3) = 
\begin{cases}
d-1 & \text{if $i=1$,}\\
-\displaystyle{\binom{d}{i-1}} & \text{if $i>1$.}
\end{cases}
\]
Statement \eqref{lemm:H:4} follows from by gathering these computations.
\end{proof}

\begin{para}
\label{def:virtual_var}
The set of isomorphism classes of smooth projective $k$-schemes (resp.\ with a $\mud$-action) is an abelian cancellative monoid, for the disjoint union operation. We will denote by $\V$ (resp.\ $\V_{\mud}$) the associated abelian group. The notions of $r$-dimensional component for $r\in \Nn$ and of (pure) dimension, the notation $X^{\mud},N_X,\lc X \rc,\lc X \rc_{\mud}$, and the functions $c_{(n)}$ and $\vartheta_d$ admit obvious extensions to these groups.
\end{para}

\begin{definition}
\label{def:X_n}
We define elements $\X_n \in \V_{\mud}$ for $n\in \Nn\smallsetminus\{0\}$ as follows. We set $\X_1 = \Pp(0,0)$, and $\X_n = \Pp(d,d-1)$ if $n=2d$. Now assume that $n=2d+1 \geq 3$. In view of \eqref{eq:omega_gcd}, we may choose integers $e_i \in \Zz$ for $1 \leq i \leq (d+1)/2$ such that
\begin{equation}
\label{eq:omega_e}
\sum_{i=1}^{\lfloor (d+1)/2 \rfloor} e_i \binom{d+1}{i} =\indic_d,
\end{equation}
where $\indic_d$ is the integer defined in \rref{p:omega}. Then we set 
\[
\X_n = \sum_{i=1}^{\lfloor (d+1)/2 \rfloor} e_i H(i,d+1-i).
\]
\end{definition}

\begin{remark}
The elements $\X_n$ for $n$ even or $n=1$ are thus uniquely determined. This is also the case when $n \in \{3,5\}$, since $e_1 = 1$ is the only possibility, so that $\X_3 = H(1,1)$ and $\X_5 = H(1,2)$. For $n \geq 7$ we fix a choice of elements $e_i$ satisfying \eqref{eq:omega_e}, which will have no effect on the rest of the paper (see \rref{p:effective} though).
\end{remark}

\begin{remark}
\label{rem:X_n_indep_G}
Since $\sigma$ is the restriction of a $\Gm$-representation, the $\mud$-action on each $\Pp(a,b)$ or $H(i,j)$ extends to a $\Gm$-action. Moreover $\Pp(a,b)^{\Gm} =\Pp(a,b)^{\mud}$ and $H(i,j)^{\Gm} = H(i,j)^{\mud}$. Therefore the $\mud$-action on each $\X_n$ extends to a $\Gm$-action such that $(\X_n)^{\Gm} = (\X_n)^{\mud}$.
\end{remark}

\begin{proposition}
\label{prop:Xn}
Let $n \in \Nn\smallsetminus\{0\}$, and $d \in \Nn$ such that $n=2d$ or $n=2d+1$. 
\begin{enumerate}[(i)]
\item \label{prop:Xn:dim} $\X_n$ has pure dimension $n$.

\item \label{prop:Xn:dim_fixed} If $n$ is odd, every component of $(\X_n)^{\mud}$ has dimension $d$ or $d-2$. If $n$ is even, every component of $(\X_n)^{\mud}$ has dimension $d$ or $d-1$.

\item \label{prop:Xn:c_n} $c_{(n)}(\X_n) = \begin{cases} 1 \mod 2 & \text{ if $n+1$ is not a power of two,}\\2 \mod 4& \text{ if $n+1$ is a power of two.}\end{cases}$

\item \label{prop:Xn:theta_n} If $n \geq 2$, there exists an integer $s_n$ such that
\[
\vartheta_d(\X_n) =
\begin{cases}
(-d-1,d) & \text{ if $n$ is even,}\\
(\indic_d-2(d+1)s_n,-\indic_d+2ds_n) & \text{ if $n$ is odd.}
\end{cases}
\]
\end{enumerate}
\end{proposition}
\begin{proof}
The scheme underlying $\X_1$ is $\Pp^1$, so that $c_{(1)}(\X_1)=-2$ by \rref{lemm:add_proj}, and $(\X_1)^{\mud} = \Speck \sqcup \Speck$. The statements follow for $n=1$. In case $n=2d$, the statements follow from \rref{p:pij}, \rref{lemm:add_proj}, \rref{lemm:pij}.

Now assume that $n=2d+1$. Statements \eqref{prop:Xn:dim}, \eqref{prop:Xn:dim_fixed}, \eqref{prop:Xn:c_n} follow from \rref{lemm:H}. Letting $s=2$ if $d=1$, and $s=1$ if $d>1$, we have by \dref{lemm:H}{lemm:H:4}
\[
\vartheta_d(H(1,d)) = (d+1-2s(d+1),-(d+1)+2sd).
\]
Taking \dref{lemm:H}{lemm:H:4} into account, we obtain \eqref{prop:Xn:theta_n} by letting $s_n = se_1$ (see \rref{eq:omega_e}).
\end{proof}

\begin{corollary}
\label{cor:gen_Laz}
Assume that $\Laz/2 \to \fund{\Hh}/2$ is bijective.
\begin{enumerate}[(i)]
\item \label{cor:gen_Laz:0} We have $\dim \lc \X_n \rc =n$ for $n \geq 1$, and $\dim \lc (\X_{2d+1})^{\mud} \rc=d$ for $d \geq 1$ (where $\dim$ is computed in the graded ring $\fund{\Hh}/2$, see \rref{p:dim}).

\item \label{cor:gen_Laz:1}
The $\Fd$-algebra $\fund{\Hh}/2$ is polynomially generated by $\lc \X_n \rc$ for $n \geq 1$.

\item \label{cor:gen_Laz:2}
The $\Fd$-algebra $\fund{\Hh}/2$ is polynomially generated by $\lc (\X_{2d+1})^{\mud} \rc$ for $d \geq 1$.
\end{enumerate}
\end{corollary}
\begin{proof}
In view of \rref{p:gen_L}, this follows from \rref{prop:Xn}.
\end{proof}

\subsection{Two degrees}
In this section we fix a graded ring $R$.

\begin{para}
\label{p:default_grading}
Unless otherwise stated, the grading on the ring $R[\bx]$ will be the one induced by letting $X_i$ have degree $i$ for $i \in \Nn\smallsetminus\{0\}$. We will denote by $\deg P$ the degree of a polynomial $P$ with respect to this grading.
\end{para}

\begin{para}
\label{p:fdeg}
We will also need to consider the grading of the ring $R[\bx]$ induced by letting $X_1$ have degree zero, and for $i \in \Nn\smallsetminus\{0\}$, letting $X_{2i}$ and $X_{2i+1}$ have degree $i$, . We will denote by $\fdeg P$ the degree of a polynomial $P$ with respect to this grading.
\end{para}

\begin{para}
\label{p:fdeg_X_alpha}
When $\alpha = (\alpha_1,\ldots,\alpha_m)$ is a partition, we have
\[
\fdeg X_\alpha = \left\lfloor \frac{\alpha_1}{2} \right\rfloor + \cdots + \left\lfloor \frac{\alpha_m}{2} \right\rfloor.
\]
Thus $|\alpha| - 2\fdeg X_\alpha$ is the number of indices $i \in \{1,\ldots,m\}$ such that $\alpha_i$ is odd. 
\end{para}

A basic relation between the two gradings is the following:
\begin{proposition}
\label{prop:deg_fdeg}
Let $s\in \Nn \smallsetminus\{0\}$. Let $P$ be a polynomial in $R[\bx]$, homogeneous with respect to the grading of \rref{p:default_grading}, and containing as a monomial a nonzero multiple of $X_{2i_1 +1}\cdots X_{2i_p+1}M$, where $i_1,\ldots,i_p \in \Nn$, and $M$ is a product of the variables $X_j$ for $j$ even or $j\geq 2s+1$. Then
\[
\deg P \leq \frac{(s-i_1)+\cdots+(s-i_p)}{s} + \Big(2 + \frac{1}{s}\Big) \fdeg P.
\]
\end{proposition}
\begin{proof}
Let $j$ be an odd integer such that $j\geq 2s+1$. Then $\fdeg X_j \geq s$, and $\deg X_j = 2 \fdeg X_j +1$, so that
\[
\deg X_j \leq ( 2 + s^{-1})\fdeg X_j.
\]
This inequality also holds when $j$ is an arbitrary even integer. Since both $\deg$ and $\fdeg$ transform products into sums, it follows that
\[
\deg M \leq ( 2 + s^{-1})\fdeg M.
\]
Now $\fdeg P \geq i_1+\cdots +i_p + \fdeg M$, hence 
\[
\deg P = (2i_1+1)+\cdots+(2i_p+1) + \deg M \leq p -s^{-1}(i_1+\cdots+i_p) + (2 +s^{-1})\fdeg P.\qedhere
\]
\end{proof}

\begin{corollary}
\label{cor:deg_fdeg_2}
In the situation of \rref{prop:deg_fdeg}, we have
\[
\deg P \leq p + \Big(2 + \frac{1}{s}\Big) \fdeg P.
\]
\end{corollary}

\begin{corollary}
\label{cor:deg_fdeg_3}
Let $P \in R[X_2,\ldots] \subset R[X_1,\ldots]$. Then $\deg P \leq 3\fdeg P$.
\end{corollary}
\begin{proof}
We may assume that $P\neq 0$, and apply \rref{cor:deg_fdeg_2} with $s=1$ and $p=0$ to the homogeneous component of top degree of $P$.
\end{proof}

\begin{para}
\label{p:def_succ}
Let $\alpha,\beta$ be partitions. Write $\beta = (\beta_1,\ldots,\beta_r)$. We will write $\alpha \succ \beta$ when there exist partitions $\alpha^1,\ldots,\alpha^r$ such that $\beta_i = |\alpha^i|$ for all $i=1,\ldots,r$ and $\alpha = \alpha^1 \cup \cdots \cup \alpha^r$. 
\end{para}

\begin{para}
\label{p:succ_length}
Assume that $\alpha \succ \beta$. Then $|\alpha| = |\beta|$ and $\length(\alpha) \geq \length(\beta)$. If in addition $\length(\beta) = \length(\alpha)$, then $\alpha=\beta$.
\end{para}

\begin{para}
\label{p:bdeg_succ}
It follows from \rref{p:fdeg_X_alpha} that $\fdeg X_\alpha \leq \fdeg X_\beta$ when $\alpha \succ \beta$.
\end{para}

\begin{lemma}
\label{lemm:change_of_gen}
Let $S$ be a graded $R$-algebra. Assume that $y_i \in S^{-i}$, resp.\ $y_i' \in S^{-i}$, for $i\in \Nn\smallsetminus\{0\}$ are polynomial generators of $S$ over $R$. Let $P,Q \in R[\bx]$ be such that $P(y_1,\ldots) = Q(y_1',\ldots) \in S$. Then $\fdeg Q = \fdeg P$.
\end{lemma}
\begin{proof}
For every $i\in \Nn\smallsetminus\{0\}$, there exists a unique polynomial $B_i \in R[\bx]$  such that $y_i' = B_i(y_1,\dots)$. Moreover $B_i$ is homogeneous of degree $i$ and contains a nonzero multiple of $X_i$ as a monomial. Then for any partition $\alpha=(\alpha_1,\dots,\alpha_m)$, the polynomial $B_\alpha =B_{\alpha_1} \cdots B_{\alpha_m}\in R[\bx]$ contains a nonzero multiple of $X_\alpha$ as a monomial, and satisfies $y_\alpha' = B_\alpha(y_1,\ldots)$. In addition, if $\beta$ is a partition such that $B_\alpha$ contains a nonzero multiple of $X_\beta$ as a monomial, then $\beta \succ \alpha$. Thus it follows from \rref{p:bdeg_succ} that $\fdeg B_\alpha = \fdeg X_\alpha$ for all partitions $\alpha$. Let us write $Q = \sum_\alpha \mu_\alpha X_\alpha$, where $\alpha$ runs over a set of partitions and $\mu_\alpha \in R$ are nonzero. Then $P = \sum_\alpha \mu_\alpha B_\alpha$, and thus $\fdeg P \leq \max_\alpha \fdeg B_\alpha = \max_\alpha \fdeg X_\alpha = \fdeg Q$. The statement follows by symmetry.
\end{proof}

\subsection{Stable generators of the ring \texorpdfstring{$\Mcc$}{M}}

\begin{para}
\label{def:x_n}
Using the elements $\X_n$ defined in \rref{def:X_n}, we will write $x_n = \lc N_{\X_n} \to (\X_n)^{\mud} \rc \in \Mcc$ for $n \in \Nn \smallsetminus\{0\}$.
\end{para}

\begin{example}
\label{ex:x_123}
\ 
\begin{enumerate}[(i)]
\item \label{ex:x_123:1} Since $\X_1=\Pp(0,0)$ has exactly two fixed points, we have 
\[
x_1 = 2v.
\]

\item  \label{ex:x_123:2} The fixed locus of $\X_2 = \Pp(1,0)$ has exactly two components: $\Pp^1$ with normal bundle $\Oc(1)$, and $\Speck$ with normal bundle of rank two. Thus
\[
x_2 = \lc \Pp^1 \rc v + \ba_1 v +v^2.
\]

\item  \label{ex:x_123:3} The fixed locus of $\X_3 = H(1,1)$ has pure dimension one, hence \dref{lemm:H}{lemm:H:4} implies that
\[
x_3 = 3\lc \Pp^1 \rc v^2 + 2\ba_1 v^2.
\]
\end{enumerate}
\end{example}

\begin{proposition}
\label{prop:Mcc_pol}
Assume that the morphism $\Laz \to \fund{\Hh}$ is bijective (see \rref{p:K_fund}).
\begin{enumerate}[(i)]
\item \label{prop:Mcc_pol:poly}
The $\Zz[v,v^{-1}]$-algebra $\Mcc[v^{-1}]$ is polynomial in the variables $x_n$ for $n \geq 2$.

\item \label{prop:Mcc_pol:deg}
Let $P \in \Zz[\bx]$ and $n\in \Nn$. Then $P$ is homogeneous of degree $n$ (for the grading of \rref{p:default_grading}) if and only if $P(x_1,\ldots) \in \Mcc$ is homogeneous of degree $-n$ (for the usual grading \rref{p:grading}).

\item \label{prop:Mcc_pol:Phi}
For any $P \in \Zz[v,v^{-1}][\bx]$, we have $\fdim(P(x_1,\ldots)) = \fdeg P$ (see \rref{p:dim} and \rref{p:fdeg} for the definitions of $\fdim$ and $\fdeg$).
\end{enumerate}
\end{proposition}
\begin{proof}
We will identify $\fund{\Hh}$ with $\Laz$. Let $\ell_d \in \Laz^{-d}$ for $d \in \Nn\smallsetminus\{0\}$ be a system of polynomial generators of the ring $\Laz$, such that $c_{(d)}(\ell_d) = \indic_d$ (see \rref{p:gen_L_exist} and \rref{p:gen_L}). By \rref{prop:M_pol}, the $\Zz[v,v^{-1}]$-algebra $\Mcc[v^{-1}]$ is polynomial in the variables $\ell_d,\ba_d$ for $d \geq 1$. For any rank $r$ vector bundle $E$ over a smooth projective $k$-scheme $S$ of pure dimension $d$, we have
\begin{equation}
\label{eq:EX}
\lc E \to S \rc - v^r \Big(\frac{c_{(d)}(S)}{\indic_d} \ell_d + c_{(d)}(E) a_d\Big) \in v^r\Zz[\ell_1,\ldots,\ell_{d-1},a_1,\ldots,a_{d-1}].
\end{equation}
For any $d \geq 1$ and $n\in \{2d,2d+1\}$, write $\vartheta_d(\X_n) = (\indic_d\lambda_n,\mu_n)$ (see \rref{prop:Xn}), and
\[
x'_n = v^{n-d}(\lambda_n \ell_d + \mu_n \ba_d) \in \Mcc.
\]
By \dref{prop:Xn}{prop:Xn:theta_n} the determinant of the matrix
\[
\left(\begin{matrix}
v^d\lambda_{2d} & v^{d+1}\lambda_{2d+1}\\
v^d\mu_{2d} & v^{d+1}\mu_{2d+1}
\end{matrix}\right)
\]
is equal to $v^{2d+1}$, hence is invertible in $\Zz[v,v^{-1}]$. It follows that the $\Zz[v,v^{-1}]$-algebra $\Mcc[v^{-1}]$ is polynomial in the variables $x'_n$ for $n \geq 2$, and that for any $d \geq 1$ the sets $\{\ell_1,\ldots,\ell_d,a_1,\ldots,a_d\}$ and $\{x'_2,\ldots,x'_{2d+1}\}$ generate the same $\Zz[v,v^{-1}]$-subalgebra of $\Mcc$. For any $d \geq 1$ and $n\in \{2d,2d+1\}$, by \eqref{eq:EX} and \dref{prop:Xn}{prop:Xn:dim_fixed} we have
\[
x_n - x'_n  \in \Zz[v,v^{-1}][x'_2,\ldots, x'_{2d-1}] \subset \Mcc[v^{-1}].
\]
We deduce \eqref{prop:Mcc_pol:poly}. Since each $x_n$ for $n\in\Nn\smallsetminus\{0\}$ is homogeneous of degree $-n$ with respect to the usual grading, we obtain \eqref{prop:Mcc_pol:deg}. Set $x_1'=x_1=2v$ (see \dref{ex:x_123}{ex:x_123:1}). Then the elements $x_n$ (resp.\ $x_n'$) for $n \in \Nn \smallsetminus\{0\}$ are algebraically independent (over $\Zz$). Let $Q \in \Zz[X_1,\ldots]$ be such that $P(x_1,\ldots) = Q(x_1',\ldots)$. Since each $x_n'$ is homogeneous of degree $-\fdeg X_n$ with respect to the base-grading, we have $\fdim(Q(x_1',\ldots)) = \fdeg Q$, and \rref{prop:Mcc_pol:Phi} follows from \rref{lemm:change_of_gen}.
\end{proof}

\begin{corollary}
\label{cor:gen_x_v}
Assume that the morphism $\Laz \to \fund{\Hh}$ is bijective. Then the ring $\Mcc[v^{-1}]$ is generated by $v,v^{-1}$ and $x_n$ for $n \geq 2$.
\end{corollary}

\section{Injectivity of the morphism \texorpdfstring{$\jj$}{g}}
\label{sect:injectivity}
From now on, we will assume that $\Hh=\Kt$ (see \rref{p:K} and \rref{p:twist}). In view of \rref{p:K_fund}, we will identify $\fund{\Hh}$ with the Lazard ring $\Laz$.

\subsection{The ring \texorpdfstring{$\Rc$}{A}}
We recall that the notation $R[[t]]$, where $R$ is a graded ring, stands for the \emph{graded} power series algebra \rref{def:graded_power_series}.

\begin{para}
We let $h\in \Laz[[t]]$ be defined by $ht = \fgl{2}(t)$, and consider the graded $\Laz$-algebra
\[
\Rc = \Laz[[t]]/h.
\]
\end{para}

\begin{para}
\label{p:h_first}
It follows from \cite[Remark 2.5.6]{LM-Al-07} that $2=\lc \Pp^1 \rc t \mod t^2 \Rc$. 
\end{para}

\begin{para}
\label{p:epsilon}
Mapping $t$ to zero yields a morphism of $\Laz$-algebras $\epsilon \colon \Rc \to \Laz/2$, whose kernel is $t\Rc$. 
\end{para}

\begin{lemma}
\label{lemm:R}
Let $R$ be a graded $\Laz$-algebra, and set $A=R[[t]]/h$. Assume that $R$ is $2$-torsion free.
\begin{enumerate}[(i)]
\item \label{lemm:R:nzd} The element $t$ is a nonzerodivisor in $A$.
\item \label{lemm:R:sep} We have $\bigcap_{n\in \Nn} t^n A =0$.
\end{enumerate}
\end{lemma}
\begin{proof}
First observe that the image of $h$ is a nonzerodivisor in $R[[t]]/t^n$ for every $n\geq 1$. In case $n=1$, this follows from the fact that $h \mod t =2$ in $R[[t]]/t = R$, as $R$ is $2$-torsion free. Assume that $n>1$, and let $a,b \in R[[t]]$ be such that $ha=t^nb$. By induction on $n$, we may write $a = t^{n-1}a'$ with $a' \in R[[t]]$, and thus $ha'=tb$. By the case $n=1$, we deduce that $a' \in tR[[t]]$, and finally $a\in t^nR[[t]]$.

\eqref{lemm:R:nzd}: Assume that $x,y \in R[[t]]$ are such that $tx = hy$. By the observation above (with $n=1$), we deduce that $y =tz$ for some $z \in R[[t]]$. Then $x =hz$ has vanishing image in $A$, as required.

\eqref{lemm:R:sep}: Let $f\in R[[t]]$ map to $\bigcap_{n \in \Nn} t^n A \subset A$. Then for each $n \geq 1$ we may find $g_n\in R[[t]]$ such that $f=g_n h \mod t^n$. Thus $g_{n+1} h =g_n h \mod t^n$ for all $n \geq 1$, hence $g_{n+1} =g_n \mod t^n$ by the observation above. Therefore there exists $g \in R[[t]]$ such that $g = g_n \mod t^n$ for all $n \geq 1$, and $f-gh \in \bigcap_{n \in \Nn} t^nR[[t]] =0$.
\end{proof}

\begin{lemma}
\label{lemm:2-torsion-free}
The group $\Rc$ is $2$-torsion free.
\end{lemma}
\begin{proof}
Let $x \in \Rc$ be nonzero and such that $2x=0$. By \rref{lemm:R}, after dividing $x$ by an appropriate power of $t$, we may assume that $x \not \in t\Rc$. In view of \rref{p:h_first} we have $\lc \Pp^1 \rc tx \in t^2\Rc$, and thus $\lc \Pp^1 \rc x \in t\Rc$ by \dref{lemm:R}{lemm:R:nzd}. Applying the morphism of \rref{p:epsilon}, we deduce that $0 = \lc \Pp^1 \rc \epsilon(x) \in \Laz/2$. Since $\lc \Pp^1 \rc$ is a nonzerodivisor in $\Laz/2$ (for instance by \dref{cor:gen_Laz}{cor:gen_Laz:1}), it follows that $\epsilon(x)=0$, hence $x \in t \Rc$, a contradiction.
\end{proof}

\begin{lemma}
\label{lemm:Laz_subring_Rc}
The morphism $\Laz \to \Rc$ is injective.
\end{lemma}
\begin{proof}
Consider an element $\lambda \in \Laz$ whose image in $\Rc$ vanishes. Using the morphism $\epsilon$ of \rref{p:epsilon}, we see that $\lambda = 2 \lambda'$ with $\lambda' \in \Laz$. It follows from \rref{lemm:2-torsion-free} that the image of $\lambda'$ in $\Rc$ vanishes. Thus the kernel of the morphism $\Laz \to \Rc$ is a $2$-divisible ideal of $\Laz$, hence vanishes because $\Laz$ is a polynomial ring (see \rref{p:gen_L_exist}).
\end{proof}

\subsection{The morphism \texorpdfstring{$\jj$}{g}}

\begin{para}
\label{p:R_K}
By \rref{lemm:inj_tau}, we may view
\[
\Rc[t^{-1}] = \Laz[[t]][t^{-1}]/h = \Laz[[t]][t^{-1}]/\fgl{2}(t)
\]
as an $\Laz$-subalgebra of $\Hh_{\mud}(\Speck)[t^{-1}]$.
\end{para}

\begin{para}
\label{p:invh}
We will denote by $\invh$ the image of the composite $\inv \subset \Hh_{\mud}(\Speck) \to \Hh_{\mud}(\Speck)[t^{-1}]$.
\end{para}

\begin{proposition}
\label{prop:in_R}
We have $\invh \subset \Rc \subset \Rc[t^{-1}] \subset \Hh_{\mud}(\Speck)[t^{-1}]$.
\end{proposition}
\begin{proof}
This follows from \rref{prop:in_Hf}.
\end{proof}

\begin{lemma}
\label{lemm:epsilon_epsilon}
For any smooth projective $k$-scheme $X$ with a $\mud$-action, the morphism $\epsilon$ of \rref{p:epsilon} satisfies $\epsilon(\lc X \rc_{\mud}) = \lc X \rc \in \Laz/2$ (using the notation of \rref{p:lc_rc} and \rref{p:class_in_inv}).
\end{lemma}
\begin{proof}
By \dref{lemm:R}{lemm:R:nzd}, the morphism $\Hh(\Speck)[[t]]/h \to \Hh(\Speck)[[t]][t^{-1}]/h = \Hh_{\mud}(\Speck)[t^{-1}]$ is injective. Its image coincides with the image of the morphism $\Hh_{\mud}(\Speck) \to \Hh_{\mud}(\Speck)[t^{-1}]$, and in particular contains $\Rc$. Thus \rref{prop:in_R} yields a commutative diagram
\[ \xymatrix{
\inv \ar[r] \ar[d] & \Hh_{\mud}(\Speck) \ar[d] \\ 
\Rc \ar[r] & \Hh(\Speck)[[t]]/h
}\]
On the other hand the following diagram is commutative
\[
\xymatrix{
\Rc\ar[r] \ar[d]_{\epsilon}& \Hh(\Speck)[[t]]/h \ar[d]^{t \mapsto 0} & \Hh_{\mud}(\Speck) \ar[l] \ar[d]^{\varepsilon \text{ (see \rref{p:change_of_groups})}}\\ 
\Laz/2 \ar[r] & \Hh(\Speck)/2 & \Hh(\Speck)\ar[l]
}
\]
By the Hattori--Stong theorem \cite[Proposition 7.16 (1)]{Mer-Ori}, the morphism $\Laz/2 \to \Hh(\Speck)/2$ is injective, and we conclude by combining the two diagrams (recall from \rref{p:varepsilon} that $\varepsilon(\lc X \rc_{\mud}) = \lc X \rc$).
\end{proof}

\begin{definition}
\label{p:J_def}
Under the mapping $x \mapsto t$, the morphism $\J$ of \rref{prop:gamma_determined} induces a morphism of $\Laz$-algebras (as in \rref{def:gamma_Gm}, we view $\g(-E)(t)$ as an element of $\Hh_{\mud}(S)$)
\[
\jj\colon \Mcc \to \Rc[t^{-1}] \quad ; \quad \lc E \to S\rc \mapsto \lc \g(-E)(t) \rc_{\mud} = \sum_{i\in \Zz} \lc \g_i(-E) \rc t^i.
\]
\end{definition}

\begin{para}
\label{p:j_trivial}
If $S$ is a smooth projective $k$-scheme and $E\to S$ is the trivial vector bundle of rank $r$, we have $\jj(\lc E \to S \rc) = t^{-r} \lc S \rc$.
\end{para}

\begin{para}
\label{p:j_v}
Since $\jj(v) = t^{-1}$, the morphism $\jj$ extends to $\jj \colon \Mcc[v^{-1}] \to \Rc[t^{-1}]$.
\end{para}

\begin{para}
\label{p:j_nu}
By \rref{lemm:X_Z_PN:2}, we have $\lc X \rc_{\mud} = \jj(\lc N_X \to X^{\mud}\rc) \in \Rc[t^{-1}]$ for any smooth projective $k$-scheme $X$ with a $\mud$-action.
\end{para}

\begin{lemma}
\label{lemm:alg_indep}
The elements $\lc \X_n \rc_{\mud} \in \invh$ for $n\in \Nn\smallsetminus\{0\}$ are algebraically independent (over $\Zz$).
\end{lemma}
\begin{proof}
Let $P$ be a nonzero polynomial in $\Zz[\bx]$ such that $P(\lc \X_1 \rc_{\mud},\ldots)=0 \in \invh \subset \Rc$. Since $\Rc$ is $2$-torsion free \rref{lemm:2-torsion-free}, we may assume that $P \not \in 2\Zz[\bx]$. By \rref{lemm:epsilon_epsilon} we have $P(\lc \X_1 \rc,\ldots) = \epsilon(P(\lc \X_1 \rc_{\mud},\ldots)) =0 \in \Laz/2$, contradicting \dref{cor:gen_Laz}{cor:gen_Laz:1}.
\end{proof}

\begin{theorem}
\label{th:J}
The morphism $\jj\colon \Mcc \to \Rc[t^{-1}]$ is injective.
\end{theorem}
\begin{proof}
Since $\lc \X_n \rc_{\mud} = \jj(x_n) \in \Rc[t^{-1}]$ for all $n\in \Nn\smallsetminus\{0\}$, it follows from \rref{lemm:alg_indep} that $\jj$ restricts to an injection on the subring $S\subset \Mcc$ generated by $x_n$ for $n\in \Nn\smallsetminus\{0\}$. Let now $x \in \Mcc[v^{-1}]$ be such that $\jj(x)=0$ in $\Rc[t^{-1}]$. Since $x_1=2v$ by \dref{ex:x_123}{ex:x_123:1} and in view of \rref{cor:gen_x_v}, we have $v^a2^bx \in S$ for some $a,b \in \Nn$. By injectivity of $\jj|_S$, it follows that $v^a2^bx =0 \in \Mcc \subset \Mcc[v^{-1}]$. Since $\Laz$ is $2$-torsion free, it follows from \rref{prop:M_pol} that $\Mcc[v^{-1}]$ is $(2,v)$-torsion free. We conclude that $x=0$.
\end{proof}

\begin{para}
\label{p:omit_jj}
In the sequel we will omit to mention of the morphism $\jj$, and think of $\Mcc$ as an $\Laz$-subalgebra of $\Rc[t^{-1}]$, where $v=t^{-1}$.
\end{para}

\subsection{Trivial normal bundle}
We are now in position to provide the first concrete application of the theory, by describing those $\mud$-actions having only isolated fixed points.

\begin{lemma}
\label{lemm:Rc_2}
Let us view $\Laz$ as a subring of $\Rc$ (see \rref{lemm:Laz_subring_Rc}). Then for all $m\in \Nn$, we have $\Laz \cap t^m\Rc = 2^m \Laz$ inside $\Rc$.
\end{lemma}
\begin{proof}
Since $x_1=2t^{-1} \in \Rc$ by \dref{ex:x_123}{ex:x_123:1}, we have $2^m\Laz=t^m x_1^m\Laz \subset t^m \Rc$. We prove the other inclusion by induction on $m$, the case $m=0$ being clear. Let $a\in\Laz \cap t^m\Rc$ with $m>0$. By induction, we have $a=2^{m-1}b$ for some $b\in \Laz$. Write $a=t^mx$ with $x \in \Rc$. Then $t^mx=2^{m-1}b = t^{m-1}x_1^{m-1}b$, hence $tx = x_1^{m-1}b$ by \dref{lemm:R}{lemm:R:nzd}. Thus, applying the morphism of \rref{p:epsilon}
\[
0 = \epsilon(tx) = \epsilon(x_1)^{m-1}b = \lc \Pp^1 \rc^{m-1} b \in \Laz/2.
\]
Since $\lc \Pp^1 \rc$ is a nonzerodivisor in $\Laz/2$ (e.g.\ by \dref{cor:gen_Laz}{cor:gen_Laz:1}), it follows that $b =0 \in \Laz/2$, hence $a\in 2^m \Laz$.
\end{proof}

\begin{proposition}
\label{prop:trivial_normal}
Let $X$ be a smooth projective $k$-scheme with a $\mud$-action such that the normal bundle $N_X$ is trivial of constant rank $c$. Then there exists $a\in \Laz$ such that $\lc X \rc_{\mud} = a (\lc \X_1 \rc_{\mud})^c \in \invh$ and $\lc X^{\mud} \rc = 2^c a \in \Laz$.
\end{proposition}
\begin{proof}
We have $\lc X \rc_{\mud} = t^{-c} \lc X^{\mud} \rc \in \Rc[t^{-1}]$ by \rref{p:j_nu} and \rref{p:j_trivial}. It follows from \rref{lemm:Rc_2} that $\lc  X^{\mud} \rc =2^ca$ for some $a\in \Laz$. Since $x_1=2t^{-1} \in \Rc$ by \dref{ex:x_123}{ex:x_123:1}, we obtain $\lc X \rc_{\mud} = a x_1^c \in \Rc$.
\end{proof}

\begin{corollary}
Let $X$ be a smooth projective $k$-scheme of pure dimension $n$ with a $\mud$-action. Assume that the number $q$ of fixed points over an algebraic closure of $k$ is finite. Then $q = 2^na$ for some $a \in \Nn$, and $\lc X \rc = a\lc \Pp^1 \rc^n$ in $\Laz/2$.
\end{corollary}
\begin{proof}
The vector bundle $N_X$ is trivial because $X^{\mud}$ is finite over $k$, and has constant rank $n$. We may thus apply \rref{prop:trivial_normal} with $c=n$. The first statement follows from the relation $\lc X^{\mud} \rc = q \in \Laz$, and the second by applying the morphism $\invh \subset \Rc \xrightarrow{\epsilon} \Laz/2$, in view of \rref{lemm:epsilon_epsilon}, as $\lc \X_1 \rc=\lc \Pp^1 \rc$.
\end{proof}

\section{The structure of \texorpdfstring{$\inv$}{O(\textmu 2)}}
\label{sect:structure}
We recall that $\Hh=\Kt$ (see \rref{p:K} and \rref{p:twist}). Starting from \rref{p:h_free}, we will assume that $\carac k \neq 2$.

\subsection{The morphism \texorpdfstring{$\nu$}{\textnu}}
\begin{proposition}
\label{prop:nu}
There exists a morphism of graded $\Laz$-algebras
\[
\nu \colon \inv \to \Mcc \quad ; \quad \lc X \rc_{\mud} \mapsto \lc N_X\to X^{\mud} \rc
\]
whose image is $\invh \subset \Rc[t^{-1}]$.
\end{proposition}
\begin{proof}
This follows from \rref{p:j_nu} and \rref{th:J}.
\end{proof}

\begin{para}
\label{p:fix}
Composing $\nu$ with the morphism $\varphi$ of \rref{p:Mcc_base}, we obtain a morphism of $\Laz$-algebras $\invh \to \Laz$ mapping $\lc X \rc_{\mud}$ to $\lc X^{\mud} \rc$.
\end{para}

\begin{para}
\label{p:dim_phi}
Let $X$ be a smooth projective $k$-scheme with a $\mud$-action. Viewing $\lc X \rc_{\mud}$ as an element of $\invh \subset \Mcc$, it follows from \rref{p:dim_bdim_dim} that we have (in the notation of \rref{p:dim})
\[
\dim (X^{\mud}) \geq \fdim(\lc X \rc_{\mud}) \geq \dim(\lc X^{\mud}\rc).
\]
\end{para}

\begin{proposition}
\label{prop:beta}
There exists a morphism of graded $\Laz$-modules
\[
\beta \colon \inv \to \Laz \quad ; \quad \lc X \rc_{\mud} \mapsto \lc Z \rc,
\]
where $Z$ is the quotient of the blow-up of $X^{\mud}$ in $X$ by its $\mud$-action (see \rref{p:blowup}).
\end{proposition}
\begin{proof}
Applying the forgetful morphism $\Hh_{\mud}(\Speck)[2^{-1}] \to \Hh(\Speck)[2^{-1}]$ to the formula \rref{lemm:X_Z_PN:3} and using \rref{cor:PE_M}, we see that $\lc Z \rc \in \Laz \subset \Hh(\Speck)[2^{-1}]$ depends only on $\lc X \rc \in \Laz$ and $\lc N_X \to X^{\mud} \rc \in \Mcc$. In view of \rref{p:varepsilon} and \rref{prop:nu}, these two quantities are determined by $\lc X \rc_{\mud} \in \inv$. Thus $\lc X \rc_{\mud} \mapsto \lc Z \rc$ induces a map $\inv^{\eff} \to \Laz$. This map is additive, hence yields the required morphism of groups $\beta$; its $\Laz$-linearity follows from the facts that $(X\times T)^{\mud} = X^\mud \times T$ and $(Y\times T)/{\mud} = (Y/\mud) \times T$ whenever $\mud$ acts trivially on $T \in \Sm_k$.
\end{proof}

\begin{lemma}
\label{lemm:nofix_h}
Let $x\in \inv$ be such that $\nu(x) = 0$. Then $x = h \beta(x)$.
\end{lemma}
\begin{proof}
Let us write $x = \lc X_1 \rc_{\mud} - \lc X_2 \rc_{\mud}$, where each $X_i$ is a smooth projective $k$-scheme with a $\mud$-action. For $i=1,2$, we use the notation $N_i,Z_i,\Lc_i$ of \rref{p:blowup}. The assumption implies that $\lc N_i \to (X_i)^{\mud}\rc \in \Mcc$ does not depend on $i\in \{1,2\}$. Let $\lambda_i = c_1(\Lc_i^{\vee}) \in \Hh(Z_i)$. By \cite[(4.8.iii)]{inv}, each $\Pp(N_i\sigma) \to Z_i$ is an effective Cartier divisor such that $\Oc_{Z_i}(\Pp(N_i\sigma)) = (\Lc_i^{\vee})^{\otimes 2}$. It follows from \rref{cor:PE_M} that for all $n \in \Nn$, the element $\lc c_1(\Oc_{\Pp(N_i\sigma)}(-1))^n \rc = \lc \fgl{2}(\lambda_i) \lambda_i^n \rc \in \Laz$ does not depend on $i \in \{1,2\}$. In view of \rref{eq:fgl} and since $\Laz$ is $2$-torsion free, we deduce by descending induction that $\lc \lambda_1^j\rc = \lc \lambda_2^j\rc \in \Laz$ for all $j \in \Nn \smallsetminus\{0\}$. Combining \rref{lemm:X_Z_PN:1} with \rref{cor:PE1_M} yields in $\Hh_{\mud}(\Speck)$
\[
x = \Big\lc \frac{\fgl{2}(\lambda_1 +_{\Hh} t)}{\fgl{-1}(\lambda_1 +_{\Hh} t)}\Big\rc_{\mud} - \Big\lc \frac{\fgl{2}(\lambda_2 +_{\Hh} t)}{\fgl{-1}(\lambda_2 +_{\Hh} t)}\Big\rc_{\mud}.
\]
As observed above, the terms involving positive powers of $\lambda_1,\lambda_2$ cancel out, leaving $x = \frac{\fgl{2}(t)}{\fgl{-1}(t)}(\lc Z_1 \rc - \lc Z_2 \rc)$. The statement follows, since $\fgl{-1}(t) = t$ (as $\sigma^\vee \simeq \sigma$).
\end{proof}

From now on, we will assume that $\carac k \neq 2$.

\begin{para}
\label{p:h_free}
It follows from \rref{lemm:nofix_h} that if $X$ is a smooth projective $k$-scheme with a $\mud$-action such that $X^{\mud} = \varnothing$, then $\lc X \rc_{\mud} = h \lc X/\mud \rc \in \inv$. Taking $X = \Speck \sqcup \Speck$ where $\mud = \Zz/2$ acts by exchanging the copies of $\Speck$, we see that $h \in \inv$, and that $\beta(h)=1$. Moreover the subgroup $h\Laz \subset \inv$ is generated by the classes $\lc X \rc_{\mud}$, where $X$ runs over the smooth projective $k$-schemes with a $\mud$-action such that $X^{\mud} = \varnothing$.
\end{para}

\begin{proposition}
\label{cor:decomp_h}
We have a split exact sequence of graded $\Laz$-modules
\[
0 \to \Laz \to \inv \to \invh \to 0,
\]
where the first morphism is induced by multiplication by $h$, the second is the natural morphism, and the splitting is given by the morphism $\beta$.
\end{proposition}
\begin{proof}
Since $\beta(h) =1$ (see \rref{p:h_free}), the morphism $\beta$ is an $\Laz$-linear retraction of the morphism $\Laz \to \inv$ given by $x \mapsto hx$. By \rref{lemm:nofix_h}, the image of the latter coincides with $\ker \nu$, or equivalently $\ker(\inv \to \invh)$ by \rref{prop:nu}. The surjectivity of the morphism $\inv \to \invh$ was proved in \rref{prop:nu}.
\end{proof}

\begin{proposition}
\label{prop:inv_invh}
The morphism of graded $\Laz$-algebras $\inv \to \invh \times \Laz$, induced by the natural morphism $\inv \to \invh$ and $\varepsilon \colon \inv \to \Laz$ of \rref{p:varepsilon}, is injective. Its image is the set of pairs $(x,\lambda)$ such that $\epsilon(x) = \lambda \in \Laz/2$ (see \rref{prop:in_R} and \rref{p:epsilon}).
\end{proposition}
\begin{proof}
Let $y \in \inv$, and denote by $\tilde{y} \in \invh$ its image. Then by \rref{lemm:epsilon_epsilon} the pair $(\tilde{y},\varepsilon(y))$ verifies $\epsilon(\tilde{y}) = \varepsilon(y) \in \Laz/2$. Assume now that $\tilde{y}=0$ and $\varepsilon(y) =0$. Then $y=ha$ for some $a\in \Laz$ by \rref{cor:decomp_h}. Since $\varepsilon(h)=2$, we obtain $0 = \varepsilon(y) = 2a \in \Laz$, hence $a=0$ as $\Laz$ is $2$-torsion free. Thus $y=0$.

Conversely, consider elements $x \in \invh$ and $\lambda \in \Laz$ satisfying $\epsilon(x) = \lambda \in \Laz/2$. Pick a lifting $z \in \inv$ of $x$. Since $\varepsilon(z) = \epsilon(x) \in \Laz/2$ by \rref{lemm:epsilon_epsilon}, we have $\varepsilon(z) - \lambda = 2b$ for some $b\in \Laz$. Then $z-bh \in \inv$ is the required preimage of $(x,\lambda) \in \invh \times \Laz$.
\end{proof}

\begin{corollary}
For any smooth projective $k$-scheme $X$ with a $\mud$-action, the class $\lc X \rc_{\mud} \in \inv$ is determined by $\lc N_X \to X^{\mud} \rc \in \Mcc$ and $\lc X \rc \in \Laz$.
\end{corollary}

\begin{para}
\label{p:inv_invh}
It follows from \rref{prop:inv_invh} that the graded $\Laz$-algebra $\inv$ can be reconstructed from the graded $\Laz$-subalgebra $\invh \subset \Rc$, so that we will focus on the latter.
\end{para}

\subsection{The fundamental exact sequence}
\label{sect:es}

\begin{para}
Recall from \S\ref{sect:M} that $a_0 = p_0 =1 \in \Pc$. It will be convenient to write $a_m=p_m=0$ when $m<0$, as well as $\Pp^m = \varnothing$. For $i \in \Zz$ consider the elements $u_i \in \fund{\Hh}$ given by \eqref{eq:fgl} for $n=2$. We define a morphism of $\fund{\Hh}$-modules $\delta \colon \Pc \to \Pc$ by setting
\[
\delta(va_m) = \sum_{i \in \Zz} u_i va_{m-i}.
\]
\end{para}

\begin{para}
\label{p:delta_p}
Note that (see \eqref{eq:pi_i})
\[
\delta(p_m) = \delta\Big(\sum_{i \in \Zz} \lc \Pp^{m-i} \rc va_i\Big) = \sum_{i,j \in \Zz} \lc \Pp^{m-i} \rc u_jv a_{i-j} = \sum_{n,j \in \Zz} \lc \Pp^{m-n-j} \rc u_jv a_n = \sum_{j \in \Zz} u_j p_{m-j}. 
\]
\end{para}

\begin{lemma}
\label{lemm:delta_n}
Let $S \in \Sm_k$ and $L\to S$ a line bundle. If $D \in \Sm_k$ is the zero-locus of a regular section of $L^{\otimes 2}$, then
\[
\delta(\lc L \to S \rc) = \lc L|_D \to D \rc.
\]
\end{lemma}
\begin{proof}
Using \eqref{eq:class_P}, \dref{lemm:zero-locus}{lemm:zero-locus:2} and the projection formula \cite[(2.1.3.iii)]{inv}, we have
\begin{align*}
\lc L|_D \to D \rc
&= v\sum_{i \in \Nn} \lc c_1(L|_D)^i\rc \ba_i
=  v\sum_{i \in \Nn} \lc c_1(L^{\otimes 2}) c_1(L)^i\rc \ba_i \\ 
&= \sum_{i,j\in \Nn} u_j v\lc c_1(L)^{i+j} \rc \ba_i
= \sum_{m,j\in \Nn} u_j \lc c_1(L)^m \rc v\ba_{m-j}\\
&= \sum_{m\in \Nn} \lc c_1(L)^m \rc \delta_n(v\ba_m)= \delta_n \Big(v\sum_{m\in \Nn} \lc c_1(L)^m \rc \ba_m\Big) = \delta(\lc L \to X \rc).\qedhere
\end{align*}
\end{proof}

\begin{lemma}
\label{lemm:im_nu_delta}
Under the inclusion $\Pc \subset \Mcc$, we have $\im \delta \subset \im \nu$.
\end{lemma}
\begin{proof}
In view of \rref{prop:P_basis}, it suffices to prove that $\delta(p_m)$ is contained in $\im \nu$ for all $m \in \Nn$. The line bundle $\Oc(2)$ over $\Pp^m$ admits a regular section $s$ whose zero-locus $Q$ is smooth. Consider the $\Oc_{\Pp^m}$-algebra $\Bc = \Oc_{\Pp^m} \oplus \Oc(-1)$, where the multiplication $\Oc(-1) \otimes \Oc(-1) \to \Oc_{\Pp^m}$ is given by $s$, and let $T = \Spec_{\Pp^m} \Bc$. The $\Oc_{\Pp^m}$-algebra $\Bc$ is $\Zz/2$-graded by letting $\Bc_0 = \Oc_{\Pp^m}$ and $\Bc_1 = \Oc(-1)$. This gives a $\mud$-action on $T$ such that $T \to \Pp^m$ is the $\mud$-quotient morphism. Moreover the composite $T^{\mud} \to T \to \Pp^m$ may be identified with the closed immersion $Q \to \Pp^m$. In particular $T \smallsetminus T^{\mud} \to \Pp^m \smallsetminus Q$ is an fppf $\mud$-torsor, which implies that $T \smallsetminus T^{\mud} \in \Sm_k$ (as $\carac k \neq 2$). On the other hand $T^{\mud} \in \Sm_k$ is an effective Cartier divisor in $T$ with normal bundle $\Oc(1)|_Q$. Thus $T \in \Sm_k$, and $\nu(\lc T \rc_{\mud}) = \lc \Oc(1)|_Q \to Q\rc = \delta(p_m)$ by \rref{lemm:delta_n}.
\end{proof}

\begin{lemma}
\label{lemm:Oc(1)}
Let $X$ be a smooth projective $k$-scheme with a $\mud$-action. Denote by $N$ the normal bundle to the immersion $X^{\mud} \to X$. Then
\begin{enumerate}[(i)]
\item \label{lemm:Oc(1):1}
We have $\lc \Oc(-1) \to \Pp(N)\rc \in \im \delta$.

\item \label{lemm:Oc(1):2}
We have $\lc \Oc(1) \to \Pp(N) \rc\in \im \delta$.
\end{enumerate}
\end{lemma}
\begin{proof}
We use the notation of \rref{p:blowup}. By \cite[(4.8.iii)]{inv}, the morphism $\Pp(N) \to Z$ is an effective Cartier divisor such that $\Oc_Z(\Pp(N)) = (\Lc^{\vee})^{\otimes 2}$. It follows from \rref{lemm:delta_n} that $\lc \Oc(-1) \to \Pp(N) \rc = \delta(\lc \Lc^{\vee} \to Z \rc)$ in $\Pc$. This proves \eqref{lemm:Oc(1):1}. Part \eqref{lemm:Oc(1):2} follows from \eqref{lemm:Oc(1):1} combined with \rref{lemm:dual_delta} below (which uses only \eqref{lemm:Oc(1):1}).
\end{proof}

\begin{lemma}
\label{lemm:dual_delta}
Let $S$ be a smooth  projective $k$-scheme, and $L\to S$ a line bundle. Then $\lc L \to S \rc + \lc L^\vee \to S\rc \in \im \delta$.
\end{lemma}
\begin{proof}
Let $X = \Pp(L\sigma \oplus 1)$. Then the normal bundle $N$ to the immersion $X^\mud \to X$ is the disjoint union of $L\to S$ and $L^\vee \to S$ (the respective normal bundles to $\Pp(1) \to \Pp(L\sigma \oplus 1)$ and $\Pp(L\sigma) \to \Pp(L\sigma \oplus 1)$). Since $N \to X^{\mud}$ is a line bundle, we may identify it with $\Oc(-1) \to \Pp(N)$, and the lemma follows from \dref{lemm:Oc(1)}{lemm:Oc(1):1}.
\end{proof}

\begin{para}
\label{p:partial}
We claim that there exists a unique morphism of abelian groups 
\[
\partial \colon \Mcc \to \Pc \quad ; \quad \lc E \to S \rc \mapsto  \lc \Oc(1) \to \Pp(E) \rc.
\]
Indeed by \rref{cor:PE_M}, this formula defines a map $\Mcc^{\eff} \to \Pc$, which is additive, hence extends to a unique morphism $\partial \colon \Mcc \to \Pc$.
\end{para}

\begin{para}
\label{p:nu_pi_delta}
When $S$ is a smooth projective $k$-scheme and $E\to S$ a vector bundle, the normal bundle to the fixed locus of $\Pp(E\sigma \oplus 1)$ is the disjoint union of $E \to S$ and $\Oc(1) \to \Pp(E)$. Therefore the morphisms $\pi$ of \rref{p:pi}, $\nu$ of \rref{prop:nu} and $\partial$ of \rref{p:partial} satisfy the relation
\[
\nu \circ \pi(e)= e + \partial(e) \quad \text{for all $e\in \Mcc$},
\]
where $\partial(e) \in \Pc$ is considered as an element of $\Mcc$ using the inclusion $\Pc \subset \Mcc$.
\end{para}

\begin{theorem}
\label{th:es}
We have an exact sequence of $\Laz$-modules
\[
0  \to \Laz \to \inv \xrightarrow{\nu} \Mcc \to \coker \delta \to 0,
\]
where the first morphism is induced by multiplication with $h$, and the last one is the composite $\Mcc \xrightarrow{\partial} \Pc \to  \coker \delta$.
\end{theorem}
\begin{proof}
Exactness at $\Laz$ and $\inv$ has been proved in \rref{cor:decomp_h} (in view of \rref{prop:nu}), and it follows from \dref{lemm:Oc(1)}{lemm:Oc(1):2} that the sequence is a complex. Since $\partial(p_0^{n+1})=p_n$ for all $n\in \Nn$, the morphism $\partial$ is surjective, hence so is the last morphism in the sequence. Let now $e\in \Mcc$ be such that $\partial(e) \in \im \delta$. Under the inclusion $\Pc \subset \Mcc$, we have $\partial(e) = \nu(x) \in \Mcc$ for some $x \in \inv$ by \rref{lemm:im_nu_delta}. Thus by \rref{p:nu_pi_delta}, we have $\nu( \pi(e) - x) = e + \partial(e) - \nu(x) = e$, hence $e\in \im \nu$, proving exactness at $\Mcc$.
\end{proof}

\begin{remark}
By \rref{prop:M_pol}, the $\Laz$-modules $\Mcc$ and $\coker \delta$ do not depend on the field $k$, and one may see that \rref{th:es} provides an intrinsic description of $\inv$. This observation will be formalised in \rref{cor:indep_field} below.
\end{remark}

\subsection{Integrality}
\label{sect:integrality}
Replacing $x$ by $t$ in \rref{prop:gamma_determined}, we have a morphism $\J \colon \Mcc \to \Laz[[t]][t^{-1}]$.

\begin{para}
Using the elements $p_i$ of \rref{p:p_i} (where $p_i=0$ for $i<0$), we define a morphism of $\Laz$-modules
\[
A \colon \Laz[[t]][t^{-1}] \to \Pc \quad ; \quad t^j \mapsto p_{-1-j} \text{ for $j\in\Zz$}.
\]
\end{para}

\begin{lemma}
\label{lemm:alpha_gamma}
We have $A \circ \J = \partial$.
\end{lemma}
\begin{proof}
Let $E$ be a vector bundle over a smooth projective $k$-scheme $S$. Using Quillen's formula \rref{prop:Quillen} and in view of \eqref{eq:pi_i} and \eqref{eq:class_P}, we have in $\Pc$
\begin{align*}
A \circ \J(\lc E \to S \rc) &= \sum_{j \in \Zz} \lc \g_j(-E)\rc p_{-1-j}\\
&= \sum_{j,m \in \Zz} \lc\g_j(-E)\rc \lc \Pp^{-1-j-m} \rc v \ba_m= \sum_{i,m \in \Zz} \lc\g_{-1-i-m}(-E)\rc \lc \Pp^i \rc v\ba_m\\
&= \sum_{m \in \Nn} \lc c_1(\Oc_{\Pp(E)}(1))^m \rc v\ba_m = \lc \Oc(1) \to \Pp(E) \rc.\qedhere
\end{align*}
\end{proof}

\begin{para}
By \rref{p:delta_p}, we have $A(ht^n) = \delta(p_{-n})$ for any $n\in \Zz$, and therefore $A(h \Laz[[t]][t^{-1}]) \subset \im \delta$. It follows that $A$ descends to a morphism of $\Laz$-modules
\[
\alpha \colon \Rc[t^{-1}] \to \coker \delta.
\]
Since $\ker A = \Laz[[t]] \subset \Laz[[t]][t^{-1}]$, we deduce that
\begin{equation}
\label{eq:ker_alpha}
\Rc \subset \ker \alpha \subset \Rc[t^{-1}].
\end{equation}
\end{para}

\begin{theorem}
\label{th:integrality}
We have $\invh = \Rc \cap \Mcc$ inside $\Rc[t^{-1}]$ (recall \rref{p:omit_jj}).
\end{theorem}
\begin{proof}
We have $\invh \subset \Rc$ by \rref{prop:in_R}, and $\invh \subset \Mcc$ by \rref{p:j_nu}. Conversely let $x \in \Mcc$. In view of the commutative diagram
\[ \xymatrix{
\Mcc \ar[rr]^{\Gamma} \ar[rrd]_{\jj}&&\Laz[[t]][t^{-1}] \ar[r]^A \ar[d] & \Pc \ar[d] \\ 
&&\Rc[t^{-1}] \ar[r]^\alpha & \coker \delta
}\]
it follows from \rref{lemm:alpha_gamma} that $\alpha(x) \in \coker \delta$ is the image of $\partial(x) \in \Pc$. If additionally $x \in \Rc$, that image vanishes by \eqref{eq:ker_alpha}, so that $x \in \invh$ by \rref{th:es}.
\end{proof}

\begin{corollary}
\label{cor:indep_field}
The graded $\Laz$-algebra $\inv$ does not depend on the field $k$ (of characteristic $\neq 2$).
\end{corollary}
\begin{proof}
Since $\fund{\Hh} = \Laz$ with its canonical formal group law (see \rref{p:K_fund}), the Laurent series $v_i(x) \in \Laz[[x]][x^{-1}]$ of \rref{eq:inverse_fgl}, as well as the $\Laz$-algebra $\Rc=\Laz[[t]]/h$, do not depend on $k$. In view of \rref{prop:M_pol}, the $\Laz$-subalgebra $\Mcc\subset \Rc[t^{-1}]$ is generated by
\[
p_n = \lc \g(-\Oc_{\Pp^n}(1))(t) \rc_{\mud} = \sum_{i=0}^n v_i(t) \lc \Pp^{n-i} \rc \quad \text{ for all $n\in \Nn$.}
\]
By \cite[Example~3.5]{Mer-Ori}, for each $m\in \Nn$, the class $\lc \Pp^m \rc \in \Laz$ does not depend on $k$. We deduce that the $\Laz$-subalgebra $\Mcc \subset \Rc[t^{-1}]$ is independent of $k$, and so is the $\Laz$-subalgebra $\invh \subset \Rc$ by \rref{th:integrality}. Thus the corollary follows from \rref{p:inv_invh}.
\end{proof}

\subsection{Stable generators}

\begin{para}
In view of \rref{p:omit_jj} and \rref{prop:nu}, the elements $x_n$ defined in \rref{def:x_n} may be regarded as elements of $\invh$. Thus we have $x_n = \lc \X_n \rc_{\mud} \in \invh$.
\end{para}

\begin{proposition}
\label{prop:M_v-1}
The morphism $\inv \to \Mcc/(v-1)$ induced by $\nu$ is surjective.
\end{proposition}
\begin{proof}
Since each $x_n$ belongs to its image, this follows from \rref{cor:gen_x_v}.
\end{proof}

\begin{remark}
A vector bundle $E$ over a smooth projective $k$-scheme $S$ has a class $\lc E \to S \rc \in \Mcc/(v-1)$, which is determined by the elements $\lc c_\alpha(E) \rc \in \Laz$ for all partitions $\alpha$ (see \rref{eq:def_Mcc}). Proposition \rref{prop:M_v-1} asserts that each such class can be expressed as a difference of classes of normal bundles to fixed loci of involutions.
\end{remark}

\begin{lemma}
\label{lemm:gen_by_t_x}
The ring $\invh \subset \Mcc[v^{-1}]$ is contained in the subring generated by $v^{-1}$ and $x_n$ for $n\geq 1$.
\end{lemma}
\begin{proof}
Let $x \in \invh$. In view of \rref{cor:gen_x_v}, there are $m,n \in \Nn$, and for $i=-m,\ldots,n$ elements $d_i = P_i(x_1,\ldots)$ where $P_i \in \Zz[\bx]$ such that (recall that $t=v^{-1}$)
\begin{equation}
\label{eq:x}
x = d_{-m} t^{-m} +\cdots +d_0 + \cdots+ d_n t^n \in \Rc[t^{-1}].
\end{equation}
We may assume that $m$ is minimal among those appearing in such a formula. Assume for a contradiction that $m>0$. Since $x$ and each $d_i$ belong to $\Rc$ (by \rref{prop:in_R}), multiplying \eqref{eq:x} with $t^m$, we deduce that $d_{-m} \in t\Rc$. Applying the morphism of \rref{p:epsilon} we obtain $\epsilon(d_{-m}) = P_{-m}(\epsilon(x_1),\ldots)$ vanishes in $\Laz/2$. It follows from \dref{cor:gen_Laz}{cor:gen_Laz:1} (and \rref{lemm:epsilon_epsilon}) that the image of $P_{-m}$ vanishes in $\Fd[\bx]$. Thus there is $Q_{-m} \in \Zz[\bx]$ such that $P_{-m} = 2Q_{-m}$. This contradicts the minimality of $m$, because
\[
d_{-m} t^{-m} = t^{1-m} 2t^{-1} Q_{-m}(x_1,\ldots) = t^{1-m} x_1 Q_{-m}(x_1,\ldots).\qedhere
\]
\end{proof}

\begin{proposition}
\label{prop:bound}
We have $\Zz[x_1,\ldots] \subset \invh \subset \Zz[t,x_1,\ldots]/(tx_1-2)$.
\end{proposition}
\begin{proof}
The first inclusion is \rref{lemm:alg_indep}. Let $R=\Zz[t,x_2,\ldots]$. Then $\Mcc[v^{-1}] = R[t^{-1}] = R[v]/(vt-1)$ by \dref{prop:Mcc_pol}{prop:Mcc_pol:poly}. In view of \rref{lemm:gen_by_t_x}, it will suffice to prove that the $R$-subalgebra of $R[v]/(vt-1)$ generated by $x_1=2v$ (see \dref{ex:x_123}{ex:x_123:1}) is isomorphic to $R[x_1]/(tx_1-2)$, in other words that, inside $R[v]$,
\[
R[2v] \cap (vt -1)R[v] = 2(vt-1)R[2v].
\]
Let $P(v) \in R[v]$ be a polynomial such that $(vt-1)P(v)$ belongs to $R[2v]$. We prove by induction on $n=\deg P$ that $(vt-1)P(v) \in 2(vt-1)R[2v]$, the case $P=0$ being clear. Write $P(v) = p v^n + Q(v)$ where $p \in R$ and $\deg Q < n$. Looking at the $v^{n+1}$-coefficient of $(vt-1)P(v)$, we see that $tp \in 2^{n+1} R$.  Since $t$ is nonzerodivisor in $R/2^{n+1} = (\Zz/2^{n+1})[t,x_2,\ldots]$, it follows that $p \in 2^{n+1}R$. Therefore $(vt -1)(P(v) - Q(v)) = (vt-1)pv^n \in 2(vt-1)R[2v]$. Now $(vt-1) Q(v) = (vt-1)P(v) - (vt-1)pv^n \in R[2v]$, and the statement follows by using the induction hypothesis on $Q$.
\end{proof}

\begin{theorem}
\label{th:poly}
Let $X$ be a smooth projective $k$-scheme of pure dimension $n$ with a $\mud$-action. Then there are unique polynomials $A_i \in \Zz[\bx]$ for $i \in \Nn$ such that
\begin{enumerate}[(i)]
\item \label{th:poly:sum} $A_i=0$ for $i$ large enough, and $\displaystyle{\lc X \rc_{\mud} = \sum_{i \in \Nn} A_i(x_1,\ldots)t^i \in \Rc}$,

\item \label{th:poly:x_1} $A_i \in \Zz[X_2,\ldots] \subset \Zz[\bx]$ for $i\geq 1$.
\setcounter{enumi_resume}{\value{enumi}}
\end{enumerate}
In addition,
\begin{enumerate}[(i)]
\setcounter{enumi}{\value{enumi_resume}}
\item \label{th:poly:hom} \label{th:poly:deg} for $i \in \Nn$, the polynomial $A_i$ is homogeneous of degree $n+i$,

\item  \label{th:poly:Phi} for $i \in \Nn$, we have $\fdeg A_i \leq \dim (X^{\mud})$,
\item \label{th:poly:unstable} for $i > \max\{0,3\dim(X^{\mud})-n\}$, we have $A_i=0$.
\end{enumerate}
\end{theorem}
\begin{proof}
The existence of the polynomials $A_i$ satisfying \eqref{th:poly:sum} follows from \rref{lemm:gen_by_t_x} (recall that $t=v^{-1}$). Since $x_1t=2$ (see \dref{ex:x_123}{ex:x_123:1}), we see that \eqref{th:poly:x_1} may be arranged.

Assume now given polynomials $A_i$ satisfying \eqref{th:poly:sum} and \eqref{th:poly:x_1}. By \dref{prop:Mcc_pol}{prop:Mcc_pol:poly} we have a decomposition, as $\Zz[x_2,\ldots]$-modules,
\[
\Mcc[v^{-1}] = \Zz[v][x_2,\ldots] \oplus \bigoplus_{i\in\Nn\smallsetminus\{0\}} v^{-i} \Zz[x_2,\ldots].
\]
The component of $\lc X \rc_{\mud} \in \Mcc[v^{-1}]$ in $\Zz[v][x_2,\ldots]$ is $A_0(x_1,x_2,\ldots)$, and for $i \geq 1$ its component in $v^{-i} \Zz[x_2,\ldots]$ is $A_i(0,x_2,\ldots)v^{-i}$. This observation implies the unicity, as well as \rref{th:poly:hom} since by \dref{prop:Mcc_pol}{prop:Mcc_pol:deg} the above decomposition is compatible with the gradings \rref{p:grading} and \rref{p:default_grading} (with $\deg v=1$). That decomposition is also compatible with gradings \rref{p:base-grading} and \rref{p:fdeg} (with $\deg v=0$), therefore \dref{prop:Mcc_pol}{prop:Mcc_pol:Phi} implies that $\fdim(\lc X \rc_{\mud}) = \max_{i \in \Nn} \fdeg A_i$, so that \eqref{th:poly:Phi} follows from \rref{p:dim_phi}.

Finally, by \eqref{th:poly:x_1}, \eqref{th:poly:Phi} and \rref{cor:deg_fdeg_3}, we have $\deg A_i \leq 3\fdeg A_i \leq 3\dim(X^{\mud})$ when $i \geq 1$. In view of \eqref{th:poly:deg}, this yields \eqref{th:poly:unstable}.
\end{proof}

\begin{example}
\label{ex:P1xP1}
Let us work out the decomposition of \rref{th:poly} for the scheme $S=\Pp^1 \times \Pp^1$ endowed with the $\mud$-action induced by the involution exchanging the factors (recall that $\carac k\neq 2$). The fixed locus is isomorphic to $\Pp^1$, with normal bundle $\Oc(2)$. Thus $\lc S \rc_{\mud} = \lc \Pp^1 \rc v + 2a_1v$ in $\Mcc$. In view of \rref{ex:x_123}, we deduce that
\[
\lc S \rc_{\mud} = -x_1^2 + 4x_2 - tx_3 \in \Rc.
\]
\end{example}

\subsection{Low-dimensional fixed loci}
In this section, we assume that $X$ is an equidimensional smooth projective $k$-scheme with a $\mud$-action (where $\carac k \neq 2$). We let $n=\dim X$ and $d =\dim (X^{\mud})$. 

\begin{theorem}
\label{th:small_dim}
If $n\geq 3d$, there exists $P \in \Zz[\bx]$ such that
\[
\lc X \rc_{\mud} = P(x_1,\ldots) \in \invh.
\]
In addition, the polynomial $P$ is divisible by $X_1^{n-3d}$. More precisely $P = X_1^{n-3d+1} Q + a X_1^{n-3d} X_3^d$, where $Q \in \Zz[\bx]$ and $a\in \Zz$.
\end{theorem}
\begin{proof}
Let us apply \rref{th:poly}. In view of \dref{th:poly}{th:poly:unstable}, we have $A_i=0$ for $i>0$, so we may set $P=A_0$. Write $P = (X_1)^m R$, where $R$ is not divisible by $X_1$. By \dref{th:poly}{th:poly:hom}, the polynomial $R$ is homogeneous of degree $n-m$ with respect to the grading \rref{p:default_grading}. Applying \rref{cor:deg_fdeg_3} to a monomial of $R$ not divisible by $X_1$, we deduce that $3\fdeg R \geq n-m$. Since $d \geq \fdeg P= \fdeg R$ by \dref{th:poly}{th:poly:Phi}, we obtain $3d \geq n-m$, hence $m \geq n-3d$. Thus we may write $P=X_1^{n-3d}(X_1Q + T)$, where $Q\in \Zz[\bx]$, and $T \in \Zz[X_2,\ldots]$. We have $d \geq \fdeg P \geq \fdeg T$ by \dref{th:poly}{th:poly:Phi}, and $3\fdeg T \geq \deg T = 3d$ by \rref{cor:deg_fdeg_3}. Therefore $\fdeg T=d$ and $\deg T = 3d$. Assume that $T$ is not a $\Zz$-multiple of $X_3^d$. Then $T$ contains as a monomial a nonzero multiple of $X_3^eU$, where $e<d$ and $U$ is a product of the variables $X_j$ for $j$ even or $j \geq 5$. Applying \eqref{prop:deg_fdeg} with $p=e,s=2$ and $i_1=\cdots=i_p=1$, we obtain
\[
3d = \deg T \leq e/2 + 5(\fdeg T) /2 < d/2 + 5d/2 = 3d,
\]
a contradiction.
\end{proof}

\begin{corollary}
\label{cor:x1^s}
Let $s= \max\{0, 3d-n\}$. Then there exists $P\in \Zz[\bx]$ such that $x_1^s\lc X \rc_{\mud} = P(x_1,\ldots)$ in $\invh$.
\end{corollary}
\begin{proof}
We replace $X$ with $X \times (\X_1)^s$, and apply \rref{th:small_dim}.
\end{proof}

\begin{corollary}
\label{cor:n_3d_det}
Assume that $n \geq 3d$. 
\begin{enumerate}[(i)]
\item \label{cor:n_3d_det:1}
The class $\lc X \rc_{\mud} \in \invh/2$ is determined by $\lc X \rc \in \Laz/2$.

\item \label{cor:n_3d_det:2}
The element $\lc X \rc_{\mud} \in \inv$ is determined by $\lc X \rc \in \Laz$ and $\lc Z \rc \in \Laz$, where $Z$ is the quotient of the blow-up of $X^{\mud}$ in $X$ by its $\mud$-action (see \rref{p:blowup}).
\end{enumerate}
\end{corollary}
\begin{proof}
\eqref{cor:n_3d_det:1}: By \dref{cor:gen_Laz}{cor:gen_Laz:1}, the morphism $\invh/2 \to \Laz/2$ induced by $\epsilon$ is injective on the $\Fd$-subalgebra generated by $x_n$ for $n\in \Nn\smallsetminus\{0\}$, and \eqref{cor:n_3d_det:1} follows from \rref{th:small_dim}.

\eqref{cor:n_3d_det:2}: In view of \rref{prop:inv_invh}, it suffices to prove that $\lc X \rc_{\mud} \in \invh$ is determined by those classes. We use the morphism $\beta$ constructed in \rref{prop:beta}. The morphism of $\Laz$-modules $2\beta - \varepsilon \colon \inv \to \Laz$ vanishes on $h$, hence descends to $\invh \to \Laz$. In view of \rref{th:small_dim}, it will suffice to prove that its kernel contains no nonzero element of the form $P(x_1,\ldots)$ with $P\in \Zz[\bx]$. As $\Laz$ is $2$-torsion free, it suffices to prove that such $P$ must belong to $2\Zz[\bx]$. Since $0=\epsilon(P(x_1,\ldots)) =P(\lc \X_1 \rc,\ldots) \in \Laz/2$, this follows from \dref{cor:gen_Laz}{cor:gen_Laz:1}.
\end{proof}

\begin{corollary}
\label{cor:Gm}
Assume that $n \geq 3d$. Then $\lc X \rc_{\mud} \in \invh$ is the image of an element of $\Inv(\Gm)$. Moreover there are smooth projective $k$-schemes $X_1,X_2$ of pure dimension $n$ with a $\Gm$-action, such that $\lc X \rc = \lc X_1 \rc - \lc X_2 \rc \in \Laz/2$ and $ \lc X^{\mud} \rc = \lc (X_1)^{\Gm} \rc  - \lc (X_2)^{\Gm} \rc \in \Laz$.
\end{corollary}
\begin{proof}
This follows from \rref{rem:X_n_indep_G} and \rref{th:small_dim}, using the morphisms $\invh \subset \Rc \xrightarrow{\epsilon} \Laz/2$ (in view of \rref{lemm:epsilon_epsilon}) and $\varphi \colon \invh \to \Laz$ of \rref{p:fix}.
\end{proof}

\section{Applications}
\label{sect:applications}
In this section we provide a series of applications, where the cobordism ring of involutions is used to prove statements not explicitly involving that ring. We recall that $\Hh=\Kt$ and that $\carac k \neq 2$. 

\subsection{Ambient variety}
\label{sect:5/2}
We will assume (until the end of the paper) that $X$ is an equidimensional smooth projective $k$-scheme with a $\mud$-action (where $\carac k \neq 2$). We let $n=\dim X$ and $d =\dim (X^{\mud})$ (we recall that $\dim \varnothing = -\infty$).

\begin{para}
Let us fix a family $\ell_i \in \Laz^{-i}/2$ for $i\in \Nn\smallsetminus\{0\}$, which generates $\Laz/2$ as a polynomial algebra over $\Fd$. Examples of such families are given in \rref{cor:gen_Laz}.
\end{para}

\begin{theorem}
\label{th:bdeg}
Write $\lc X \rc = P(\ell_1,\ldots) \in \Laz/2$, where $P \in \Fd[\bx]$. Then $\fdeg P \leq d$ (see \rref{p:fdeg}).
\end{theorem}
\begin{proof}
Let us apply \rref{th:poly}. Let $\bar{A_0} \in \Fd[\bx]$ be the image of $A_0$. Then $d \geq \fdeg A_0 \geq \fdeg \bar{A_0}$ by \dref{th:poly}{th:poly:Phi}. Since $\lc X \rc = \bar{A_0}(\epsilon(x_1),\ldots) \in \Laz/2$ (recall from \rref{p:epsilon} that $\epsilon(t)=0$), it follows from \dref{cor:gen_Laz}{cor:gen_Laz:1} and \rref{lemm:change_of_gen} that $\fdeg \bar{A_0} = \fdeg P$.
\end{proof}

\begin{corollary}
\label{th:ideal_power}
Let $q \in \Nn$ and $s\in \Nn\smallsetminus\{0\}$. Let $J(s)$ be the ideal of $\Laz/2$ generated by the homogeneous elements of degrees $-2i-1$ for $i=0,\ldots,s-1$. If $\lc X \rc \in \Laz/2$ does not belong to $J(s)^{q+1}$, then
\[
n \leq \Big(2 + \frac{1}{s}\Big) d + q.
\]
\end{corollary}
\begin{proof}
Write $\lc X \rc = P(\ell_1,\ldots) \in \Laz/2$ with $P \in \Fd[\bx]$. Then $P$ satisfies the conditions of \eqref{cor:deg_fdeg_2}, for some $p \leq q$. The statement follows from \rref{th:bdeg}.
\end{proof}

\begin{remark}
\label{rem:indec}
Taking $s>d$ in \rref{th:ideal_power} and $q=1$, we recover \cite[(7.3.4)]{inv}, which asserts that $n\leq 2d+1$ when $\lc X \rc \in \Laz/2$ is indecomposable (see \rref{p:decomposable}).
\end{remark}

\begin{corollary}
If $n \geq 3d$, then $\lc X \rc \in \Laz/2$ is divisible by $\lc \Pp^1 \rc^{n-3d}$.
\end{corollary}
\begin{proof}
Since $J(1)$ is generated by $\lc \Pp^1 \rc$ (see \dref{cor:gen_Laz}{cor:gen_Laz:1}), we may take $s=1$ in \rref{th:ideal_power}. (Alternatively, this follows from \rref{th:small_dim}).
\end{proof}

We deduce an algebraic version of Boardman's theorem \cite[Theorem 1]{Boardman-BAMS}:
\begin{corollary}
\label{cor:Boardman}
If a Chern number of $X$ is odd, then $n \leq 5d/2$.
\end{corollary}
\begin{proof}
Reduction modulo $2$ induces a ring morphism $\Laz \subset \Zz[\bb] \to \Fd[\bb]$. Its kernel is generated as an abelian group by classes of smooth projective $k$-schemes all of whose Chern numbers are even, and contains all homogeneous elements of degrees $-1$ and $-3$ (see e.g.\ \cite[(1.1.4)]{fpt}). Thus we may take $s=2$ and $q=0$ in \rref{th:ideal_power}.
\end{proof}

\begin{definition}
We say that an element of $x$ of $\Laz$ is \emph{$\alpha$-primitive modulo $2$} if $c_\alpha(x) \not \in 2c_\alpha(\Laz)$ (using the notation of \rref{p:c_alpha}). This property depends only on the class of $x$ in $\Laz/2$, and we will thus use this terminology for elements of $\Laz/2$.
\end{definition}

\begin{remark}
Let $\alpha = (\alpha_1,\dots,\alpha_m)$ be a partition.
\begin{enumerate}[(i)]
\item Let $x\in \Laz$. If $c_\alpha(x)$ is odd, then $x$ is $\alpha$-primitive modulo $2$.

\item Let $x\in \Laz$, and assume that $\alpha \neq \varnothing$. If $c_\alpha(x)$ is not divisible by $4$ and each $\alpha_i+1$ is a power of two, then $x$ is $\alpha$-primitive modulo $2$ (see \cite[(7.3.3)]{inv}).

\item The element $\ell_\alpha$ is $\alpha$-primitive modulo $2$ when $\length(\alpha)=1$ (see \cite[(7.3.2.iii)]{inv}), but not in general. For instance $c_{(11)}(\ell_1^2)$ is divisible by $4$, while $c_{(11)}(\ell_2)$ is not.
\end{enumerate}
\end{remark}

\begin{lemma}
\label{lemm:decomp_primitive}
If $\ell_\beta$ is $\alpha$-primitive modulo $2$, then $\alpha \succ \beta$ (see \rref{p:def_succ}).
\end{lemma}
\begin{proof}
Let us write $\beta = (\beta_1,\dots,\beta_m)$. For each $i$ choose a lifting $\ell_i' \in \Laz^{-i}$ of $\ell_i$. Since by assumption $c_\alpha(\ell'_\beta)$ is nonzero, it follows from \rref{eq:c_alpha_prod} that there exist partitions  $\alpha^1,\dots, \alpha^m$ such that $\alpha = \alpha^1 \cup \cdots \cup \alpha^m$ and the integer $c_{\alpha^1}(\ell'_{\beta_1}) \cdots c_{\alpha^m}(\ell'_{\beta_m})$ is nonzero. When $i=1,\dots,m$, the integer $c_{\alpha^i}(\ell'_{\beta_i})$ vanishes unless $|\alpha^i| = \beta_i$, which implies the statement.
\end{proof}

\begin{corollary}[of \rref{th:bdeg}]
\label{cor:alpha_primitive}
Let $\alpha = (\alpha_1,\ldots,\alpha_m)$ be a partition. If $\lc X \rc$ is $\alpha$-primitive modulo $2$, then 
\[
\left\lfloor \frac{\alpha_1}{2} \right\rfloor + \cdots + \left\lfloor \frac{\alpha_m}{2} \right\rfloor \leq d.
\]
In other words  $n \leq 2d + c$, where $c$ is the number of indices $i$ such that $\alpha_i$ is odd.
\end{corollary}
\begin{proof}
Write $\lc X \rc = P(\ell_1,\ldots) \in \Laz/2$, with $P \in \Fd[\bx]$. Then $P$ contains a nonzero multiple of a monomial $X_\beta$ such that $\ell_\beta$ is $\alpha$-primitive modulo two. We must have $\alpha \succ \beta$ by \rref{lemm:decomp_primitive}. By \rref{th:bdeg} and \rref{p:bdeg_succ}, we have
\[
d \geq \fdeg P \geq \fdeg X_{\beta} \geq \fdeg X_\alpha = \left\lfloor \frac{\alpha_1}{2} \right\rfloor + \cdots + \left\lfloor \frac{\alpha_m}{2} \right\rfloor.
\]
The last statement follows from the fact that $|\alpha|=n$.
\end{proof}

\subsection{Fixed locus}
\label{sect:smalldim}

\begin{theorem}
\label{th:fixed}
The element $\lc X^{\mud} \rc \in \Laz$ belongs to the subgroup generated by the elements $P(\lc (\X_1)^{\mud} \rc,\ldots)$, where $P \in \Zz[\bx]$ is homogeneous of degree at least $n$ and satisfies $\fdeg P \leq d$.
\end{theorem}
\begin{proof}
Applying the morphism $\varphi \colon \Mcc[v^{-1}] \to \Laz$ of \rref{p:Mcc_base} to the equation of \dref{th:poly}{th:poly:sum} yields, in view of \rref{p:fix},
\[
\lc X^{\mud} \rc = \varphi(\lc X \rc_{\mud}) = \sum_{i=0}^m A_i(\varphi(x_1),\ldots) =  \sum_{i=0}^m A_i(\lc(\X_1)^{\mud} \rc,\ldots)\in \Laz.
\]
The theorem thus follows from \dref{th:poly}{th:poly:deg} and \dref{th:poly}{th:poly:Phi}.
\end{proof}

\begin{corollary}
\label{cor:fixed}
Let $q \in \Nn$ and $s \in \Nn\smallsetminus\{0\}$. Let $L(s)$ be the ideal of $\Laz$ generated by $\lc (\X_{2i+1})^{\mud} \rc$ for $i=0,\ldots,s-1$. If $\lc X^{\mud} \rc$ does not belong to $L(s)^{q+1}$, then
\[
n \leq \Big(2 +\frac{1}{s}\Big)d + q.
\]
\end{corollary}
\begin{proof}
By \rref{th:fixed}, there exists a polynomial $P \in \Zz[\bx]$ homogeneous of degree at least $n$ such that $\fdeg P \leq d$ and $P(\lc (\X_1)^{\mud} \rc,\ldots) \not \in L(s)^{q+1}$. Then $P$ must satisfy the conditions of \rref{cor:deg_fdeg_2} for some $p \leq q$, whence the result.
\end{proof}

\begin{corollary}
\label{prop:small_dim_fix}
If $n \geq 3d$, then $\lc X^{\mud} \rc \in 2^{n-3d}\Laz$.
\end{corollary}
\begin{proof}
Since $\lc (\X_1)^{\mud} \rc=\varphi(x_1) = 2$ by \dref{ex:x_123}{ex:x_123:1}, we may take $s=1$ in \rref{cor:fixed}.
\end{proof}

\begin{remark}
Corollary \rref{prop:small_dim_fix} also follows from \rref{th:small_dim}, which more precisely yields $\lc X^{\mud} \rc \in 2^{n-3d}(2\Laz + 3^d\lc\Pp^1\rc^d \Zz) \subset \Laz$ when $n \geq 3d$.
\end{remark}

\begin{corollary}
\label{cor:n-5d}
Let $q \in \Nn$, and assume that a Chern number of $X^{\mud}$ is not divisible by $2^{q+1}$. Then $n\leq q+ 5d/2$.
\end{corollary}
\begin{proof}
We take $s=2$ in \rref{cor:fixed}. Since $\lc \Pp^1 \rc=-2b_1 \in \Laz \subset \Zz[\bb]$, the elements $\varphi(x_1)=2$ and $\varphi(x_3)=3\lc \Pp^1 \rc$ (see \rref{ex:x_123}) belong to $2\Zz[\bb]$. Thus $L(2)^{q+1} \subset 2^{q+1}\Zz[\bb]$.
\end{proof}

\begin{remark}
\label{rem:isolated}
Assume that $k$ is algebraically closed, and that the number of isolated fixed points of $X$ is not divisible by $2^{q+1}$. Then \rref{cor:n-5d} implies that $X^{\mud}$ has a component of dimension at least $2(n-q)/5$.
\end{remark}

\subsection{Fixed-dimensional components}
For $i=0,\dots,d$ let us denote by $F_i$ the union of the $i$-dimensional components of $X^{\mud}$.

\begin{corollary}
\label{cor:fixed_component}
Let $s \in \Nn \smallsetminus\{0\}$. Let $I(s)$ be the ideal of $\Laz$ generated by $2$ and the homogeneous elements of degrees $-i$ for $i =1,\dots,s-1$. Let $q \in \Nn$ and $j \in \{0,\dots,d\}$. If $\lc F_{d-j} \rc \not \in I(s)^{q+1}$, then 
\[
n \leq \Big(2+\frac{1}{s}\Big)d + q + \Big\lfloor \frac{j}{2} \Big \rfloor\Big(1 - \frac{2}{s} \Big).
\]
\end{corollary}
\begin{proof}
For $e\in \Zz$, we will denote by $\pi_e \colon \Laz \to \Laz^e$ the projection to the homogeneous component of degree $e$. For $i\in \Nn$, let us write $y_i = \lc (\X_{2i+1})^\mud \rc \in \Laz$. Let $x=y_{i_1} \cdots y_{i_p}z$, where $i_1,\dots,i_p \in \{1,\dots,s-1\}$ and  $z\in \Laz$ are such that $\dim(x) \leq d$. By \dref{cor:gen_Laz}{cor:gen_Laz:0} we have $\dim(y_i) = i$ for all $i \in \Nn\smallsetminus\{0\}$, so that $\dim(z) \leq d -i_1-\cdots-i_p$ (as $\Laz$ is an integral domain). Moreover $\pi_e(y_i) =0$ if $e \not \in \{-i,2-i\}$ by \dref{prop:Xn}{prop:Xn:dim_fixed}. It follows that $\pi_{j-d}(x)$ is a sum of products $\pi_{e_1}(y_{i_1})\cdots\pi_{e_p}(y_{i_p})\pi_e(z)$ where $e_m \neq -i_m$ for at most $\lfloor j/2 \rfloor$ values of $m$, in which case $e_m =2-i_m$ (and $i_m\geq 2$). In particular $\pi_{j-d}(x) \in I(s)^{p-r}$, where $r=\min(c,\lfloor j/2 \rfloor)$ and $c$ is the number of indices $m \in \{1,\dots,p\}$ such that $i_m=2$. Thus the assumption of the corollary implies that $\lc X^{\mud} \rc$ does not belong to the subgroup generated by above elements $x$ satisfying $p-r\geq q+1$.

Now by \rref{th:fixed}, there exists a polynomial $P \in \Zz[\bx]$, each of whose nonzero monomial has degree at least $n$, such that $\fdeg P \leq d$ and $\lc X^{\mud} \rc= P(\lc (\X_1)^{\mud} \rc,\ldots)$. Then at least one nonzero monomial of $P$ must satisfy the conditions of \rref{prop:deg_fdeg} for some $p \leq q+r$, where $r\leq \lfloor j/2\rfloor$ and $i_m=2$ for at least $r$ values of $m$. We obtain
\[
n \leq (2+s^{-1})d+r(s-2)s^{-1}+p-r \leq (2+s^{-1})d+\lfloor j/2 \rfloor(s-2)s^{-1} +q.\qedhere
\]
\end{proof}

\begin{remark}
\begin{enumerate}[(i)]
\item Taking $s=1$ in \rref{cor:fixed_component}, we obtain that if $\lc F_{d-j} \rc \not \in 2^{q+1} \Laz$, then $n \leq 3d+q - \lfloor j/2 \rfloor$. This refines \rref{prop:small_dim_fix}.

\item Taking $s=d+1$ and $q=1$ in \rref{cor:fixed_component}, we obtain that if $\lc F_{d-j} \rc$ is indecomposable in $\Laz/2$ (see \rref{p:decomposable}), then $n \leq 2d+1+j/2$.
\end{enumerate}
\end{remark}

\begin{corollary}
\label{cor:c_alpha_component}
Let $q \in \Nn$ and $j \in \{0,\dots,d\}$. If $c_\alpha(F_{d-j}) \not \in 2^{q+1}c_\alpha(\Laz)$, then
\[
n \leq 2d + q + \length(\alpha) +j/2.
\]
\end{corollary}
\begin{proof}
It follows from \rref{eq:c_alpha_prod} that $c_\alpha(z_1\cdots z_p) =0$ when $p > \length(\alpha)$ and each $z_i \in \Laz$ is homogeneous of nonzero degree. Thus, in the notation of \rref{cor:fixed_component}, for any $s$ we have $c_\alpha(I(s)^{q +1 + \length(\alpha)}) \in 2^{q+1}c_\alpha(\Laz)$. Therefore the assumption implies that  $\lc F_{d-j} \rc \not \in I(s)^{q +1 + \length(\alpha)}$, and the statement follows by taking $s>d$ in \rref{cor:fixed_component}.
\end{proof}

\begin{remark}
Corollary \rref{cor:c_alpha_component} implies that if $c_\alpha(X^{\mud})$ is not divisible by $2^{q+1}$, then $n \leq 5d/2 + q + \length(\alpha) - |\alpha|/2$. This improves the bound of \rref{cor:n-5d} for partitions $\alpha$ whose average index is greater than $2$.
\end{remark}

\begin{proposition}
\label{prop:top_indec}
Assume that $n=2d+1$ with $d>0$. Then $\lc X \rc$ is indecomposable in $\Laz/2$ (see \rref{p:decomposable}) if and only if $\lc F_d \rc$ is so.
\end{proposition}
\begin{proof}
Let us apply \rref{th:poly}. We will use the morphism $\varphi \colon \Mcc[v^{-1}] \to \Laz$ of \rref{p:Mcc_base}. Using the fact that $\varphi(x_1)=2$, we see that the component of degree $-\fdeg(X_\alpha)$ of $\varphi(x_\alpha)$ is decomposable in $\Laz/2$ when $\length(\alpha) >1$. Since both $\lc \X_n \rc$ and the component of degree $-d$ of $\varphi(x_n)$ are indecomposable in $\Laz/2$ by \rref{cor:gen_Laz}, we deduce that $\lc X \rc =  A_0(\lc \X_1 \rc,\ldots)$ is indecomposable in $\Laz/2$ if and only if the component of degree $-d$ of $A_0(\varphi(x_1),\ldots)$ is indecomposable in $\Laz/2$ (these conditions are equivalent to the fact that $A_0$ contains as a monomial an odd multiple of $X_n$). Now for $i>0$, assume that $A_i$ contains a nonzero multiple of $X_\alpha$ as a monomial. Since $\fdeg(X_\alpha) \leq d$ by \dref{th:poly}{th:poly:Phi}, we must have $\length(\alpha)>1$ (because $\deg(X_\alpha) > 2d+1$), and thus the component degree $-d$ of $\varphi(x_\alpha)$ is decomposable in $\Laz/2$ by the observation above. Since $\lc X^\mud \rc = A_0(\varphi(x_1),\ldots) + \cdots + A_m(\varphi(x_1),\ldots)$ by \rref{p:fix}, the statement follows.
\end{proof}

\subsection{The Euler number}

\begin{para}
The \emph{Euler number} of a smooth projective $k$-scheme $S$ of pure dimension $m$ is the integer $\chi(S) = \deg c_m(\Tan_S)$. This integer depends only on $\lc S \rc \in \Laz$. Indeed, the ring morphism $\Laz \subset \Zz[\bb] \to \Zz$ given by $b_i \mapsto (-1)^i$ sends $\lc S \rc$ to $\chi(S)$ (see \cite[(6.1.6)]{inv}), and will be denoted again by $\chi \colon \Laz \to \Zz$.
\end{para}

\begin{remark}
It is well-known that the integer $\chi(S)$ coincides with the alternate sum of the $\ell$-adic Betti numbers of $S$, for $\ell \neq \carac k$.
\end{remark}

\begin{para}
\label{p:Euler_odd}
The integer $\chi(S)$ is even when $S$ has pure odd dimension, see e.g.\ \cite[(6.1.6), (6.1.7)]{inv}. This also follows from \dref{cor:gen_Laz}{cor:gen_Laz:1} and \rref{lemm:chi_X} below.
\end{para}

\begin{proposition}[{\cite[(6.2.5)]{inv}}]
\label{prop:Euler_odd}
If $\chi(X)$ is odd, then $n \leq 2 d$. If $n$ is odd and $\chi(X)$ is not divisible by $4$, then $n\leq 2 d +1$.
\end{proposition}
\begin{proof}
In view of \rref{p:Euler_odd} we may take $s>d$ and $q=0$ (resp.\ $q=1$) in \rref{th:ideal_power}. 
\end{proof}

\begin{lemma}
\label{lemm:chi_X}
For any $j$, we have $\chi(\lc \X_j \rc) = \chi(\lc (\X_j)^{\mud} \rc)$. Moreover this integer coincides with $j+1$ modulo $2$.
\end{lemma}
\begin{proof}
When $E$ is a rank $r$ vector bundle over a smooth projective $k$-scheme $S$, then $\chi(\Pp(E)) = r \chi(S)$ (see \cite[(6.1.9)]{inv}). This implies that $\chi(\Pp^i) = i+1$ and $\chi(H_{i,j}) = (i+1)j$. The lemma then follows from the description of the fixed loci given in \rref{p:pij} and \rref{p:fixed_Hij}, and an easy computation.
\end{proof}

The cases $q=0,1$ of the next statement were obtained in \cite[(6.2.5)]{inv}.
\begin{proposition}
\label{prop:Euler_fix}
Let $q\in \Nn$, and assume that $\chi(X^{\mud})$ is not divisible by $2^{q+1}$. Then $n \leq 2d + q$.
\end{proposition}
\begin{proof}
It follows from \rref{lemm:chi_X} that $\chi((\X_n)^{\mud})$ is even when $n$ is odd, and the statement follows from \rref{cor:fixed} applied with $s>d$.
\end{proof}

\begin{remark}
\label{rem:Euler_tau}
Let $\tau \colon \Laz \to \Zz$ be a ring morphism, and write $\tau(S) = \tau(\lc S \rc)$ for every smooth projective $k$-scheme $S$. If $\tau$ coincides with $\chi$ modulo $2$, then $\chi$ may be replaced by $\tau$ in \rref{prop:Euler_odd} and \rref{prop:Euler_fix}. An example of such $\tau$ arises from the signature, see \cite[Example~5.10]{ELW}.
\end{remark}

\subsection{The genus \texorpdfstring{$\psi$}{psi}}
\begin{para}
Consider the ring morphism  $\ninv \colon \Laz \to \Zz[\bb] \to \Zz$ given by $b_i \mapsto 0$ if $i$ is odd, and $b_i \mapsto (-1)^{i/2}$ if $i$ is even. When $S$ is a smooth projective $k$-scheme we will write $\ninv(S) = \ninv(\lc S \rc) \in \Zz$. If $S$ has pure dimension $m$, then
\begin{equation}
\label{eq:psi_odd}
\ninv(S) =
\begin{cases}
\deg c_{(2,\ldots,2)}(\Tan_S)  & \text{ if $m$ is even},\\
0 & \text{ if $m$ is odd,}
\end{cases}
\end{equation}
where $(2,\ldots,2)$ denotes the partition of length $m/2$ where each entry is $2$.
\end{para}

\begin{remark}
Observe that $\psi$ does not coincide with $\chi$ modulo $2$. For instance $\psi(\Pp^4)=10$ and $\chi(\Pp^4) = 5$.
\end{remark}

\begin{proposition}
\label{prop:psi_odd}
If $\psi(X)$ is odd, then $n \leq 2d$. 
\end{proposition}
\begin{proof}
We may take $s >d$ and $q=0$ in \rref{th:ideal_power}. 
\end{proof}

\begin{lemma}
\label{lemm:psi_X}
The integer $\psi((\X_{2d+1})^{\mud})$ is even if $d \in \{0,1,2,3\}$.
\end{lemma}
\begin{proof}
This is clear if $d=0$, since $(\X_1)^{\mud} = \Speck \sqcup \Speck$. If $d=1,3$, then $(\X_{2d+1})^{\mud}$ has vanishing even-dimensional components by \dref{lemm:H}{lemm:H:2}, so that the statement follows from \eqref{eq:psi_odd}. Finally, we have $\X_5 = H(1,2)$ and by \rref{p:fixed_Hij}
\[
\psi((\X_5)^{\mud}) = \psi(\Pp^2) + \psi(\Pp^1 \times \Pp^1) + \psi(H_{1,2}) + \psi(\Speck).
\]
This integer is even by \rref{eq:psi_odd}, \rref{lemm:add_proj} and \rref{lemm:add_Milnor}.
\end{proof}

\begin{proposition}
\label{prop:psi_fix}
Assume that $\psi(X^{\mud})$ is not divisible by $2^{q+1}$. Then $n \leq 9d/4 + q$.
\end{proposition}
\begin{proof}
Apply \rref{cor:fixed} with $s=4$, and note that  $\psi(L(4)^{q+1})\subset 2^{q+1}\Zz$ by \rref{lemm:psi_X}.
\end{proof}

\begin{remark}
As in \rref{rem:Euler_tau}, we may replace $\psi$ in \rref{prop:psi_odd} and \rref{prop:psi_fix} by any ring morphism $\tau \colon \Laz \to \Zz$ coinciding with $\psi$ modulo $2$.
\end{remark}

\subsection{Additional observations}
\begin{para}
If $\lc X \rc \in 2\Laz$, then the polynomial $A_0$ of \rref{th:poly} must be divisible by two. This yields improvements of \rref{th:fixed}, \rref{cor:fixed}, \rref{prop:small_dim_fix}, \rref{cor:n-5d}, \rref{rem:isolated}, \rref{cor:fixed_component}, \rref{cor:c_alpha_component}, \rref{prop:Euler_fix}, \rref{prop:psi_fix}, where $n$ may be replaced by $n+1$ under this additional assumption. (In the proof of \rref{th:fixed}, since $2=tx_1$, after modifying $A_1$ we may assume that $A_0=0$.)
\end{para}

\begin{para}
\label{p:effective}
When $n=2d+1$ is odd and $d\geq 3$, the element $\X_n \in \V_{\mud}$ might not be the class of a smooth projective $k$-scheme with a $\mud$-action (it is a difference of such classes). It is however possible to approximate $\X_n$ by such classes. Indeed let $m\in \Nn$. For each $n\in \Nn\smallsetminus\{0\}$, in the notation of \rref{def:X_n}, choose $e'_i \in \Nn$ such that $e_i'=e_i \mod 2^m$ and consider the smooth projective $k$-scheme with a $\mud$-action $\X_n' = \bigsqcup_i H(i,d+1-i)^{\sqcup e_i'}$. Then $\lc \X_n'\rc_{\mud} = \lc \X_n \rc_{\mud} \mod 2^m$, and $\dim (\X_n')^{\mud} = d$.
\end{para}

\begin{para}
The sharpness of the statements provided above can be illustrated as follows. Let us use the notation of \rref{p:effective}.
\begin{enumerate}[(i)]

\item Take $m=1$ and $X = (\X_1')^q\times(\X_{2s+1}')^r$. Then $d=sr$, and $n=q+(2s+1)r$ takes the maximum value allowed by \rref{th:ideal_power}. Taking $m=q+1$, the same example shows the sharpness of \rref{cor:fixed}.

\item The example $X=(\X_1)^q\times(\X_3)^d$, where $n=3d+q$, shows the sharpness of \rref{prop:small_dim_fix}.

\item Let $X=(\X_1)^q\times(\X_5)^r$. Then $d=2r$ and $n=q+5r$. The number of isolated fixed points of $X$ is $2^q$ (all of them are rational over $k$), proving that \rref{rem:isolated} (and thus also \rref{cor:n-5d}) is sharp.

\item Let $X = (\X_2)^d$. Then $\chi(X)$ and $\psi(X)$ are odd, and $n=2d$. Letting $X = \X_1\times (\X_2)^d$ instead, the integer $\chi(X)$ is not divisible by $4$, and $n=2d+1$. This shows the sharpness of \rref{prop:Euler_odd} and \rref{prop:psi_odd}.

\item Let $X = (\X_1)^q\times \X_{2d}$. Then $\chi(X^{\mud}) =2^q \mod 2^{q+1}$ and $n=2d+q$, showing the sharpness of \rref{prop:Euler_fix}.

\item One may see that the integer $\psi((\X_9)^{\mud})$ is odd (for every choice of $e_1,e_2$ in \rref{def:X_n}). Taking $m=q+1$ and $X = (\X_1)^q \times (\X_9')^r$, we have $n=q+9r$ and $d=4r$, showing the sharpness of \rref{prop:psi_fix}.

\item Let $Y$ be a smooth projective $k$-scheme of dimension $d>0$ such that $\lc Y\rc$ is indecomposable in $\Laz/2$. Considering the scheme $X = Y \times Y$ with the $\mud$-action exchanging the factors, we see that \rref{prop:top_indec} does not hold when $n=2d>0$.
\end{enumerate}
\end{para}

\subsection{Curves as fixed loci}
\label{sect:fixed_curves}

In this section, we describe the cobordism classes of involutions whose fixed locus has dimension one, and conversely we determine which cobordism classes of vector bundles over one-dimensional varieties arise as normal bundles to fixed loci of involutions.\\

Let $\Vi_n(d)$ be the subgroup of $\invh$ generated by classes $\lc X \rc_{\mud}$, where $X$ is a smooth projective $k$-scheme of pure dimension $n$ with a $\mud$-action such that $\dim(X^{\mud}) \leq d$. We have seen that $\Vi_n(0)$ is the free abelian group generated by $x_1^n$ (see \rref{prop:trivial_normal}, or \rref{th:small_dim}). We now describe $\Vi_n(1)$:

\begin{proposition}
\label{prop:basis_1}
The abelian group $\Vi_n(1)$ is free, with a basis given by
\begin{center}
\begin{tabular}{ |c|c| } 
 \hline
 $n$ & Basis  \\ 
 \hhline{|=|=|}
 $0$ & $1$  \\ 
 \hline
 $1$ & $\lc \Pp^1 \rc, x_1$ \\
 \hline
 $2$ & $\lc \Pp^1 \times \Pp^1 \rc_{\mud},x_2,x_1^2$\\
 \hline
 $\geq 3$ & $x_1^{n-3}x_3,x_1^{n-2}x_2,x_1^n$\\
 \hline
\end{tabular}
\end{center}
(The $\mud$-action on $\Pp^1 \times \Pp^1$ is induced by the exchange of factors.)
\end{proposition}
\begin{proof}
$n=0$: This is clear.

$n=1$: Linear independence is obtained by applying the morphism $\invh \to \Laz$ of \rref{p:fix}, and composing with $c_{\varnothing},c_{(1)} \colon \Laz \to \Zz$. Any smooth projective $k$-scheme of pure dimension $n$ with a $\mud$-action decomposes as $T \sqcup N$, where $\mud$ acts trivially on $T$ and $\dim(N^{\mud}) <n$. It follows that $\Vi_1(1) = \Laz^{-1} + \Vi_1(0)$, an abelian group generated by $\lc\Pp^1\rc$ and $x_1$.

$n=2$: By \rref{th:poly}, any $x \in \Vi_2(1)$ can be written in $\Rc$ as $x = a x_1^2 +bx_2 +ctx_3$ with $a,b,c \in \Zz$. Since $\lc \Pp^1 \times \Pp^1 \rc_{\mud} = -x_1^2+4x_2-tx_3$ (see \rref{ex:P1xP1}), it follows that $x_1^2,x_2,\lc \Pp^1 \times \Pp^1 \rc_{\mud}$ generate $\Vi_2(1)$. This family is linearly independent, because it becomes so after multiplication with $x_1$, by \rref{lemm:alg_indep} (recall that $x_1t=2$).

$n\geq 3$: This follows from \rref{th:small_dim} and \rref{lemm:alg_indep}.
\end{proof}

\begin{para}
Let now $E$ be a vector bundle over a smooth projective $k$-scheme $S$, and assume that the scheme $E$ has pure dimension $n$. If $\dim S =0$, then $\lc E \to S \rc \in \nu(\inv) \subset \Mcc$ if and only if $\#S(\overline{k})$ is divisible by $2^n$ (see \rref{prop:trivial_normal} or \rref{th:small_dim}), where $\overline{k}$ is an algebraic closure of $k$. Assume that $\dim S=1$. Then $S = S_0 \sqcup S_1$, where $S_i$ has pure dimension $i$. The class $\lc E \to S\rc$ is determined by the degree of $E|_{S_1}$, the arithmetic genus of $S_1$, and the cardinality of $S_0(\overline{k})$. More precisely, we have $\lc E \to S \rc = a\ba_1v^{n-1} + b \lc \Pp^1 \rc v^{n-1} + c v^n$ in $\Mcc$, where
\[
a=\deg c_1(E|_{S_1})\quad ; \quad b = \frac{\deg c_1(\Tan_{S_1})}{2} \quad ; \quad c = \deg [S_0].
\]
It follows from \rref{prop:basis_1} (in view of \rref{ex:x_123} and \rref{ex:P1xP1}) that $\lc E \to S \rc \in \nu(\inv)$ if and only if
\begin{center}
\begin{tabular}{ |c|c| } 
 \hline
 $n$ & Condition \\
 \hhline{|=|=|}
 $0$ &  \text{(no condition)} \\ 
 \hline
 $1$ & $c \in 2\Zz$ \\
 \hline
 $2$ & $a+2b+c \in 4\Zz$\\
 \hline
 $\geq 3$ & $a \in 2^{n-2}\Zz \; \text{ and } \; b \in 2^{n-3}\Zz \; \text{ and } \; 3a - 2b -c \in 2^n\Zz$\\
 \hline
\end{tabular}
\end{center}
\end{para}

\def\cprime{$'$}

\end{document}